\documentclass[oneside, 11pt]{amsart}

\usepackage{amsmath, amssymb, amsthm, hyperref, graphicx,   verbatim, enumerate, xcolor, tikzsymbols, stmaryrd, epsfig, xypic}
\usepackage{array, colortbl}
\usepackage{mathrsfs}  
\usepackage{mathabx}
\usepackage{times}
 
 \DeclareFontFamily{U}{mathc}{}
\DeclareFontShape{U}{mathc}{m}{it}%
{<->s*[1.03] mathc10}{}

\DeclareMathAlphabet{\mathscr}{U}{mathc}{m}{it}

\usepackage[paper=a4paper, marginpar=2.4cm]{geometry}
\usepackage[all]{xy}
\usepackage[utf8]{inputenc}

\usepackage[english]{babel}
 
\newcommand{\e}{\varepsilon}

\newcommand{\set}[1]{\left\{#1\right\}}
\newcommand{\norm}[1]{{\left\Vert#1\right\Vert}}
\newcommand{\abs}[1]{\left\vert#1\right\vert}

\newcommand{\rest}[1]{ \arrowvert_{#1}}

\newcommand{\unsur}[1]{\frac{1}{#1}}

\newcommand{\lrpar}[1]{\left(#1\right)}
\newcommand{\bra}[1]{\left\langle #1\right\rangle}

\newcommand{\inv}{^{-1}}
\renewcommand{\d}{\mathrm{d}}

\DeclareMathOperator{\supp}{Supp}

\DeclareMathOperator{\leb}{Leb}

\DeclareMathOperator{\length}{length}

\DeclareMathOperator{\id}{id}
\DeclareMathOperator{\jac}{Jac}

\DeclareMathOperator{\dist}{dist}

\DeclareMathOperator{\Lim}{Lim}

\DeclareMathOperator{\Ker}{Ker}

\DeclareMathOperator{\Pent}{Pent} 
\DeclareMathOperator{\Per}{Per} 
\DeclareMathOperator{\Exc}{Exc} 
\DeclareMathOperator{\Ax}{Geo}

\newcommand{\C}{\mathbf{C}}
\newcommand{\R}{\mathbf{R}}
\newcommand{\Q}{\mathbf{Q}}
\newcommand{\Z}{\mathbf{Z}}
\newcommand{\N}{\mathbf{N}}

\newcommand{\bfe}{{\mathbf{e}}}

\newcommand{\Hyp}{\mathbb{H}}
\renewcommand{\P}{\mathbb{P}}

\newcommand{\NS}{{\mathrm{NS}}}
\newcommand{\Pic}{{\mathrm{Pic}}}

\newcommand{\vol}{{\sf{vol}}}

\def\Jac{\mathrm{Jac}}

 \def\Diff{{\mathsf{Diff}}}

\newcommand{\Aut}{\mathsf{Aut}}
\newcommand{\Isom}{\mathsf{Isom}}

 \newcommand{\Bir}{{\mathsf{Bir}}}

 \newcommand{\PGL}{{\sf{PGL}}}

\renewcommand{\O}{{\sf{O}}}
\newcommand{\SO}{{\sf{SO}}}

\newcommand{\GL}{{\sf{GL}}}
\newcommand{\SL}{{\sf{SL}}}

\newcommand{\SU}{{\sf{SU}}}

\newcommand{\m}{\mathscr{m}}

\theoremstyle{plain}

\newtheorem{thm}{Theorem}[section]
\newtheorem{cor}[thm]{Corollary}
\newtheorem{pro}[thm]{Proposition}
\newtheorem{lem}[thm]{Lemma}

\newtheorem{que}[thm]{Question} 
\newtheorem{prob}[thm]{Problem}

\theoremstyle{definition}
\newtheorem{defi}[thm]{Definition}

\newtheorem{eg}[thm]{Example}
\newtheorem{rem}[thm]{Remark}

\numberwithin{equation}{section}       

\addtocounter{section}{0}             
\numberwithin{equation}{section}       

\begin{document}

\setlength{\parskip}{.2em}
\setlength{\baselineskip}{1.26em}   

\begin{abstract}
We first present an overview   of our previous work on the dynamics of subgroups of 
automorphism groups of compact complex surfaces, 
together with a selection of open problems and new classification results. 
Then, we study two families of examples in depth: the first one comes from 
folding plane pentagons,  and the second one is 
a family of groups introduced by Jérémy Blanc, which exhibits interesting new dynamical features. 
\end{abstract}

\title[Automorphism groups of projective surfaces]{Dynamics of automorphism groups of projective surfaces: classification, examples and outlook}
\date{\today}

\author{Serge Cantat}
\address{Serge Cantat, IRMAR, Campus de Beaulieu,
b\^atiments 22-23
263 avenue du G\'en\'eral Leclerc, CS 74205
35042  RENNES C\'edex, France}
\email{serge.cantat@univ-rennes1.fr}
 \author{Romain Dujardin}
\address{Romain Dujardin,  Sorbonne Universit\'e, CNRS, Laboratoire de Probabilit\'es, Statistique  et Mod\'elisation  (LPSM), F-75005 Paris, France}
\email{romain.dujardin@sorbonne-universite.fr}

\thanks{The research of the first author is partially funded 
by the European Research Council (ERC GOAT 101053021); he also benefited from the stimulating atmosphere of the Center Henri Lebesgue, a French government “Investissements d’Avenir”, bearing the following reference ANR-11-LABX-0020-01. }
\maketitle

\setcounter{tocdepth}{1}
\tableofcontents

 \newpage
 
\section*{Introduction}

This article is the  sixth  of a series dedicated to  the dynamics of groups of automorphisms of 
compact complex surfaces~\cite{cantat_groupes, stiffness, finite_orbits, invariant, hyperbolic}. 
Our purpose is to review our previous work and to enrich it with new examples, applications, and 
open problems. Let us briefly summarize its contents and  new features. 

In the first two sections, we offer a detailed presentation  of our results, 
illustrated with a large number of open questions. 

A standing assumption in our previous papers is 
that the surfaces into consideration are projective. 
It is natural to question this assumption and to relate the classification of surfaces to the dynamics of their groups of automorphisms.
This is dealt with in the first part of the paper: in \S~\ref{par:projectivity_of_the_surface}, 
we show that only projective surfaces can   carry ``non-elementary'' automorphism groups and we describe a few examples on Hopf surfaces in \S~\ref{subs:remarks_non_kahler}.  

 Then, we explain how our  theory applies  to the following geometric examples: \

(1) The first one comes from classical Euclidean geometry, and 
 is given by folding plane pentagons of given side length 
along their diagonals. Surprisingly enough, this gives rise to a group action on a K3 surface, 
which is reminiscent  of the Wehler family of examples, which has been a 
thread in our  work. 
These examples are studied in Part~\ref{part:pentagons}.   
Section~\ref{sec:pentagons_geometry}
describes the underlying algebraic geometry of the problem, 
which goes back to the work of Darboux~\cite{darboux} on quadrilaterals 
and elliptic functions. The ergodic theory of random pentagon foldings is analyzed in Section~\ref{sec:pentagons_dynamics},   in the spirit of the work of Benoist and Hulin~\cite{benoist-hulin, benoist-hulin2}.\

(2)  In Part~\ref{part:ergodic_blanc} (Sections~\ref{sec:invariant_curves_parabolic} to~\ref{sec:stiffness_blanc}),
 we focus on groups of automorphisms introduced by Blanc in~\cite{Blanc:Michigan}. Each of these examples is determined
by the choice of a plane cubic curve $C\subset \P^2_\C$, an integer $m\geq 3$, and $m$ points $q_1$, $\ldots$, $q_m$ on $C$. 
To each $q_j$, one associates a Jonquières involution, that fixes $C$ pointwise and preserves the pencil of lines through $q_j$. The group 
generated by these $m$ involutions lifts to a group of automorphisms on a rational surface $X$ which is obtained by blowing-up $5m$ points of $C$. 
A key property of  these automorphism groups is that they
 preserve a singular volume form with poles along the strict transform $C_X$ of $C$ 
 and whose total mass is infinite.
Using our previous results, we prove that, {\emph{for an appropriate  choice of the points $q_i$, 
the only ergodic stationary measures on $X(\R)$ 
are fixed points and are contained in $C_X(\R)$}}. In particular, random orbits   almost surely converge, on average, to the curve $C_X(\R)$.  
Thus, as we shall explain,  this  behavior differs strongly from that of 
automorphisms of non-rational surfaces.

\section{Overview and open problems}
 
In this section, we provide a detailed  overview of our former results, 
with short introductions to their proofs, as well as  a description of our main new results 
and examples. 
 
\subsection{From orbit closures to stationary measures}\label{subs:generalities}
Let $X$ be a compact space and $\Gamma$ be an infinite  
group of homeomorphisms of $X$.
We make the standing assumption that $\Gamma$ is countable; this will actually not be a restriction in the cases of interest to this paper. 
Our general  aim   is to study the dynamics of such an action. 
In particular we  wish to address the  following usual problems:
\begin{enumerate} 
\item[(Pb1)] describe the orbit closures $\overline {\Gamma \cdot x}$, for $x\in X$;
\item[(Pb2)] study the finite  orbits of $\Gamma$
\begin{equation}
\mathrm{Per}(\Gamma) = \set{x; \ \Gamma\cdot x \text{ is finite}};
\end{equation}
by definition, points of $\mathrm{Per}(\Gamma)$ will be called {\bf{$\Gamma$-periodic points}}. 
\item[(Pb3)] classify $\Gamma$-invariant measures, that is, probability measures $\mu$ on $X$ such that $f_*\mu=\mu$ for all $f\in \Gamma$;
\item[(Pb4)] describe the asymptotic distribution of $\Gamma$-orbits.
\end{enumerate}
Observe that the last question is not properly formulated until  
a specific way of going to infinity in $\Gamma$ (in other words, a notion of ``time'') has been described. 
A common choice for this  is to fix a probability measure $\nu$ on $\Gamma$ such that $\bra{\supp(\nu)} = \Gamma$ 
and to explore $\Gamma$ by walking at random according to $\nu$. 
Then the notion of asymptotic distribution of the orbit of $x\in X$ may  either refer to the asymptotics 
of the orbital 
averages
\begin{equation}\label{eq:orbital}
\int \delta_{f_{n-1}\cdots f_0(x)} \d\nu(f_0)\cdots \d\nu(f_{n-1}),
\end{equation}
 or to the time averages along random trajectories
 \begin{equation}\label{eq:empiric}
 \unsur{n} \sum_{k=1}^n   \delta_{f_{n-1}\cdots f_0(x)}
 \end{equation}
 for  $\nu^{\N}$-almost every $(f_n)_{n\geq 0}$. 
In both cases, understanding the limit points essentially boils down to the following problem:
\begin{enumerate}
\item[(Pb5)] classify $\nu$-stationary probability measures.
\end{enumerate}
Of course, (Pb5) subsumes (Pb3). 
To understand the meaning and relevance of (Pb5), let us first recall that a probability measure $\mu$ on $X$ is  $\nu$-\textbf{stationary} if 
 \begin{equation}
\mu = \nu\ast\mu:= \int f_\varstar \mu \, \d\nu(f).
 \end{equation}
 Since the measure in Equation~\eqref{eq:orbital} is the $n$-th convolution $\nu^{\ast n}  \ast\delta_x$, any limit point of this sequence of measures is $\nu$-stationary. 
 Breiman's ergodic theorem shows that the same is true 
for the random empirical measures  in Equation~\eqref{eq:empiric} (see \cite[\S 3.2]{benoist-quint_book}). In particular, if $K$ is a compact $\Gamma$-invariant subset of $X$, 
then starting from $x\in K$ one constructs $\nu$-stationary measures with support in $K$; 
these may also be obtained by applying a 
  fixed point theorem to the operator $\mu\mapsto \nu\ast \mu$ acting on probability measures on $K$. 
  As a consequence, 
taking $K=\overline{\Gamma \cdot x}$, we see  that  a solution to 
 (Pb5) is also useful for (Pb1). This  approach, using ergodic theoretic methods 
 to study orbit closures (that is,   use (Pb5) and (Pb3) to study  (Pb1)), is now commonplace in this area of research. 

So, the set of stationary measures on $X$ is a non-empty, compact and convex subset 
of the set of probability measures on $X$. It contains $\Gamma$-invariant probability measures, but in many situations invariant measures fail to exist. 
Thus,  stationary measures can be viewed as  the 
correct analogues of invariant measures when studying (large) groups of transformations instead of cyclic groups (i.e.\ the iterations of a single homeomorphism). 
In  this respect, it may seem  hopeless at first sight to classify all stationary measures, but, in fact, they often  satisfy  rigidity properties  
which make such a classification feasible. Results in homogeneous dynamics, in particular the work of Ratner~\cite{Ratner:Duke, Ratner:Annals, Ghys:Bourbaki} and Benoist and Quint \cite{benoist-quint1}, illustrate perfectly this line of thought. 
Similar phenomenon also appear in non-homogeneous dynamics, notably in the work of 
Eskin and Mirzhakani~\cite{eskin-mirzakhani} and Brown and Rodriguez Hertz~\cite{br}. 
Our work is strongly influenced by these former results and the methods developed to reach them.

\begin{rem}
For $\nu$-stationary (resp. $\Gamma$-invariant)
measures, there is a notion of ergodicity: such a measure is ergodic if it is an 
extremal point of the compact convex set of $\nu$-stationary (resp.\ $\Gamma$-invariant) measures. 
It is standard, albeit non-trivial,  that these two notions coincide (see~\cite[Prop. 2.8]{benoist-quint_book}), that is 
a $\nu$-stationary measure $\mu$ is ergodic if and only if every
 almost $\Gamma$-invariant subset $A$  (i.e.\ $\mu(f(A)\Delta A) = 0$ for every $f\in \Gamma$)
has measure $0$ or $1$. 
Hence, in the classification problems   (Pb5) and (Pb3), it is enough to restrict to ergodic measures. 
\end{rem}

\subsection{Non-elementary groups of automorphisms and the classification of surfaces} \label{subs:NE}

In this paper, $X$ will be a compact complex surface. We denote by $\Aut(X)$ 
its group of holomorphic diffeomorphisms, which we call {\textbf{automorphisms}}, even when 
$X$ is not algebraic. The group $\Gamma$ will be contained in  $\Aut(X)$. 

It would also be natural to consider actions by automorphisms on quasi-projective or affine surfaces, for instance by polynomial automorphisms of $\C^2$, or even by  birational transformations. 
Extending our results to this more general setting or to higher dimensions 
is an important challenge, which would undoubtedly lead to serious difficulties. 
References 
include~\cite{Cantat:BHPS, Rebelo-Roeder, gamburd_icm}
(see also Example~\ref{eg:intro_markov}, 
Remark~\ref{rem:wehler_markov_finite_orbits}, Question~\ref{que:intro_henon}, 
and Theorem~\ref{thm:rigidity_henon} below). 

Two related constraints will be imposed on $\Gamma$. Firstly,  $\Gamma$ must be sufficiently large to expect some measure rigidity
 property. For instance, if $\Gamma$ is abelian, its stationary measures are automatically invariant; and if $\Gamma$ is  generated by an automorphism with positive entropy, there are uncountably many invariant ergodic measures, without any hope of classifying them.  Secondly,  the action of $\Gamma$ on the cohomology of $X$ must also be ``sufficiently large''. To explain this,
recall that  $\Aut(X)$ is a complex Lie group and, assuming $X$ to be Kähler, a theorem of Fujiki and Lieberman asserts that the connected component $\Aut(X)^0$ of $\Aut(X)$ is a subgroup of finite index in the kernel of the homomorphism 
\begin{equation}
 \Aut(X) \ni f \mapsto f^*\in \GL(H^{2}(X;\Z))
\end{equation}
that describes the action on the cohomology (see~\cite{fujiki, lieberman}). 
Then, using the invariance of the intersection form on $H^2(X;\Z)$, the Hodge index theorem, and the Tits alternative, one obtains easily (see~\cite{Cantat:Milnor, stiffness}) the equivalence of the following properties for any subgroup $\Gamma\subset \Aut(X)$:
\begin{enumerate}[(a)]
\item[(a)]  the image $\Gamma^*$ of $\Gamma$ in $ \GL(H^{2}(X;\Z))$ is not virtually Abelian;
\item[(b)]  the image $\Gamma^*$ of $\Gamma$ in $ \GL(H^{2}(X;\Z))$ contains a pair of linear maps $(f^*,g^*)$ generating a non-Abelian free group;
\item[(c)] $\Gamma$ contains a pair of automorphisms $(f,g)$ with positive topological entropy generating a non-Abelian free group.
\end{enumerate}
Moreover, the classification of compact Kähler surfaces implies that if these properties are satisfied for some $\Gamma\subset \Aut(X)$, then either $X$ is a torus or $\Aut(X)^0$ is trivial. From this, one deduces  that, up to finite index, only two distinct regimes  need to be studied: 
\begin{enumerate}
\item the case where $\Gamma$ is contained in $\Aut(X)^0$, 
\item the case where $\Gamma$ satisfies the three equivalent properties (a), (b), (c). 
\end{enumerate}
The first case boils down to homogeneous dynamics, the main example being given by subgroups of $\PGL_3(\C)$ acting by linear projective transformations on $\P^2(\C)$. Thus, we focus on the second case that is, we shall impose that  $\Gamma$ satisfies (a), (b), (c); examples include certain  rational, K3, and Enriques surfaces, as well as certain tori (the only case that can satisfy (a), (b), (c) and $\dim(\Aut(X)^0)\geq 1$). 

Even if our understanding of non-Kähler compact complex surfaces is less satisfactory, we shall see in Section~\ref{sec:projective} that these properties are also relevant in this more general context. So,  we introduce the following

\begin{defi}
Let $X$ be a compact complex surface. A subgroup $\Gamma$ of $\Aut(X)$ is \textbf{non-elementary} if $\Gamma^*$ contains a non-Abelian free group, otherwise, it is {\textbf{elementary}}. 
\end{defi}

A first important result, proven below in Section~\ref{sec:projective}, says that the existence of a non-elementary subgroup 
$\Gamma\subset \Aut(X) $ forces 
$X$ to be projective, hence, by the GAGA principle, it also  forces 
 $\Gamma$ to act by regular algebraic transformations.  
 Consequently, the study of non-elementary actions belongs to 
 algebraic dynamics.  
 
\begin{thm} \label{thm:X-is-projective_intro}
Let $X$ be a compact complex surface such that there exists a non-elementary subgroup 
$\Gamma\subset \Aut(X)$. Then $X$ is projective and it must be  a blow-up of $\P^2$,   a K3 surface, 
 an Enriques surface, or   an Abelian surface 
\end{thm}
 
A consequence  is that, unless $X$ is a torus, $\Aut(X)^0$ is trivial, hence 
$\Aut(X)$ is   discrete, and  the countability assumption on $\Gamma$ is automatically satisfied. 
As suggested by this theorem, examples of non-elementary group actions are scarce and rely on algebro-geometric constructions. Besides the following two classes of examples, 
several new ones will be discussed in this paper. 

\begin{eg} The first examples are found on (certain) compact tori $X=\C^2/\Lambda$. Recall
 that 
every automorphism of such a surface is induced by an affine transformation 
$x\in X\mapsto Ax + B$, where  $A\in \GL_2(\C)$ preserves $\Lambda$ and $B\in \C^2$. 
By  Theorem~\ref{thm:X-is-projective_intro}, if 
 $\Aut(X)$ is non-elementary, then  $X$ must be projective, that is, be an Abelian surface.
 
More generally, a \textbf{Kummer group} $(X,\Gamma)$ is, by definition,   a compact complex surface $X$ endowed with a subgroup $\Gamma$ of $\Aut(X)$,  
such that there exists a torus $X'=\C^2/\Lambda$, a subgroup  $\Gamma'\subset \Aut(X')$, and a generically finite, dominant rational map $\varphi\colon X'\to X$ that 
 semi-conjugates $\Gamma'$ to $\Gamma$ (i.e.\ $\varphi\circ \Gamma'=\Gamma\circ \varphi$). We refer to \cite[\S 4]{finite_orbits} for details and  a classification of such examples.  Kummer examples can be analyzed with tools from homogeneous dynamics, and as we will see, they often stand out for their exceptional, distinctive properties, somewhat similar to those of 
  monomial, Tchebychev, and Lattès mappings in one-dimensional dynamics (see~\cite{Cantat:Panorama-Synthese}). 
 \end{eg}
 
\begin{eg}\label{eg:intro_Wehler} 
Consider the family of all  K3 surfaces in 
$\P^1\times \P^1\times \P^1$ (equivalently, smooth surfaces of tri-degree $(2,2,2)$); they depend on $26$ parameters. 
If $X$ is such a surface, the three natural projections onto $\P^1\times \P^1$ are dominant morphisms of degree $2$, and each of them determines a regular involution of $X$ (which exchanges the two points in the fibers). Thus, $X$ comes with a group $\Gamma\subset \Aut(X)$, generated by this involutions. For a general choice of $X$, $\Gamma$ is non-elementary.
We shall refer to these examples $(X,\Gamma)$ as \textbf{Wehler examples}. These were studied 
 thoroughly in our work (see~\cite{stiffness, finite_orbits,hyperbolic}). 
The family  of pentagon folding groups introduced in Part~\ref{part:pentagons} is in many ways 
reminiscent from  the Wehler family.
\end{eg}

\begin{eg}\label{eg:intro_markov}
Let us mention an example which is unfortunately excluded by our assumptions, 
because the surface is not compact (i.e. when compactifying 
 the surface, the automorphisms become birational transformations). 
 Consider the affine surface $M_0\subset \C^3$ defined by
\begin{equation}
 x^2+y^2+x^2 = 3xyz.
 \end{equation}
It is a cubic surface, endowed  with three natural $2$-to-$1$ ramified covers onto $\C^2$, obtained by
forgetting one of the variables, each of which defining a regular involution on $M_0$. 
The dynamics of the group generated by these involutions was studied by Markov in his 1879 thesis, 
in relation  to the study of binary quadratic forms and Diophantine approximation, and $M_0$ is now called \emph{the Markov surface}. It can be considered as a degenerate   Wehler example. 
We refer to~\cite{Cantat:BHPS, deroin-dujardin, Rebelo-Roeder, gamburd_icm} and~\cite{Rebelo-Roeder2} in this volume for a dynamical viewpoint on these surfaces.
\end{eg}

\subsection{Further remarks: parabolic elements, fields of definition, irrational surfaces} 
\label{par:parabolic_field_kodaira}
From now on we will assume that $\Gamma\subset \Aut(X)$ is non-elementary, and therefore that $X$ 
is projective. Our strongest results feature an additional hypothesis on $\Gamma$, namely the 
existence of parabolic elements. The vocabulary is as follows. An automorphism $f\in \Aut(X)$ is 
{\textbf{elliptic}} if $f^*$ is a finite order element of $\GL(H^2(X;\Z))$. It is {\textbf{loxodromic}} if the 
spectral radius $\lambda_f$ of $f^*$ is $>1$; in that case, $\lambda_f$ is the only eigenvalue of $f^*$ 
with modulus $>1$, it is simple, and $\log(\lambda_f)$ is equal to the topological entropy of $f\colon X\to 
X$. Otherwise, $f$ is {\bf{parabolic}}; this means that some iterate $(f^*)^k$, for some $k\geq 1$, is unipotent and distinct from the identity. In this case, there is a unique genus $1$ fibration $\pi\colon X\to B$ which is $f$-invariant; this means that there is an automorphism $f_B$ of the Riemann surface $B$ such that $\pi\circ f=f_B\circ \pi$; moreover, $f_B$ has finite order, except if $X$ is a torus $\C^2/\Lambda$.  Hence, assuming $X$ is not a torus and changing $f$ into a positive iterate, $f$ preserves every fiber, each smooth fiber is a genus $1$ curve, and $f$ acts upon it as a translation; such 
maps are   referred to as \textbf{Halphen twists}. 
The analysis of these parabolic automorphisms leads to rich structures and is connected to the field of discrete integrable systems (see~\cite[\S 3]{invariant} for an account, and~\cite{duistermaat} for a thorough treatment).

With this vocabulary, $\Gamma$ is non-elementary if and only if it contains a pair of loxodromic elements $(f,g)$ that generates a non-Abelian free group.
The hypothesis that $\Gamma$ contains a parabolic element is
 of a different nature --it is analogous, in Ratner's theory, to the existence of unipotent elements. 
 If $\Gamma$ contains two parabolic elements preserving distinct fibrations, then $\Gamma$ is non-elementary; conversely, if $\Gamma$ contains a parabolic element and is non-elementary, then it contains two parabolic elements with distinct invariant fibrations. We refer to~\cite{Cantat:Milnor, stiffness} for   these results. 

Our results may also depend on the field of definition.  When $X$ is  projective, $X$ and $\Gamma$ are defined by polynomial equations and formulas with coefficients in a subfield of $\C$. Some of our results require 
that $X$ and $\Gamma$ be defined over $\R$, in which case one might restrict the dynamics to the real part $X(\R)\subset X$. Some require that $X$ and $\Gamma$ are defined over $\overline{\Q}$.

Another thing that the reader should keep in mind comes from the classification of surfaces. By Theorem~\ref{thm:X-is-projective_intro}, surfaces for which  $\Aut(X)$ is non-elementary fall into
 two types.

\smallskip

\noindent--  If 
$X$ is rational, 
then $X$ is a blow-up of the plane $\P^2$ at at least 10 points; as we shall see, 
a non-elementary subgroup $\Gamma\subset \Aut(X)$ may or may 
not preserve a continuous volume form.

\smallskip

\noindent--  If 
$X$ is not rational,
then $X$ is a blow-up of a K3,   Enriques, or   Abelian surface $X_0$; moreover, by the uniqueness of the minimal model, the group $\Aut(X)$ is obtained by pulling back of a subgroup of $\Aut(X_0)$. On $X_0$, there is a natural $\Aut(X_0)$-invariant volume form, induced by the triviality of the canonical bundle $K_{X_0}$ or its square $K_{X_0}^{\otimes 2}$. This volume form lifts to $X$ as an $\Aut(X)$-invariant form which vanishes only along the exceptional divisor of the birational morphism $X\to X_0$. 

\smallskip

Some of our dynamical results do  require the existence of 
such an invariant volume form, a hypothesis that might fail only for rational surfaces. 
Note that if $X$ is irrational and defined over $\R$, there is an $\Aut(X_\R)$-invariant 
volume form on $X(\R)$ as well. 
 

\subsection{Stiffness}\label{subs:intro_stiffness}
Let us now come back to our five initial problems.
As explained in the previous sections, we only deal 
with non-elementary group actions on projective surfaces. We refer to~\cite[\S 10.3]{stiffness} 
for elementary actions on compact Kähler surfaces and to~\S~\ref{subs:remarks_non_kahler}
 below for some remarks on the  non-Kähler case. 

We   say that a probability measure $\mu$ on $X$ is \textbf{Zariski diffuse} if it gives no mass to proper Zariski closed subsets of $X$. 
Ergodic stationary or invariant measures which are not Zariski diffuse are easily  analyzed 
(see \cite[Prop. 10.6]{stiffness}), so we focus on   Zariski diffuse   measures. 

In certain situations, (Pb5) is actually equivalent to (Pb3); this happens whenever 
all $\nu$-stationary measures are invariant: this is the \textbf{stiffness}  (or $\nu$-stiffness) property of 
Furstenberg~\cite{furstenberg_stiffness}. An obvious obstruction to stiffness is when there is no $\Gamma$-invariant measure at all. The overall philosophy of~\cite{stiffness} is the converse: 
\emph{non-elementary 
group actions on compact complex surfaces admitting  a 
 natural  Zariski diffuse invariant measure  tend to be stiff} (in fact, this ``principle''   also applies
 to elementary groups). In this respect, recall from  \S~\ref{par:parabolic_field_kodaira}, that  if 
$X$ is not rational,  then, working on the minimal model $X_0$ of $X$, there is a natural invariant volume form. 
 
Let us be more specific. Let  $\Gamma\subset\Aut(X)$ be a non-elementary group of automorphisms of a complex projective surface. Fix a probability measure $\nu$ on $\Gamma$ satisfying  the {\emph{generating condition}} 
 \begin{equation}\tag{S}
\supp(\nu) \; \text{ generates } \;  \Gamma
\end{equation} 
and the  \emph{moment condition} 
 \begin{equation}\tag{M}
\int \lrpar{\log\norm{f}_{C^1(X)}+ \log\norm{f\inv}_{C^1(X)}} \, \d\nu (f) < + \infty.
\end{equation}
Condition (S) is natural and necessary  if we want the random dynamics  to faithfully describe the group action.  Condition (M) is   required to apply   the tools of smooth ergodic theory; note that by the Cauchy estimates, it automatically implies the same finiteness for the $C^k$ norm. 
Pick an ergodic  $\nu$-stationary measure $\mu$. Then $\mu$ admits two Lyapunov exponents $\lambda^- \leq \lambda^+$. These exponents are defined by applying the Oseledets theorem fiberwise 
to the (non-invertible) dynamical system  associated to the random dynamics: 
\begin{equation}
F_+: \ \begin{matrix}
\Omega\times X & \longrightarrow & \Omega \times X \\
(\omega, x) & \longmapsto& (\sigma\omega, f_\omega^1(x))
\end{matrix}
\end{equation}
where $\Omega = \Aut(X)^\N$, $\omega = (f_n)_{n\geq 0}$, $\sigma$ is the shift 
and $f_\omega^n = f_{n-1}\circ \cdots \cdots \circ f_0$ (so that $f_\omega^1 = f_0$). 

If $\lambda^-<0<\lambda^+$, one says that  $\mu$ is {\textbf{hyperbolic}}. 
Then, the  Oseledets theorem provides 
a measurable line field of stable directions $E^s_\omega(x)$ 
defined for $\nu^\N\times\mu$-almost every $(\omega,x)$. 
 We say that $E^s$ is 
\textbf{non-random} if it does not depend on $\omega$; more precisely, if for $\mu$-almost every $x$, there is a line $E^s(x)\subset T_xX$ such that $E^s_\omega(x) = E^s(x)$ for $\nu^\N$-almost every trajectory $\omega$. 
Our first key result is the following (see~\cite[Thms C and 9.1]{stiffness}).

\begin{thm}\label{thm:non_random}
Let $\Gamma\subset \Aut(X)$ be a non-elementary group of automorphisms of a (necessarily projective) compact complex surface. 
Let $\nu$ be a probability measure on $\Gamma$ satisfying Conditions (S) and (M), and let $\mu$ be a Zariski diffuse, ergodic, and hyperbolic $\nu$-stationary measure. 
If the field of Oseledets stable directions is non-random, then  $\mu$ is $\Gamma$-invariant. 
\end{thm}

When $\lambda^+\geq \lambda^- \geq0$,  the \emph{invariance principle} of 
Crauel~\cite{crauel} (this terminology is due to Avila-Viana~\cite{avila-viana}) 
directly implies that $\mu$ is invariant.  
If $\Gamma$ preserves a volume form, we have 
$\lambda^++\lambda^-=0$, so there are two possibilities: either   $\lambda^+=\lambda^-=0$ 
and we can apply Crauel's principle, or $\mu$ is hyperbolic. 
We thus see that for the  stiffness problem we just have to 
consider hyperbolic measures. 

Now, assume in addition that $X$ and $\Gamma$ are defined over $\R$, that $X(\R)$ is non-empty, and that $\Gamma$ preserves an area form on $X(\R)$; more generally, 
we can consider a $\Gamma$-invariant, smooth, totally real surface 
$\Sigma\subset X$ on which $\Gamma$ preserves an area form. 
Then, we can apply the  results of 
Brown and Rodriguez Hertz~\cite{br}, which assert that 
for a volume preserving action on a  $2$-dimensional real surface, the randomness (or non-non-randomness, according to the terminology of~\cite{br}) 
of the stable line field $E^s_\omega(x)$ implies the invariance of $\mu$ (see~\cite[Thm 3.4]{br}). 
 From this and the previous considerations, we get: 

 \begin{thm}\label{thm:stiffness_for_real_surfaces}
 Let $\Gamma\subset \Aut(X)$ be a non-elementary group of automorphisms of a compact complex surface. 
Let $\nu$ be a probability measure on $\Gamma$ satisfying Conditions (S) and (M).
Let $\Sigma\subset X$ be a $\Gamma$-invariant, smooth, totally real surface that supports a $\Gamma$-invariant area form.
Then, any Zariski diffuse $\nu$-stationary measure on $\Sigma$ is $\Gamma$-invariant. 
 \end{thm}

We  expect that similar techniques will provide the same result on the complex surface $X$, without restricting to a real submanifold. 
Indeed, the following result  ought to be true: \emph{under the assumptions of Theorem~\ref{thm:non_random}, 
if $\Gamma$ preserves a smooth volume form on $X$, then every Zariski diffuse $\nu$-stationary 
measure is $\Gamma$-invariant.} 
In particular, this automatic  invariance of Zariski diffuse stationary measures should hold for K3 surfaces. 
At the time of writing these lines, new results have been announced by Brown, Eskin, Filip and  Rodriguez-Hertz~\cite{BEFR} and Roda~\cite{roda}, 
which  should  complete this story.

\noindent{\textbf{Notes on the proof.-- }}The proof of Theorem~\ref{thm:non_random} occupies most of~\cite{stiffness}. 
It relies on the following chain of ideas: if $E^s$ is non-random then either (1) $\mu$ 
satisfies a strong form of zero fiber entropy (see Section~\ref{sec:classification} for this notion) which makes it 
invariant  or (2) the (global) stable manifolds $W^s_\omega(x)$, obtained from Pesin theory, are non-random as well. 
Now, almost every stable manifold $W^s_\omega(x)$ is biholomorphic to $\C$ and, to any non-constant entire curve $\C\to X$, one can associate its
so-called Ahlfors-Nevannlina currents (obtained as limits of integration on large disks in $\C$). 
Thus, to almost every point $(\omega,x)$ is associated a set of closed positive 
  currents. 
In this dynamical situation, we show that there is in fact a unique current $T_\omega^s$ associated to $W^s_\omega(x)$ 
and that $T_\omega^s$ depends only on $\omega$. 
Thus, if the stable manifolds are non-random, the $T_\omega^s$ are non-random as well, hence so are their cohomology classes $[T_\omega^s]\in H^{2}(X, \R)$. 
But these cohomology classes can be analyzed by looking at the action of $\Gamma$ on $H^2(X;\R)$. Using  
Furstenberg's theory of random products of matrices (applied  to   the action of 
$(\Gamma^*, \nu)$ on  $ H^{2}(X, \R)$) and the non-elementary assumption, we prove that these classes $[T_\omega^s]$ should in fact depend  non-trivially on $\omega$. This contradiction rules out case (2), and the theorem follows.  
 
\subsection{A new instance of stiffness} 

Theorem~\ref{thm:stiffness_for_real_surfaces} requires the existence of an invariant area form. In Part~\ref{part:ergodic_blanc}, we study a family of examples on rational surfaces, first defined by Blanc in~\cite{Blanc:Michigan}, that do not preserve any   smooth volume form.  
The following theorem summarizes   the main properties of our examples:

\begin{thm}\label{thm:blanc_examples}
\newcounter{numero}
Let $C$ be a smooth, connected, real cubic curve in $\P^2$. 
Let $k\geq 4$ be an integer. One can find a set of 
$5k$ points of $C(\C)$, consisting of 
 $3k$ real points  and $k$ pairs of complex conjugate points,  
 such that, after blowing up the plane at these $5k$ points, 
one gets  the following objects:
\begin{enumerate}
\item a rational surface $X$, in which the strict transform $C_X$ of $C$ is a curve of genus $1$;
\item a non-elementary group $\Gamma\subset \Aut(X)$ isomorphic to the free product $\bigast_{1}^k\Z/2\Z$;
\item a rational $2$-form $\Omega_X$ that does not vanish, has poles of order $1$ along $C_X$, and is $\Gamma$-invariant; 
\item $X$, $\Gamma$, $C_X$, and $\Omega_X$ are defined over $\R$; the total volume of $\Omega_X\wedge \overline{\Omega_X}$ is infinite and, after restriction to $X(\R)$, the total area of $\Omega_{X}$ is also infinite.\setcounter{numero}{\value{enumi}}
\end{enumerate}
The dynamics on $X(\C)$ satisfies:
\begin{enumerate}\setcounter{enumi}{\value{numero}}
\item every orbit $\Gamma\cdot x\subset X(\C)$ is infinite, except for $x\in C_X(\C)$, in which case $\Gamma\cdot x=\set{x}$;
\item $(\Omega_X\wedge \overline{\Omega_X})$-almost every $\Gamma$-orbit is dense in $X(\C)$.\setcounter{numero}{\value{enumi}}
\end{enumerate}
And for the real dynamics on $X(\R)$ we get:
\begin{enumerate}\setcounter{enumi}{\value{numero}}
\item every infinite orbit in $X(\R)$ is dense in $X(\R)$; 
\item if $\nu$ is a probability measure on $\Gamma$ satisfying Conditions (S) and (M), then every
 ergodic $\nu$-stationary measure on $X(\R)$  is a Dirac measure  $\delta_x$ at some point  
 $x\in C_X(\R)$; such stationary measures have zero Lyapunov exponents; 
\item if $\nu$ is a probability measure as in (8), 
the action of $\Gamma$ on $H^1(X(\R);\R)$ has a positive Lyapunov exponent (in the sense of Furstenberg's theory of random products of matrices).
\end{enumerate}
\end{thm}

The last property contains in particular 
the existence of automorphisms $f\in \Gamma$ with an eigenvalue $>1$ on $H^1(X(\R);\R)$: such 
automorphisms have positive entropy in the real locus $X(\R)$. The proof occupies all of Part~\ref{part:ergodic_blanc} and uses
most of our previous results. In particular, the data ($C$ and the points that we blow up) must be defined over $\overline \Q$.

 It is instructive to compare this family of examples with the following ones coming from homogeneous dynamics. 

\begin{eg} 
Consider the group $\SL_2(\R)$, acting by linear projective transformations on the real projective line $\P^1(\R)$. 
The diagonal action on $\P^1(\R)\times \P^1(\R)$ preserves the form 
\begin{equation}
\frac{dx\wedge dy}{(x-y)^2}.
\end{equation}
This meromorphic section of the canonical bundle $K_{\P^1\times \P^1}$ does not vanish and has a double pole along the diagonal; its total area is 
infinite. Now, consider $\Gamma=\SL_2(\Z)\subset \SL_2(\R)$ acting diagonally on $\P^1(\R)\times \P^1(\R)$. The closure of every $\Gamma$-orbit  is equal to $\P^1(\R)\times \P^1(\R)$ or to $\Delta$.
If $\nu$ is a probability measure on $\Gamma$   satisfying  Condition (S),
then there is a unique $\nu$-stationary measure $\mu_\nu$ on $\P^1(\R)\times\P^1(\R)$, this measure is supported on $\Delta$, and it is not invariant. Moreover, the action of $\Gamma$ on $\pi_1(\P^1(\R)\times\P^1(\R))\simeq \Z^2$ is trivial, and the entropy of every $f\in \Gamma$ vanishes. 
The standard action of $\SL_2(\Z)$ on $\R^2\subset \P^2(\R)$ provides a similar example, for which the origin is fixed and rational points have discrete orbits in~$\R^2$. \end{eg} 

\noindent{\textbf{Notes on the proof.-- }}
The mechanism for the stiffness property in 
Theorem~\ref{thm:blanc_examples} is quite 
different from that of Theorem~\ref{thm:stiffness_for_real_surfaces}. 
This time, we use 
Theorem~\ref{thm:non_random} and~\cite{br}, together with   the existence of a singular invariant volume to prove by contradiction that all ergodic  stationary probability 
measures $\mu$ have vanishing Lyapunov exponents. So by the invariance principle all stationary measures are invariant, and the existence of parabolic elements in the group allows for a complete classification (see Theorem~\ref{thm:intro_classification_invariant_measures} below). 

\smallskip

All the examples of non-elementary groups encountered so far admit an invariant 
 probability measure. 

\begin{que}
Does there exist a non-elementary group action on a rational projective surface without any invariant probability measure? 
\end{que}

Such an example could be found only on a rational surface. 
A natural candidate would be Lesieutre's tri-Coble 
examples, briefly described in \S~\ref{subs:lesieutre} below.   

Another question, which is somehow dual to the previous one, is the following:

\begin{que}
If $\Gamma \subset \Aut(X)$ is a 
non-elementary  subgroup preserving a smooth volume form, 
does $\Gamma$ admit a hyperbolic stationary measure on $X$
(for some probability measure on $\Gamma$ satisfying (S) and (M))?  
\end{que}

If $\Gamma$ possesses a parabolic element, then by \cite[Thm 1.3]{hyperbolic}, 
the invariant volume itself is hyperbolic. 
In the general case,    if there is no hyperbolic measure, then by the invariance principle,  every stationary measure is invariant, so this is really a question about the classification of invariant measures (more precisely, about the second line of the table in Section~\ref{sec:classification}). Note that without the 
volume preservation assumption, Theorem~\ref{thm:blanc_examples} shows that the answer to this question (on $X(\R)$) is negative. 
Note also the connection with the following fundamental open problem from conservative dynamics:
 \emph{if $f$ is a volume-preserving loxodromic 
 automorphism on a compact complex surface $X$, does the invariant volume have 
positive  measure theoretic entropy?}

These questions,  as well as the proof of Theorem~\ref{thm:blanc_examples} 
are a source of motivation for the classification of invariant measures, 
which   we discuss now.

 \subsection{Classification of invariant measures}\label{subs:classification}
Probability measures invariant under  a  non-elemen\-tary group are subject to so many constraints 
that it is reasonable to hope for a complete classification. 
Still, in its full generality, this problem remains out of reach for the moment\footnote{
To illustrate the difficulty of this kind of  question, it may be interesting to 
recall that  Furstenberg's famous problem asking for a
classification of measures on the circle invariant under $\times 2$ and $\times 3$ is still open to date 
(this is  however  quite different from our setting   since it concerns an Abelian semigroup).}.


\begin{thm}\label{thm:intro_classification_invariant_measures}
Let $\Gamma\subset\Aut(X)$ be a non-elementary group of automorphisms of a projective surface  containing a parabolic element. 
Let $\mu$ be an ergodic,  Zariski diffuse, $\Gamma$-invariant measure. 
There exists a proper $\Gamma$-invariant subvariety $Z\subset X$ such that:
\begin{itemize}
\item either $\mu$ is a measure with a smooth positive density in $X\setminus Z$;
\item or there exists a $\Gamma$-invariant totally real surface embedded  in $X\setminus Z$ such   
that $\mu$ is a measure with a smooth positive density on $\Sigma$. 
\end{itemize}
\end{thm}

This result, obtained in~\cite{cantat_groupes, invariant}, shows that 
\emph{if $\Gamma$ is non-elementary and contains a parabolic element, 
the classification problem for invariant measures 
is completely solved}
(see~\cite[Thm A]{invariant} for a precise statement). 
This classification comes with an almost complete  description of orbit closures 
(which is most satisfactory under additional assumptions, see \S~\ref{subs:orbit_closures}). 
A consequence of this classification and of the results described 
in Section~\ref{subs:intro_finite_orbits} is that 
\emph{if $\Gamma$ does not preserve any curve, there are only {finitely many} ergodic $\Gamma$-invariant measures, 
unless $X$ is a torus.}

\noindent{\textbf{Notes on the proof.-- }}
The first input to prove Theorem~\ref{thm:intro_classification_invariant_measures}  is that if $h\in \Gamma$ is a parabolic automorphism preserving the fibers of a genus $1$ fibration  $\pi\colon X\to B$, then $h$ acts as a uniquely ergodic translation on most fibers of $\pi$ (it also acts periodically on a dense, countable set of fibers). Thus, any $h$-invariant measure can be described in terms of its disintegration along the fibers of $\pi$. From this it follows rather easily that 
any ergodic Zariski diffuse $\Gamma$-invariant  measure satisfies a local homogeneity property (i.e.\  some local invariance under connected groups of translations), 
which implies that it is either absolutely continuous with respect to the Lebesgue measure on $X$, or 
to the Lebesgue measure on an immersed totally real surface $\Sigma$. The delicate issue is to 
show, in the latter situation, that $\Sigma$ globalizes to an embedded 
 submanifold outside some algebraic subset $Z$.  
 
When $\Gamma$ does not contain parabolic elements, the classification of invariant measures is still 
at an early stage, and  is arguably the main open problem in our program.  
It can be subdivided in several sub-problems of independent interest; 
since the discussion at this point becomes rather technical, we defer it to Section~\ref{sec:classification}.  
Let us just mention one open question, which 
we formulate   in the easiest context of  automorphisms of the affine plane. 
 
\begin{que}\label{que:intro_henon}
Can one construct two polynomial automorphisms $f$ and $g$ of $\C^2$ such that 
\begin{enumerate}
\item[(i)] $f$ and $g$ fix the origin $o$;
\item[(ii)] there is a germ of holomorphic diffeomorphism $\varphi\colon (\C^2, o)\to (\C^2, o)$ that conjugates $f$ and $g$ simultaneously to two elements of $\SU_2(\C)$;
\item[(iii)] the group generated by $f$ and $g$ is a non-Abelian free group containing loxodromic automorphisms (i.e.\  automorphisms conjugate to a composition of generalized Hénon transformations). 
\end{enumerate}
\end{que}

Our expectation is that such an example does not exist. If such an example were to be found, it would have an invariant probability measure for each small sphere centered at the origin.

\subsection{Finite orbits and arithmetic dynamics}\label{subs:intro_finite_orbits}
For a non-elementary group $\Gamma$, the existence of a finite orbit is an 
overdetermined property. Indeed, if $f\in \Aut(X)$ is a loxodromic element, then its periodic points of sufficiently large period form a 
countable subset of $X$, and any $\Gamma$-periodic point must be $f$-periodic for every $f\in \Gamma$. 
Since $\Gamma$ is non-elementary, it contains two loxodromic automorphisms $f$ and $g$ satisfying no algebraic relation (see \S~\ref{subs:NE}). In this case it is  expected that 
$\Per(f)\cap \Per(g)$ is finite,  
or contained in a curve of $\bra{f,g}$-periodic points.

\begin{eg}\label{eg:no_finite_orbit_in_families} 
To implement these ideas, 
a first approach is  to work in families and prove generic results. 
In the family of   Wehler examples
(see Example~\ref{eg:intro_Wehler}), we proved in~\cite[Thm A]{finite_orbits} that a very general member   has no periodic orbit.  Here, by {\emph{very general}}, 
we mean outside a countable union of nowhere dense Zariski closed subsets in the parameter space. 
Such a result  is still open for other natural families  (see Questions~\ref{que:Enriques-finite-orbits} and~\ref{que:Gamma-vs-Gammaext} below). In this paper, we 
 develop new methods to get similar results for pentagon foldings 
 (see Part~\ref{part:pentagons}) and Blanc's examples (see Part~\ref{part:ergodic_blanc}). \end{eg}
 
A deeper problem is to deal with specific, individual examples. Then the existence of infinitely many 
finite orbits for a non-elementary group falls into the setting of \emph{unlikely intersection problems} (see~\cite{Zannier:AnMaSt}), and  it will not be a surprise to the experts that methods from 
arithmetic geometry arise here. For instance, the following result was established in~\cite[Thm B]{finite_orbits}. 

\begin{thm}\label{thm:finite_orbits}
Let $X$ be a projective surface defined over $\overline \Q$. Let $\Gamma\subset \Aut(X)$ be a 
non-elementary group, defined over $\overline \Q$ and containing a parabolic element. If 
$\Per(\Gamma)$ is Zariski dense, then $(X, \Gamma)$ is a Kummer example. 
\end{thm}

Thus, when $X$ is not a Kummer surface,  $\Per(\Gamma)$ is the union of a  $\Gamma$-invariant curve, together with finitely many isolated points\footnote{Here, we use the following fact: if $\Gamma$ acts by automorphisms on a curve $C$ with a Zariski dense set of periodic points, then the image of $\Gamma$ in $\Aut(C)$ is finite. Indeed, $C$ has only finitely many irreducible components, and if a subgroup $\Gamma_0$ acts on some irreducible
component $C_i$ by fixing a subset $F_i$ such that the Euler characteristic of $C_i\setminus F_i$ is negative, then the restriction of $\Gamma_0$ to $C_i$ is finite.}.  

\noindent{\textbf{Notes on the proof.-- }}Fix a probability measure $\nu$ on $\Gamma$ with finite support generating $\Gamma$. 
First, we construct a height function $h_\nu$ on $X(\overline \Q)$  
such that $\sum_f \nu(f) h\circ f = \lambda_\nu h$ for some $\lambda >1$. 
It follows that  $\Gamma$-periodic points have  zero  height.
Assuming the existence of a Zariski dense set of periodic orbits, 
and using Yuan's equidistribution theorem for small points, we 
construct a  $\Gamma$-invariant measure with special properties.  To conclude, we use the classification 
of invariant measures described in Theorem~\ref{thm:intro_classification_invariant_measures}, and a rigidity result proved in~\cite{cantat-dupont}.
The existence of parabolic elements is used at several places in the proof, 
notably to apply Theorem~\ref{thm:intro_classification_invariant_measures}   (see~\cite[\S 6.3]{finite_orbits} for a discussion). 
Corvaja, Tsimerman and Zannier~\cite{corvaja-tsimerman-zannier} recently proved a 
closely related theorem
which also requires
 parabolic elements. Their proof is based on different ideas (o-minimal geometry, 
variation of the canonical height for families of elliptic curves, etc).

\begin{rem} \label{rem:wehler_markov_finite_orbits} 
In~\cite{fuchs-litman-silverman-tran},  Wehler examples with a finite but large number of periodic points are constructed. 
In the closely related case of Markov type surfaces 
\begin{equation}
x^2+y^2+z^2=xyz+ax+by+cz+d,
\end{equation}
where $(a,b,c,d)\in \C^4$ is a parameter, 
all possible finite orbits have been classified (see~\cite{Lisovyy-Tykhyy}). The  only 
 parameter with infinitely many 
periodic orbits is $(0,0,0,4)$; the corresponding surface is the quotient of the multiplicative group $\C^\times\times\C^\times$ by the involution $(u,v)\mapsto (1/u,1/v)$, and the group action on $M_{(0,0,0,4)}$ lifts to the monomial action of $\GL_2(\Z)$; finite orbits correspond to torsion points. Thus, the situation is perfectly  similar to that of Kummer examples,  but for a multiplicative torus.
\end{rem}   

Surprisingly enough, Theorem~\ref{thm:finite_orbits} can be used to prove the same result when $X$ 
and $\Gamma$ are defined over $\C$; 
in that case we must   suppose (for technical reasons) that  $\Gamma$ 
does not preserve any curve, and 
the existence of infinitely many periodic orbits then implies that $X$ is a torus 
(see~\cite[Thm C]{finite_orbits}). 
Related, stronger, results were obtained for polynomial automorphisms of $\C^2$, and 
more generally affine surfaces  \cite{dujardin-favre, abboud:these}. This motivates the following question:

\begin{que}
Can we classify pairs of   automorphisms $(f,g)$ of positive entropy on a compact projective  
surface having a Zariski dense set of common periodic points? 
Or having the same maximal entropy measure? 
\end{que}

Note that in the second question the common maximal entropy measure would be  a 
$\bra{f, g}$-invariant measure, so we are back to the theme of \S~\ref{subs:classification}.
If one can show that this measure is smooth, then the main result of \cite{cantat-dupont}
implies that $(X,\langle f,g\rangle)$ is a Kummer example; the proof of Theorem~\ref{thm:finite_orbits} relies on this fact. 
  
In one-dimensional algebraic dynamics, unlikely intersection problems are better understood. For instance, one can produce {\emph{uniform bounds}} on the number of common periodic points (instead of orbits) for pairs of distinct quadratic maps $z\mapsto z^2+c$, $z\mapsto z^2+c'$: we refer to~\cite{demarco-krieger-ye_legendre, demarco-krieger-ye_quadratic} for precise results. 
Here is a version of this  problem for non-elementary groups of automorphisms:

\begin{que}
Is there a uniform bound for $\#\Per(\Gamma)$ in the  family of (smooth) Wehler examples? 
\end{que}

Note that for a general parameter in the Wehler family there is no invariant curve, so 
$\Per(\Gamma)$ is indeed finite (see~\cite[Prop. 3.6]{stiffness}). Examples of  Wehler groups with 
$\#\Per(\Gamma)\geq 288$ were constructed in~\cite{fuchs-litman-silverman-tran} (see~Remark~\ref{rem:wehler_markov_finite_orbits}). 
On the other hand, some singular Wehler examples are Kummer 
 hence give rise to a dense set of periodic orbits (see~\cite[Ex. 8.3]{Cantat:Panorama-Synthese}).

Let us conclude this section with a last question which may be considered as a variation on a prediction of Kawaguchi and Silverman~\cite{kawaguchi-silverman}. 
Assume that $X$ and $\Gamma$ are defined over $\overline{\Q}$ and that $\Gamma$ is non-elementary. 
Fix a polarization $H$ of $X$, denote by $h\colon X(\overline{\Q})\to \R_+$ a Weil height associated to $H$. For every $f\in \Aut(X)$,    let 
\begin{equation}
\deg_H(f)=H\cdot (f^*H).
\end{equation}
 Fix a point $x\in X(\overline{\Q})$ with a Zariski dense orbit, and set 
\begin{equation}
N(x; R):=\# \set{y\in \Gamma\cdot x\;  ;  \; h(y) \leq R}. 
\end{equation}
Kawaguchi and Silverman suggested 
that the Zariski density of $\Gamma\cdot x$ should imply that $h(f(x))$ is of the same order of magnitude as
$\deg(f)$.  More precisely, introduce the counting function
\begin{equation}\label{eq:counting_degree}
N_{\deg}(R):= \#\set{f\in \Gamma\; ; \;  \deg(f)\leq R}, 
\end{equation} 
and   ask: 
  
\begin{que}\label{que:baragar_general} 
If $\Gamma\subset \Aut(X)$ is a   non-elementary automorphism group 
 defined over $\overline \Q$, and $x\in X$ has a Zariski dense  orbit, does $N(x; R)$ grow like  $N_{\deg}(R)$?
\end{que}

The group $\Gamma$ acts on a hyperbolic space $\Hyp_X\subset H^{1,1}(X;\R)$ (see Section~\ref{sec:projective}), 
and $\log(\deg_H(f))$ is comparable to the distance in $\Hyp_X$ between the cohomology class $(H\cdot H)^{-1/2} [H]$ and its image by $f^*$.
From this, we see  that $N_{\deg}(R)$ is  a classical counting function for orbits of discrete isometry groups in 
 hyperbolic geometry. As a consequence, $N_{\deg}(R)$ 
 grows like $R^\alpha$, where $\alpha$ is the Hausdorff dimension of the limit set of the isometry group $\Gamma^*\subset \Isom(\Hyp_X)$ (see~\cite{baragar:JNT,Dolgachev_counting}).
Question~\ref{que:baragar_general} was answered positively by Baragar in~\cite{baragar:JNT} for some K3 examples; this problem is also well understood for Markov surfaces (see~\cite{gamburd_icm} for an account). The interested reader should also consult~\cite{Filip-Tosatti:Heights}.

\subsection{Equidistribution and uniform expansion}\label{subs:UE}
Up to now, we addressed Problems (Pb2), (Pb3), (Pb5), 
i.e.\ the study of stationary and invariant measures. 
We asserted in \S~\ref{subs:generalities} that classifying stationary measures is 
the key step to understand the asymptotic, stochastic distribution of orbits, or more precisely the limiting behavior of 
the averages~\eqref{eq:orbital} and~\eqref{eq:empiric}. This assertion, while certainly
correct, hides a final difficulty. Imagine an ideal situation where the set
of 
$\nu$-stationary measures is completely described, 
and its extremal points are given by some smooth (or Zariski diffuse) measure $\mu_0$ and a finite or countable set of 
finite orbits. If $x\in X$ is a general point, the limit points of 
the averages~\eqref{eq:orbital} and~\eqref{eq:empiric} are 
convex combinations of these ergodic measures. The problem is to decide which combinations do arise; this is where 
uniform expansion comes into play. 
Let $\nu$ be a probability measure on $\Aut(X)$, and set $\Gamma=\langle\supp(\nu)\rangle$. We say
 that $\nu$ 
  is \textbf{uniformly expanding} if there exists $c>0$ and an integer $n_0$ such that for \emph{every} $x\in X$ and \emph{every} $v\in T_xX\setminus\{ 0\}$, 
 \begin{equation}\label{eq:UE}
 \int \log\frac{\norm{D_xf (v)}}{\norm{v}} \d\nu^{(n_0)}(f) \geq c,
 \end{equation}
where $\nu^{(n)}$ denotes the $n^\mathrm{th}$ convolution power of $\nu$, and the norm is relative 
to some given Riemannian metric on $X$. This notion was introduced 
 in non-linear random dynamics in~\cite{liu_thesis, chung}, notably 
  to get equidistribution results analogous to those
   obtained in homogeneous dynamics in~\cite{eskin-margulis, benoist-quint3}. Now, suppose that $\nu$ is uniformly expanding and satisfies 
\begin{equation}
\tag{$\mathrm M_+$}
\text{there exists }\, p>1 \, \text{ such that }\; 
\int \lrpar{\log\norm{f}_{C^1(X)}+ \log\norm{f\inv}_{C^1(X)}}^p \, \d\nu (f) < + \infty.
\end{equation}
Let $F$ be a finite $\Gamma$-orbit. Then,  
\emph{if $x\in X$ has an infinite orbit, any cluster value of the sequence of measures 
in~\eqref{eq:orbital} or~\eqref{eq:empiric} gives zero mass to $F$} 
(see~\cite[Thm 4.3]{hyperbolic}). 
Indeed,  the uniform expansion condition makes $F$ repelling on average, an idea which is formalized by the notion of \emph{Margulis function}. 
So in the ideal situation described above, every infinite orbit equidistributes towards the smooth stationary measure $\mu_0$.  Moreover, in the real
volume-preserving setting, where two natural  Zariski diffuse ergodic measures exist,  given respectively by an invariant volume form on $X$ and an invariant area form on $X(\R)$,  
we can also decide which orbits are equidistributed with respect to the first or the second measure 
(see~\cite[\S 4.3]{hyperbolic}).

In homogeneous dynamics, 
 establishing uniform expansion boils down to an expansion 
property  for a  random product  of matrices. In a non-linear context, the situation is more delicate; fortunately,   an abstract ergodic criterion for expansion was devised  in~\cite{liu_thesis, chung},
which fits well with our holomorphic setting. This
leads to a neat necessary and sufficient condition for uniform expansion
on non-rational surfaces  (see~\cite[Thm 1.5]{hyperbolic}):
 
 \begin{thm}\label{thm:criterion_uniform_expansion}
Let $X$ be a compact complex surface which is not rational.
Let $\Gamma\subset \Aut(X)$ be a non-elementary group containing parabolic elements. 
 Let $\nu$ be a  probability measure on $\Aut(X)$   satisfying  Conditions~(S) and~(M).  
Then, $\nu$ is uniformly expanding if and only if the following two conditions hold:
\begin{enumerate}[ \em (1)]
\item every finite $\Gamma$-orbit is uniformly expanding;  
\item  there is no $\Gamma$-invariant algebraic curve.
\end{enumerate}
\end{thm}

Again, we assume that $\Gamma$ contains a parabolic element because  the 
 classification of  invariant measures is used in the proof.  

In many interesting situations, we have a family $(X_\lambda; \Gamma_\lambda)_{\lambda\in \Lambda}$ of non-elementary automorphism groups, parameterized 
by some quasi-projective manifold. Then, as in Example~\ref{eg:no_finite_orbit_in_families}, we developed
 tools to exclude the existence of proper, Zariski closed invariant subsets 
 {\emph{for very general parameters $\lambda\in \Lambda$}}. 
 Hence, for these parameters, uniform expansion holds. In addition,  Condition~\ref{eq:UE} is robust under $C^1$ perturbations, so this dense subset of uniformly expanding parameters is automatically promoted to an open and dense subset. In 
Part~\ref{part:pentagons}, we apply these  ideas to the dynamics of    random pentagon foldings.

Another interesting consequence of uniform expansion is that a volume preserving 
uniformly expanding action on a (connected, compact, real or complex) surface 
 is automatically ergodic with respect to the invariant volume. 
 This follows from a Hopf-type argument (see \cite{dolgopyat-krikorian, chung}). 
 An interesting question (which goes beyond complex dynamics) is 
 whether such an ergodicity can be made quantitative (for uniformly expanding actions). 
To formalize this question, denote by  $\d\vol$ the invariant volume form. Then, consider the Markov operator 
\begin{equation}
P_\nu: \varphi\mapsto \int \varphi\circ f \, \d\nu(f) 
\end{equation}
acting on some Banach space of continuous functions.  
Then, the question is to determine the speed of convergence of 
$P_\nu^n$ towards $\int \varphi \, \d\vol$. Is it exponential? 
This problem is open even in the homogeneous case, 
except for linear maps on tori, for which the Fourier techniques of~\cite{BFLM} provide effective estimates. (\footnote{Note: 
this question is addressed in 
the recent preprint~\cite{dewitt-dolgopyat}, proving exponential mixing for area preserving, expanding, random dynamical systems on surfaces.}) 

\subsection{Back to orbit closures}\label{subs:orbit_closures} Putting together    the   results of \S\S~\ref{subs:intro_stiffness}, \ref{subs:classification}, \ref{subs:intro_finite_orbits} and \ref{subs:UE}, 
we arrive at a complete understanding of orbit closures, under rather  strong 
assumptions (see~\cite[Thm. 10.1]{hyperbolic}).

\begin{thm}
Let $X$ be a compact complex surface which is not rational.
Let $\Gamma\subset \Aut(X)$ be a non-elementary group   containing  parabolic elements, and satisfying   conditions~(1) and~(2) of Theorem~\ref{thm:criterion_uniform_expansion}.  Then 
there exists a finite set $F$ and a real-analytic (possibly singular) totally real surface $\Sigma$, 
both $\Gamma$-invariant,  such that 
\begin{itemize}
\item if $x\in F$ then $\Gamma\cdot x $ is finite;
\item if $x\in \Sigma\setminus F$, then  $\overline {\Gamma\cdot x }$ is a union of components of $\Sigma$;
\item otherwise $\overline {\Gamma\cdot x } =X$. 
\end{itemize}
\end{thm}

Under less stringent hypotheses, but still assuming that $\Gamma$ contains parabolic elements, 
weaker  results on orbit closures are obtained in~\cite[\S 8]{invariant}, which deserve further study.

\begin{prob}
Complete the classification of orbit closures when $\Gamma$ is non-elementary and 
contains a parabolic element. 
\end{prob}

\noindent\textbf{Comments.--}   In~\cite[\S 8]{invariant}, we classified closed, $\Gamma$-invariant subsets $F$ 
whose accumulation locus $\mathrm{Acc(F)}$ is not contained in some explicit invariant 
algebraic subset $\mathrm{STang}_\Gamma$. Thus, when 
there is no proper invariant algebraic subset,  or 
more generally when uniform expansion holds, this leads to a complete classification. 
In the general case we must understand the situation where 
$\mathrm{Acc}(\Gamma\cdot x)\subset \mathrm{STang}_\Gamma$. It is easy to see that in this situation $x$ must be $g$-periodic for each parabolic $g\in \Gamma$. In this case one expects that $x\in \Per(\Gamma)$, but this result is not yet available (see~\cite{corvaja-tsimerman-zannier} for partial results).
See Theorem~\ref{thm:blanc_orbit_closures_real} below for a worked out example.

 \section{Appendix: classification of invariant measures}\label{sec:classification}
 
This section  is a complement to \S~\ref{subs:classification} and may be skipped on a first reading. 
We discuss what needs to be done to describe invariant probability measures for 
non-elementary groups that, a priori, do not contain parabolic elements. 

\subsection{Reducing the number of  cases} We fix a complex projective surface $X$ and a non-elementary subgroup $\Gamma$ of $\Aut(X)$. 
 Let  $\mu$ be a Zariski diffuse,  ergodic, $\Gamma$-invariant probability measure.  It is useful 
 to introduce a probability measure $\nu$ on $\Gamma$ satisfying Conditions~(S) and~(M) in order to speak about 
 the Lyapunov exponents of $\mu$; since they depend on $\nu$, we denote them by $\lambda^-(\nu)\leq \lambda^+(\nu)$.  
 
A preliminary observation is that the exponents cannot be simultaneously positive (resp. negative), because in such a case a classical Pesin-theoretic argument implies that  $\mu$ would be atomic. 
So $\lambda^-(\nu) \leq 0\leq \lambda^+(\nu)$ and if $\lambda^-(\nu) = \lambda^+(\nu)$, then both exponents vanish. 
This gives three distinct possibilities: either $\lambda^-(\nu) <0<\lambda^+(\nu)$, or exactly one of the exponents vanish, or both of them vanish.
We shall also distinguish three possibilities, depending on the number of invariant line fields: (1) no invariant line field, or pair of invariant line fields; 
(2)  there exists a measurable $\Gamma$-invariant line field $E\subset TX$; 
 (3) there exists an invariant  measurable pair of line fields  $E_1 , E_2\subset TX$.
 This leads to the following table of $12$ possibilities:

\medskip

\begin{small}
\newcolumntype{M}[1]{>{\raggedright}m{#1}}
\noindent
\begin{tabular}{| M{4cm} |M{3cm}|M{3cm}|M{3.2cm}|}
\hline
 & no invariant line field or pair of line fields & $\exists E\leq TX$ \\ invariant line field & $\exists E_1, E_2\leq TX$ pair of 
 invariant line fields \tabularnewline \hline
 $\exists \nu, \ \lambda^+(\nu) = \lambda^-(\nu) =0$ &\centering A1&\centering A2&\centering A3  \tabularnewline \hline
 $\forall \nu, \ \lambda^+(\nu) = \lambda^-(\nu) =0$ &\centering B1&\centering B2&\centering B3 \tabularnewline \hline
 $\exists \nu, \ \lambda^-   < \lambda^+ $  \& $\lambda^-   \lambda^+ =0 $ &\centering C1& \centering C2&\centering C3 \tabularnewline \hline
$\exists \nu, \ \lambda^+ (\nu) <0   < \lambda^- (\nu) $ &\centering D1& \centering D2& \centering D3 \tabularnewline \hline
 \end{tabular}
 \end{small}
 \medskip
 
Let us discuss  this table and the  relationship between its entries. 
Note that the third row of this table does not appear in the volume preserving case. 

\subsubsection{The first column} The case when both exponents vanish (the first line) is  analyzed in  
Section 7 of  \cite{hyperbolic}. In case A1,  Theorems 7.2 and 7.3 of \cite{hyperbolic} imply that the cocycle given by the  action of $\Gamma$ on the projectivization of the tangent space is measurably reducible to a compact group. In particular A1 implies B1.

The proof of Theorem 5.1 in \cite{br} (see \S 13.2.4) shows that case C1 
reduces to case D1.

\subsubsection{The second column}\label{par:second_column} When there is a measurable 
 invariant line field $x\mapsto E(x)$, the Lyapunov exponent in the direction of $E$, which is one of $\lambda^-(\nu)$ or $\lambda^+(\nu)$,
  is given by the explicit formula 
 \begin{equation}\label{eq:lambda_E}
 \lambda_E(\nu) = \int \log\norm{Df\rest{E(x)}} \d\mu(x) \d\nu(f) = \int \lambda_E(f, \mu) \d\nu(f),
 \end{equation}\label{eq:explicit_exponent_1}
where  $\lambda_E(f, \mu)$ is the integrated Lyapunov exponent of $\mu$ as an $f$-invariant measure (note that 
$\mu$ is not necessarily $f$-ergodic). The second exponent admits also a simple expression, because
it is equal to 
$\lambda^-(\nu) + \lambda^+(\nu) - \lambda_E(\nu)$, and
\begin{equation}
\lambda^-(\nu)+\lambda^+(\nu) = \int \log\abs{\mathrm{Jac} f(x)} \d\mu(x) \d\nu(f),
\end{equation}
where $\mathrm{Jac} $ is the Jacobian 
determinant with respect to any given smooth volume form on $X$.

Let us show that A2 reduces to B2, C2 or D2. There are two possibilities: either $\lambda_E(f, \mu) =0$ 
for every $f$, or not. In the first case we are in situation B2. In the second, we can change $\nu$ to put more weight on an element 
with $\lambda_E(f, \mu) \neq 0$ to impose $ \lambda_E(\nu)\neq 0$; doing so, we  fall in C2 or D2
(moreover, in C2 the non-vanishing exponent is along the invariant line field).

\subsubsection{The third column}
There are two possibilities: either every $f\in \Gamma$ 
preserves (resp. exchanges) the directions $E_1$ and $E_2$ almost everywhere, or not. 
In the first case, there is an index 2 subgroup preserving both directions, 
and we are in a special case of column 2. Otherwise,  the directions are intertwined by the dynamics, 
and we claim that we are in case B3 (this argument is a variation on~\cite[Lem. 5.2]{br}). 
Indeed, note that by definition there is a set of positive, hence full, $\mu$-measure where $E_1(x)\neq E_2(x)$, and let
\begin{equation}
\Lambda(\nu)  = \unsur{2} \int \lrpar{ \log\norm{Df\rest{E_1(x)}} + \log \norm{Df\rest{E_2(x)}} }\d\mu(x) \d\nu(f).
\end{equation}
By ergodicity, for $(\nu^\N\times \mu)$-almost every $(\omega, x)$ and every $v\in E_1(x)\cup E_2(x)\setminus\set{0}$, 
\begin{equation}
\lim_{n\to +\infty} \unsur{n}\log \norm {D f^n_\omega(x)} =  \Lambda(\nu).
\end{equation} 
If 
 $\lambda^+(\nu) = \lambda^-(\nu)$, then,  by the preliminary remark, both exponents vanish. If 
 $\lambda^-(\nu)<\lambda^+(\nu)$, then by Oseledets' theorem, for $\mu$-a.e. $x$ there is a line $E^s(x)$ such that if $v\notin E^s(x)$, then $\unsur{n}\log \norm {D f^n_\omega(x)}\to \lambda^+$. 
Thus $E^s(x)\notin \set{E_1(x), E_2(x)}$ and $\Lambda = \lambda^+$. But now there is a set of 3 invariant directions, so the projective tangent action recurs to a compact set (see the proof of Theorem 7.3 in \cite{hyperbolic})
and we conclude that
$\lambda^+(\nu) = \lambda^-(\nu)= 0$, a contradiction.

\subsubsection{Summary} After these reductions,  the situation is summarized in 
 the following table. Colored cells do not need to be treated because they reduce to other ones. 

\medskip

\begin{small}
\newcolumntype{M}[1]{>{\raggedright}m{#1}}
\noindent
\begin{tabular}{| M{4cm} |M{3.2cm}|M{3cm}|M{3.2cm}|}
\hline
 & no invariant line field or pair of line fields & $\exists E\leq TX$ \\ invariant line field & $\exists E_1, E_2\leq TX$ pair of 
 invariant line fields \tabularnewline \hline
 $\exists \nu, \ \lambda^+(\nu) = \lambda^-(\nu) =0$ &\cellcolor{gray}&\cellcolor{gray}&\cellcolor{gray}  \tabularnewline \hline
 $\forall \nu, \ \lambda^+(\nu) = \lambda^-(\nu) =0$ &B1: Fatou behavior? & \centering B2  & \centering B3 \tabularnewline \hline
 $\exists \nu, \ \lambda^+   < \lambda^- $  \& $\lambda^-   \lambda^+ =0 $ &\cellcolor{gray}& \centering  
C2 &\cellcolor{gray} \tabularnewline \hline
$\exists \nu, \ \lambda^+ (\nu) <0   < \lambda^- (\nu) $ & D1: $\mu$  homogeneous? & D2: $\mu$ not Z. diffuse?& \cellcolor{gray} \tabularnewline \hline
 \end{tabular}
 \end{small}
 \medskip
 
\subsection{Comments} Let us now analyze some of the  remaining cases.
We first note that we do not know any example, nor even a possible model for the dynamics, for the cases B2, B3, and C2 (in C2 it would make sense  to further distinguish the cases whether the Lyapunov exponent in the invariant direction vanishes or not).
 
$\bullet$ B1: in that case, \cite[Thm. 7.3]{hyperbolic} shows that the cocycle describing the action  of $\Gamma$ on the projectivized tangent bundle of $X$ is measurably reducible to a compact subgroup of $\GL_2(\C)$. 
This suggests a Fatou-type behavior for the dynamics. 
\emph{Must an invariant measure of type B1 
be contained in the Fatou set of the $\Gamma$-action?}
As far as we know, there is no known example of a 
Fatou domain for a  non-elementary group action by automorphisms on a projective  
surface. See Question~\ref{que:intro_henon} for a related problem.  

%
%

$\bullet$ D1: $\mu$ is hyperbolic and there is no invariant line  field. 
The techniques of~\cite{br} show that in this case  $\mu$ has some 
homogeneity properties along stable and unstable manifolds. 
This implies in particular that $\mu$ has positive fiber entropy (as a $\nu$-stationary measure). In a nutshell, recall that the fiber entropy   can be defined by 
\begin{equation}
h_\mu(X, \nu) = h_{\nu^\N\times \mu}(F_+)   - h_{\nu^\N}(\sigma)
\end{equation}
and it quantifies the relative entropy of $F_+$ in the fibers of the projection 
$\Omega\times X\to X$
(of course this definition needs to be adapted when $h_{\nu^\N}(\sigma) = \infty$, 
see \cite[\S 7.6]{stiffness} for a brief presentation). 
\emph{In the volume preserving (rational) case, it should expected that, under   assumption D1,
 $\mu$ is either absolutely continuous or 
absolutely continuous along some real analytic surface $\Sigma$.
In the general case, this has to be replaced by  a SRB property.}
Partial results in this spirit include~\cite[Thm 5.1]{br} and~\cite[Thm 11.1]{stiffness} for the real case.
In the general complex case, such a result seems to be 
 out of reach for the moment, even using~\cite{BEFR}. 
 The most delicate point would be to construct $\Sigma$ from an information on the 
 stable and unstable conditionals of $\mu$ only. 
 
$\bullet$ D2: $\mu$ is hyperbolic and there is an invariant line field. Since $\mu$ is $\Gamma$-invariant, the natural extension of $(F_+, \nu^\N\times \mu)$ is  $(F, \nu^\Z\times \mu)$, 
where $F$ is the skew product over the 2-sided shift associated to the random dynamical system 
$(X, \nu, \mu)$ (see \cite[\S 7.1]{stiffness}). 
For this invertible dynamical system, the invariant line field corresponds to either $E^s$ or $E^u$. 
Note that by the invariance of $\mu$, 
  $F\inv$ is the skew product map associated to the random dynamical system 
induced by the reversed measure $\check \nu$, defined by $\check \nu(f) = \nu(f\inv)$.  Therefore, if 
$E = E^u$,  $E$ is the stable Oseledets 
bundle associated to the random dynamical system $(X,\check \nu, \mu)$. From this discussion, we conclude that in either case we may assume that $E = E^s$, i.e. 
the field of Oseledets stable directions is non-random. Theorem 9.1 in~\cite{stiffness} then asserts that 
$h_\mu(X, \nu)=0$, and furthermore by~\cite[Thm. B.1]{hyperbolic}, $h_\mu(f) = 0$ for every $f\in\Gamma$. In other words,  $\mu$ is a common zero entropy measure for all group elements. We 
expect that such a situation does not happen, so the conclusion should be that 
\emph{a Zariski diffuse invariant measure cannot satisfy D2}.

\subsection{A dual point of view and an example}\label{subs:dual_example} Let us remark that 
another way of formulating this problem is to start with a loxodromic automorphism 
$f$ of $X$ and some $f$-invariant measure $\mu$
 and to ask for a description of the stabilizer of $\mu$ in $\Aut (X)$ 
 (cf.~\cite[Thm 5.1]{br} and~\cite[Thm 11.1]{stiffness}). Here is one instance of this problem:
 
 \begin{que}
 Let $f$ be a loxodromic automorphism of a complex projective surface $X$, and let $\mu_f$ be its unique invariant measure of maximal entropy. Is
 the stabilizer  of $\mu$ in $\Aut(X)$ virtually Abelian? \end{que}

To conclude, we answer a similar question in the specific case of polynomial automorphisms of the plane. One goal  here
 is to illustrate some dynamical similarities between affine and projective surfaces.
Before stating our result, recall that the Jacobian determinant $\jac(f)$ of such an automorphism $f$ is constant;  when $\jac(f)=\pm 1$, the Lebesgue measure $\leb_{\R^2}$ is invariant and $f$ is said to be conservative.

\begin{thm}\label{thm:rigidity_henon}
Let $f$ be a   polynomial  automorphism of $\R^2$. 
Let $\mu$ be an ergodic  $f$-invariant measure   with positive entropy supported on $\R^2$. 
If $g\in\Aut(\R^2)$  satisfies $g_\varstar \mu   = \mu$,  then: 
\begin{enumerate}[\em (a)]
\item either $f$ and $g$ are conservative and $\mu$ is the restriction of  $\leb_{\R^2}$ to a Borel set of positive measure invariant under $f$ and $g$;
\item or  the group generated by $f$ and $g$ is solvable and virtually cyclic; in particular, there exists $(n,m)\in \Z^2\setminus{\set{(0,0)}}$ such that $f^n = g^m$. 
\end{enumerate}
\end{thm}

\begin{rem}
With the techniques developed in \cite{Cantat:BHPS}, the same result should apply  to the dynamics of 
${\mathsf{Out}}({\mathbb{F}}_2)$ acting on the real 
part of the character surfaces of the once punctured torus. 
\end{rem}

\begin{proof}[Sketch of Proof] 
Since the proof is a direct adaptation of~\cite[Thm 11.1]{stiffness}, we briefly 
explain the argument and leave the details to the reader. 

Set $\Gamma=\langle f, g \rangle$. 
Since its entropy  is positive, $f$ is of H\'enon type in the sense of~\cite{lamy}: it is conjugate to a composition of 
generalized H\'enon maps, as in~\cite[Thm. 2.6]{friedland-milnor}. Thus, the support of $\mu$  is a compact subset of $\R^2$, because the 
basins of attraction of the line at infinity for $f$ and $f^{-1}$, respectively, cover the complement of a compact subset of $\C^2$. 
In addition $\mu$ must be Zariski diffuse, because its entropy is positive. 

Let $w$ be an element of $\Gamma$ and set $h=wfw^{-1}$; then $h$ is also  of H\'enon type. 
We follow the proof of~\cite[Thm 11.1]{stiffness}, which leads us to three cases. 

Case 3 is treated exactly  in the same way  and implies that $\mu$ is absolutely continuous. 
This  implies that $f$ is conservative and, $\mu$ being ergodic for $f$, 
it must be the restriction of $\leb_{\R^2}$ to some $\Gamma$-invariant subset. 

In Cases 1 and 2, arguing as in \cite[Thm 11.1]{stiffness} and keeping the 
same notation, we arrive at $W^s(h,x) = W^s(f,x)$ or $W^u(f,x)$ on a set of positive measure. 
For a H\'enon type automorphism of $\C^2$, the closure of any stable manifold is 
equal to the forward Julia set  $J^+$, and $J^+$ carries a unique positive closed current $T^+$ of mass 
$1$ relative to the Fubini Study form in $\P^2(\C)$ (see \cite{sibony}). So we 
infer that $T^+_h = T^+_f$ or $T^+_h =T^-_f$; as a consequence, the Green functions of $f$ and $h$ satisfy $G^+_h=G^+_f$
or $G^+_h=G^-_f$, respectively.

The group $\Aut(\C^2)$ is the amalgamated product the affine and the elementary subgroups along their intersection. Let $T$ be the associated Bass-Serre tree. Each $u\in \Aut(\C^2)$
gives rise to an isometry $u_*$ of $T$ and,  $u$ is of H\'enon type if and only if $u_*$ is loxodromic (its axis $\Ax(u_*)$ is the unique $u_*$-invariant geodesic, and $u_*$ acts as a translation along it). From~\cite[Thm.~5.4]{lamy}, $G^+_h = G^+_f$ implies  $\Ax(h_*)=\Ax(f_*)$; changing $f$ into
$f\inv$, $G^+_h =G^-_f$ gives $\Ax(h_*)=\Ax(f_*\inv)=\Ax(f_*)$ because $\Ax(f_*\inv)=\Ax(f_*)$. 
Since $w_* \Ax(f_*)=\Ax(h_*)$, we see that $\Gamma$ preserves $\Ax(f_*)$; so, all
$u\in \Gamma$ of H\'enon type satisfy $\Ax(u_*)=\Ax(f_*)$. From \cite[Prop. 4.10]{lamy}, we conclude that $\Gamma$ is solvable and virtually 
cyclic. \end{proof}

 \newpage

\part{Classification and first examples} 

\section{Surfaces admitting non-elementary groups of automorphisms} \label{sec:projective}

\subsection{More Kähler vocabulary} \label{par:compact-kahler-surfaces} 

Let $X$ be a compact Kähler surface. Recall from \S~\ref{subs:generalities} that a subgroup 
$\Gamma\subset \Aut(X)$ is non-elementary when its image $\Gamma^*\subset \GL(H^*(X, \Z))$ 
contains a  non-Abelian   free subgroup.  
By Hodge theory, $\Gamma$ acts on $H^{1,1}(X, \R)$ by preserving the intersection form, which is of Minkowski type; thus 
 $\Gamma$ acts by isometries on the associated hyperbolic space $\Hyp_X$, 
 which is the component of the hyperboloid 
 \begin{equation}
 \set{u\in  H^{1,1}(X, \R); \ \bra{u\vert u}=1}
 \end{equation} 
 containing the class of a Kähler form.  Then  $\Gamma$ is  non-elementary in the sense of~\S~\ref{subs:NE} if, and only if,
 the induced subgroup of 
 $\Isom(\Hyp_X)$ is non-elementary in the sense of hyperbolic geometry 
 (this is the original definition from~\cite[\S 2]{stiffness}). The classification of automorphisms in the elliptic, parabolic, and loxodromic types, as described in Section~\ref{par:parabolic_field_kodaira}, corresponds exactly to the three possible types of isometries of hyperbolics spaces. 
By theorems of Gromov and Yomdin, $f\in \Aut(X)$ has positive entropy if and only if $f$ is loxodromic: indeed, the translation length of $f^*$ on $\Hyp_X$ is equal to the topological entropy of $f\colon X\to X$ and to the logarithm of the spectral radius $\lambda(f)$ of $f^*$ on $H^{1,1}(X;\R)$ (see~\cite{Cantat:Milnor}). 
By a ping pong argument (see Lemma~\ref{lem:non-elementary_free_groups} below), 
we can add the following characterization to the equivalent conditions~(a)-(c) of~\S~\ref{subs:NE}. 

 \begin{enumerate}[(a)]
\item[(d)] $\Gamma$ contains a non-Abelian free group all of whose elements (distinct from $\id_X$) have positive entropy (i.e. are  loxodromic).
\end{enumerate}

 \subsection{Non-elementary groups of automorphisms on general surfaces} 
 
Let $M$ be a compact manifold. Exactly as in the complex case, let us 
 say that a  group $\Gamma$ of 
homeomorphisms of $M$ is {{non-elementary}} if its 
image $\Gamma^*$ in $\GL(H^*(M;\Z))$ contains a non-Abelian free subgroup.

\begin{lem}\label{lem:non-elementary_free_groups}
Let $M$ be a compact manifold, and   $\Gamma$ be a non-elementary 
subgroup of $\Diff^\infty(M)$. Then $\Gamma$ 
contains a non-Abelian free group $\Gamma_0$ 
such that the topological entropy of every $f\in \Gamma_0\setminus \{ \id\}$ is positive.
\end{lem}

\begin{proof} We split the proof in two steps. The first one concerns groups of matrices,  
the second one is where topological entropy enters into place.

\smallskip

{\bf{Step 1.-}}  {\emph{The image $\Gamma^*$ of $\Gamma$ in $\GL(H^*(M;\Z))$ contains a free subgroup $\Gamma_1^*$, such that every element of $\Gamma_1^*\setminus \{ \id\}$
has spectral radius larger than $1$}}.

\smallskip

The proof uses basic ideas involved in Tits's alternative, here in the simple case of subgroups of $\GL_n(\Z)$.
Let $N$ be the rank of $H^*_{t.\!f\!.}(M;\Z)$, where $t.\!f\!.$ 
stands for ``torsion free''. Fix a basis of this free $\Z$-module. 
Then $\Gamma^*$ determines a subgroup of $\GL_N(\Z)$. Our assumption implies that the 
derived subgroup of $\Gamma^*$ contains a non-Abelian free group $\Gamma_0^*$ of rank~$2$.  

If all (complex) eigenvalues 
of all elements of $\Gamma_0^*$ have modulus $\leq 1$, then by Kronecker's lemma all of them are roots 
of unity. This implies that $\Gamma_0^*$ contains a finite index nilpotent subgroup (see 
Proposition 2.2 and Corollary 2.4 of~\cite{Benoist:Grp_Disc}), contradicting the existence of a non-Abelian free subgroup. 
Thus, there is an element $f^*$ in $\Gamma_0^*$ with a complex eigenvalue of modulus  $\alpha>1$. 
Let $m$ be the number
of eigenvalues of $f^*$ of modulus $\alpha$, counted with multiplicities. Consider the linear representation of $\Gamma_0^*$ 
on $\bigwedge^m H^*(M;\C)$; the action of $f^*$ on this space has a unique dominant eigenvalue, of 
modulus $\alpha^m$; the corresponding eigenline determines an attracting fixed point for $f^*$ in the
projective space $\P(\bigwedge^m H^*(M;\C))$; the action of $f^*$ on the topological space $\P(\bigwedge^m H^*(M;\C))$ is proximal. 
Let 
\begin{equation}
\{0\}=W_0 \subset W_1\subset \cdots \subset W_k \subset W_{k+1}=\bigwedge^m H^*(M;\C)
\end{equation}
be a Jordan-H\"older sequence for the representation of $\Gamma^*$: the subspaces $W_i$ are 
invariant and the induced representation of $\Gamma^*$ on $W_{i+1}/W_i$ is irreducible for all $0\leq i\leq k$. 
Let $V$ be the quotient space $W_{i+1}/W_{i}$ in which the eigenvalue of $f^*$ of modulus $\alpha^m$
appears. Since $\Gamma_0^*$ is contained in the derived subgroup of $\Gamma$, the linear 
transformation of $V$ induced by $f^*$ has determinant $1$; thus, $\dim(V)\geq 2$. Now, 
we can apply Lemma~3.9 of~\cite{Benoist:Grp_Disc} to (a finite index, Zariski connected subgroup of) $\Gamma^*_{0}\rest{V}$:
changing $f$ if necessary, both $f^*\rest{V}$ and $(f^{-1})^*\rest{ V}$ are proximal, and there is 
an element $g^*$ in $\Gamma^*$ that maps the attracting fixed points $a^+_f$ and $a^-_f\in \P(V)$ of $f^*\rest{ V}$ and $(f^*\rest{V})\inv$ to two distinct
points (i.e. $\{ a^+_f, a^-_f\}\cap \{ g^*(a^+_f), g^*(a^-_f)\} =\emptyset$). Then, by the ping-pong lemma, large powers of $f^*$ and $g^*\circ f^*\circ (g^*)^{-1}$ 
generate a non-Abelian free group $\Gamma_1^*\subset \Gamma^*$ such that each element 
$h^*\in \Gamma_1^*\setminus \{\id\}$ has an attracting fixed point in $\P(V)$. This implies that every element of $\Gamma_1^*\setminus \{\id\}$
has an eigenvalue of modulus $>1$ in $H^*(M;\C)$. 

\smallskip

{\bf{Step 2.-}}  Since $\Gamma_1^*$ is free, there is a free subgroup  $\Gamma_1\subset \Gamma$ such that the homomorphism
$\Gamma_1\to \Gamma_1^*$ is an isomorphism. By Yomdin's theorem~\cite{yomdin}, all elements of $\Gamma_1\setminus \{\id\}$ have positive entropy, and we are done. 
\end{proof}

\begin{thm}\label{thm:non_kahler}
Let $M$ be a compact complex surface such that there exists a non-elementary subgroup 
$\Gamma\subset \Aut(M)$. Then $M$ is Kähler. 
\end{thm}

%
 
\begin{proof}
By Lemma~\ref{lem:non-elementary_free_groups}, $M$ admits an automorphism of 
positive topological entropy. It  was shown in~\cite{Cantat:CRAS} that  this property implies that $M$ is K\"ahler. 
\end{proof}

\subsection{Projectivity}\label{par:projectivity_of_the_surface}

\begin{thm}\label{thm:X-is-projective} 
Let $X$ be a compact complex  
surface and $\Gamma$ be a non-elementary 
subgroup of $\Aut(X)$.  Then $X$ is  projective, and is birationally 
equivalent to a rational surface, an Abelian surface, 
a K3 surface, or an Enriques surface. 
\end{thm}

The next sections will provide examples of non-elementary groups of automorphisms for each of these four classes of surfaces. 

It follows from this classification that when $X$ is not rational, there is a canonical volume form preserved by $\Gamma$; moreover, such a form
induces an invariant volume form on $X(\R)$ when  the action is by real automorphisms (see~\cite[Rmk 2.3]{invariant}). 
This constraint has deep consequences on the dynamics of $\Gamma$.

Let us prove Theorem~\ref{thm:X-is-projective}.
If $\Gamma$ is a non-elementary group of automorphisms on $X$, 
Theorem~\ref{thm:non_kahler}  asserts  that $X$ must be Kähler. Then 
the last assertion of the theorem  readily follows from 
the classification given in~\cite[Thm 10.1]{Cantat:Milnor}.  
What remains to show is that a compact Kähler surface admitting a 
non-elementary automorphism group is projective. The proof follows closely the arguments given in~\cite{Cantat:Milnor} and \cite{Reschke:Osaka}.  

\begin{lem}[see also \cite{Reschke:Osaka}, Thm. 2.2]\label{lem:proj_root} 
Let $f$ be a loxodromic automorphism of a compact K\"ahler surface~$X$. 
The following properties are equivalent:
\begin{enumerate}
\item on $H^{2,0}(X;\C)$, $f^*$ acts by multiplication by a root of unity;
\item $X$ is projective.
\end{enumerate}
\end{lem}

\begin{rem}\label{rem:jacobian}
If $X$ supports a loxodromic automorphism, then $\dim H^{2,0}(X, \C)\leq 1$. When this dimension equals 1, that is when  
$X$ is a torus or K3 surface, 
$H^{2,0}(X;\C)$ is generated by   a holomorphic $2$-form $\Omega_X$ 
that does not vanish and satisfies $\int_X\Omega_X\wedge {\overline{\Omega_X}} = 1$. It is unique up
to multiplication by a complex number of modulus $1$. 
So for every $f$, we can write  $f^*\Omega_X=J(f)\Omega_X$, and 
 the Jacobian determinant
\begin{equation} f\in \Aut(X)\mapsto J(f)\in {\mathbb{U}}_1\end{equation} 
defines 
 a unitary character on the group $\Aut(X)$. 
 It follows that the first condition of Lemma~\ref{lem:proj_root} can be reformulated as:
\begin{enumerate}
\item[{\em (1')}] {\emph{either $H^{2,0}(X;\C)=0$ or $J(f)$ is a root of unity}}. 
\end{enumerate}
Moreover, by Kodaira's embedding theorem, $X$ is projective when $H^{2,0}(X;\C)=0$.
\end{rem}

\begin{proof}[Proof of Lemma~\ref{lem:proj_root}]
The characteristic polynomial $\chi_f$ of $f^*\colon H^2_{t.\!f\!.}(X;\Z)\to H^2_{t.\!f\!.}(X;\Z)$ is  a monic polynomial with 
integer coefficients. 

 The subspace $H^{2,0}(X;\C)\subset H^2(X;\C)$ is $\Aut(X)$-invariant. By
the Hodge index theorem, the hermitian form $\Omega\in H^{2,0}(X;\C)\mapsto \int_X\Omega\wedge\overline{\Omega}$ is positive definite and $\Aut(X)$-invariant. 
Thus,  all eigenvalues of $f^*$ on $H^{2,0}(X;\C)$ (resp.\ on $H^{0,2}(X;\C)$) have modulus $1$. 
By the Hodge index theorem, the intersection form on $H^{1,1}(X;\R)$ is non-degenerate, of signature $(1, h^{1,1}(X)-1)$ and $f^*$ acts by isometries upon it; $f$ being loxodromic, $f^*$ has a real eigenvalue $\lambda(f)>1$. It follows  (see \S~2.4 of~\cite{Cantat:Milnor} for instance) that (a) the eigenspace of $f^*$ for the eigenvalue $\lambda(f)$ is an isotropic line $\R\theta^+_f$ in $H^{1,1}(X;\R)$, (b) $\lambda(f)\inv$ is also an eigenvalue of $f^*$ whose eigenspace is an isotropic line $\R\theta^-_f\subset H^{1,1}(X;\R)$, and (c) the intersection form is negative definite on the orthogonal complement of $(\R\theta^+_f+ \R\theta^-_f)^\perp\subset H^{1,1}(X;\R)$. Consequently, besides $\lambda(f)$ and $1/\lambda(f)$, all other roots 
 of  $\chi_f$ have modulus $1$. So 
$\lambda(f)$ is a reciprocal quadratic integer or a Salem number
 (see \S~2.4.3 of~\cite{Cantat:Milnor} for more details) and the decomposition of $\chi_f$ into irreducible factors can be written as
\begin{equation}
\chi_f(t)=S_f(t) \times R_f(t) = S_f(t) \times \prod_{i=1}^m C_{f,i}(t), 
\end{equation}
where $S_f$ is a Salem polynomial or a reciprocal quadratic polynomial, and 
the $C_{f,i}$ are cyclotomic polynomials. 
In particular, if $\xi$ is an eigenvalue  of $f^*$ and a root of unity, we see that 
 $\xi$ is a root of $R_f(t)$ but not of $S_f(t)$.
Note, conversely, that 
 if an eigenvalue of $f^*\rest{H^{2,0}(X;\C)}$ is not a root of unity, then it is a root of $S_f$.

Assume that all eigenvalues of $f^*$ on $H^{2,0}(X;\C)$ are roots of unity. 
Then $\Ker(S_f(f^*))\subset H^2(X;\R)$ is a $f^*$-invariant 
subspace of $H^{1,1}(X;\R)$. This subspace is defined over $\Q$ and
is of Minkowski type; in particular, it contains integral classes of
positive self-intersection. Thus, by the  Kodaira embedding
theorem, $X$ is projective.  

Conversely, assume that $X$ is projective. 
The N\'eron-Severi group $\NS(X;\Q)\subset H^{1,1}(X;\R)$ is $f^*$-invariant and 
contains vectors of positive self-intersection; so, by the description of the linear action of $\Gamma$ given in  \cite[Prop. 2.8]{stiffness}, 
$\NS(X;\R)$ contains all 
isotropic lines associated to loxodromic automorphisms.  
Now, any $f^*$-invariant subspace defined over $\Q$ and containing the eigenspace associated to 
$\lambda(f)$  contains $\Ker(S_f(f^*))$, hence $\Ker(S_f(f^*))\subset \NS(X;\Q)$.  In particular, $\Ker(S_f(f^*))$ does not intersect $H^{2,0}(X;\C)$, which is invariant, 
 and we conclude that all eigenvalues 
 of $f^*$ on  $H^{2,0}(X;\C)$ are roots of unity. 
\end{proof}

\begin{lem}\label{lem:virtually_cyclic_non_projective}
Let $X$ be a  compact K\"ahler surface.  
If $X$ is not projective, then $\Aut(X)^*$ is virtually Abelian and if it contains a loxodromic element it is virtually cyclic.
\end{lem}

\begin{proof}  Let $\Gamma\subset \Aut(X)$ be a free subgroup (possibly reduced to a cyclic group or to $\set{\id}$)
such that every $f\in \Gamma\setminus\set{\id}$ is loxodromic. Since $X$ is not projective, Lemma~\ref{lem:proj_root} and Remark~\ref{rem:jacobian} show that $h^{2,0}(X)=1$ and that $\Gamma$ acts faithfully by scalar multiplication  on $H^{2,0}(X;\C)$. Indeed otherwise the kernel 
of $\Gamma\to \GL(H^{2,0}(X;\C))$ would contain loxodromic elements and $X$ would be projective. 
From this, we deduce that $\Gamma$ has rank at most $1$, hence $\Aut(X)$ is elementary.
To conclude, we refer to Theorem~3.2 of \cite{Cantat:Milnor}, which says that there are two possibilities: either $\Aut(X)^*$ contains a finite index, cyclic subgroup generated by a loxodromic automorphism; or $\Aut(X)^*$ contains a finite index, Abelian subgroup, all of whose non-trivial elements are parabolic (permuting the fibers of a genus $1$ fibration on~$X$). 
\end{proof}

We can now conclude the proof of Theorem~\ref{thm:X-is-projective}: 
indeed we already know that  $X$ is Kähler by Theorem~\ref{thm:non_kahler}, and then it is projective by Lemma~\ref{lem:virtually_cyclic_non_projective}.   
 
\subsection{Remarks on the non-Kähler case}  \label{subs:remarks_non_kahler}
Consider the Hopf surface $X_\alpha$ obtained by taking the 
 quotient of $\C^2\setminus\set{(0,0)}$ by the group of homotheties 
 $(x,y)\mapsto (\alpha^m x, \alpha^my)$, for some $\alpha \in \C^\times$ with $\vert \alpha\vert <1$.
Taking the quotient of $\C^2\setminus\{(0,0)\}$ 
by all homotheties, we get the projective line $\P^1(\C)$: 
this yields a fibration $\eta\colon X_\alpha\to \P^1(\C)$ with fibers isomorphic to the genus $1$ curve $E_\alpha=\C^\times/\bra{\alpha^\Z}$. 

 Recall that a subgroup  $\Gamma\subset\SL_2(\C)$ is 
  {\bf{non-elementary}} if it   contains a non-Abelian free group and is not relatively compact; 
  equivalently, $\Gamma$ induces a non-elementary group of isometries of 
    the hyperbolic space $\Hyp^3$, 
 whose boundary is ${\mathbb{S}}^2\simeq \P^1(\C)$.  By definition the {\bf{limit set}}
 $\Lim(\Gamma)$ is the closure of the set of fixed points of loxodromic elements of $\Gamma$.

Let us fix such a non-elementary group $\Gamma$.  
Let $\nu$ be a probability measure on $\SL_2(\C)$, 
whose support generates $\Gamma$  as a closed semigroup.  
It follows  from Furstenberg's theory of random products of matrices 
that  {$\Lim(\Gamma)$ is the  unique  closed, minimal $\Gamma$-invariant subset of $\P^1(\C)$, 
and there is a unique $\nu$-stationary measure $\mu_{\P^1(\C)}$ on $\P^1(\C)$; moreover, the support of   $\mu_{\P^1(\C)}$
coincides with  $\Lim(\Gamma)$.}  
 
Now,  
consider the action of $\Gamma$ on $X_\alpha$ induced by the natural action of $\Gamma$ on $\C^2$. 
The fibration $\eta\colon X_\alpha\to \P^1(\C)$ is $\Gamma$-equivariant.  Any  $f\in \Gamma$
 acts by scalar multiplication along the fibers of $\C^2\setminus\{(0,0)\}\to \P^1(\C)$. 
 Since the  multiplication $z\mapsto \beta z$ induces a translation on the elliptic curve $E_\alpha$, the action of $\Gamma$ on $X_\alpha$ is an isometric extension of the action on $\P^1(\C)$, so we are in the setting of~\cite{Guivarch-Raugi:2007}. From this  we obtain: \textit{if $\Gamma$ is Zariski dense in $\SL_2(\C)$, viewed as a real Lie group, then $X_\alpha$ supports a unique minimal invariant subset $\Lambda_X$ and a unique $\nu$-stationary measure $\mu_X$; this measure is not $\Gamma$-invariant; both $\Lambda_X$ and $\mu_X$ are invariant under the action of $\C^\times$ on $X_\alpha$ by homotheties; and the marginal $\eta_*\mu_X$ of $\mu_X$ is equal to $\mu_{\P^1(\C)}$}(\footnote{With the notation of~\cite{Guivarch-Raugi:2007}, the $KAN^+$ decomposition of $\SL_2(\C)$ can be chosen in such a way that $A=\R_+^\times$ is the group of diagonal matrices with real positive coefficients, and its centralizer $M$ in the maximal compact subgroup $K$ is ${\mathbb{S}}^1$, the group of diagonal matrices with eigenvalues of modulus $1$. Then, the product $MA$ is just $\C^\times$, the group of complex diagonal matrices  in $\SL_2(\C)$. In particular, this group is connected (it corresponds to the group $C$ in~\cite{Guivarch-Raugi:2007}). Thus, our assertions follows from the main theorems stated on the second page of~\cite{Guivarch-Raugi:2007}.}).

Now, take $\alpha\in \R_+^\times$ and assume that $\Gamma_\nu$ is a non-elementary subgroup of $\SL_2(\R)$ (in particular, $\Gamma_\nu$ is not Zariski dense in the real Lie group $\SL_2(\C)$). The limit set $\Lim(\Gamma)$ is contained in $\P^1(\R)$; thus, 
the $\nu$-stationary measures on $X_\alpha$ are supported in the preimage of $\P^1(\R)$ by the fibration $\eta$. There is, in that case, a one parameter family of such measures, parametrized by an angle $\theta\in \R/\Z$. Indeed, the plane $\R^2\subset \C^2$ is $\Gamma$-invariant, and the projection of $\R^2\setminus\{(0,0)\}$ in $X_\alpha$ supports a unique $\nu$-stationary measure $\mu_{X(\R)}$. 
Then, all the   stationary measures are obtained  by ``rotating''  $\mu_{X(\R)}$ by 
$(x,y)\mapsto \exp(2i \pi \theta)\cdot (x, y)$.

\begin{rem} Let $X$ be a Hopf or a Inoue surface. According to the description of $\Aut(X)$ by 
 Namba, Matumoto and Nakagawa (see~\cite{Cantat-PR-Xie, Namba:Tohoku}),    either $\Aut(X)$ is virtually solvable, or $X$ is a Hopf surface obtained as a quotient of its universal cover 
 $\C^2\setminus\{(0,0)\}$ by a group of homotheties 
\begin{equation}
\Phi_{m,n}(x,y)=(\alpha^m\xi^nx,\alpha^m\xi^ny)
\end{equation}
where $\alpha\in \C^\times$ has modulus $<1$, $\xi$ is a root of unity (of order $q$ for some $q\geq 1$), and $(m,n)\in \Z\times \Z/q\Z$. 
Taking a finite cover brings us back to the case $\xi=1$, which we just described. 
\end{rem}

\begin{rem}
If $X$ is a Kodaira fibration, then $X$ comes with an $\Aut(X)$-invariant fibration on an elliptic curve with elliptic fibers. In this case every stationary measure is invariant. This follows for instance from the fact that such actions are distal (see~\cite[Thm. 3.5]{furstenberg_stiffness}). 
\end{rem}

These remarks do not exhaust all possible non-Kähler compact complex surfaces. Indeed, there are other examples of surfaces in class VII. Moreover, the classification of VII$_0$ surfaces is not complete yet (even, as far as we know, if we assume that $\Aut(X)$ is not virtually nilpotent) .  
 
\begin{que} Let $X$ be a non-Kähler compact complex surface. Suppose there is a probability measure $\nu$ on $\Aut(X)$ such that $X$ supports a $\nu$-stationary measure which is not $\Gamma_\nu$-invariant and has a Zariski dense support. Then must  $X$ be a Hopf surface?
\end{que}

\section{From Enriques to rational surfaces}

In this Section  we start  by describing 
two families of surfaces with non-elementary automorphism groups, obtained by taking quotients of K3 surfaces by an involution. For Enriques surfaces, the involution is fixed point free; 
this is not the case for Coble surfaces. We briefly mention the  examples of Blanc, 
whose detailed study is the purpose of Part~\ref{part:ergodic_blanc}, and conclude 
 with an example of
  Lesieutre of a non-elementary group on a rational surface without invariant curve. 
Coble, Blanc, and Lesieutre surfaces are all rational, but their dynamical features 
happen to be  quite different: this stems from the existence or 
non-existence of a global invariant measure. 

\subsection{Enriques (see~\cite{Cossec-Dolgachev:book, Dolgachev:Kyoto})}\label{par:Enriques}
Recall that Enriques surfaces are quotients of K3 surfaces by fixed point free involutions. 
According to Horikawa and Kond\={o} (\cite{Horikawa:Enriques1,Horikawa:Enriques2, Kondo:1994}),
the moduli space ${\mathcal{M}}_E$ of complex Enriques surfaces is a rational quasi-projective variety of dimension 10.
An Enriques surface $X$ is nodal if it contains a smooth rational curve; such rational curves have self-intersection
$-2$, and are called nodal-curves or $(-2)$-curves. Nodal Enriques surfaces form a set of codimension $1$ in ${\mathcal{M}}_E$.

For any Enriques surface $X$, the lattice $(\NS(X;\Z), q_X)$ is isomorphic to the orthogonal direct sum $E_{10}=U \operp  E_8(-1)$ 
(\footnote{Here, $U$ is the standard $2$-dimensional Minkowski lattice, $(\Z^2, x_1x_2)$, and $E_8$ is the root lattice given by the 
corresponding Dynkin diagram; so $E_8(-1)$ is negative definite, and $E_{10}$ has signature $(1,9)$ (see \cite[Chap. II]{Cossec-Dolgachev:book}). Also, in this  paper
$\NS(X;\Z)$ denotes the 
 torsion free part of the N\'eron-Severi group, which is sometimes  denoted by ${\mathsf{Num}}(X;\Z)$ in the literature on 
Enriques surfaces.}). 
Let $W_X\subset \O(\NS(X;\Z))$ be the subgroup generated by reflexions about classes $u$ such that  $u^2=-2$, 
 and   $W_X(2)$ be the subgroup of $W_X$  acting trivially on $\NS(X;\Z)$ modulo $2$. 
Both $W_X$ and $W_X(2)$ have  finite index in $\O(\NS(X;\Z))$. 
The following result is due independently to Nikulin and to Barth and Peters (see \cite{Dolgachev:Kyoto} for details and references).

\begin{thm} 
If $X$ is an  Enriques surface which is not nodal, 
\begin{enumerate}[(1)]
\item the homomorphism 
$ \Aut(X)\ni f \mapsto f^*\in \GL(H^2(X, \Z))$ is injective, 
\item its image  satisfies
$W_X(2) \subset \Aut(X)^* \subset W_X$.
\end{enumerate}
In particular, $\Aut(X)^*$ is a finite index subgroup in $\O(\NS(X;\Z))$, thus $\Aut(X)^*$ is a lattice in the rank $1$ Lie group $\O(\NS(X;\R))\simeq \O_{1,9}(\R)$ and it acts irreducibly on 
$\NS(X;\R)$. \end{thm}

This implies that $\Aut(X)$ is non-elementary, contains parabolic automorphisms, and  does not preserve any curve (because $\Aut(X)$ does not have a fixed point in $H^2(X;\R)$). From these properties, and from~\cite{invariant}, we obtain a classification of $\Aut(X)$-invariant probabilitity measures on unnodal complex Enriques surfaces; from~\cite{finite_orbits}, we know that $\Aut(X)$ has only finitely many finite orbits. On the other hand, we expect that generically 
all orbits are   infinite:

\begin{que}\label{que:Enriques-finite-orbits} Is it true that if $X\in {\mathcal{M}}_E$ is a general (resp. very general) Enriques surface, then $\Aut(X)$ does not have any finite orbit?
\end{que}

If the answer is positive, then one could also apply the main results of~\cite{hyperbolic} to describe the distribution of random orbits on (real, unnodal) Enriques surfaces.

\subsection{Coble (see~\cite{Cantat-Dolgachev})}
In this article, a {\bf{Coble surface}} is, by definition, obtained by blowing up
the ten nodes of a general rational sextic curve $C_0\subset \P^2$. The result is a rational surface 
$X$ with a large group of automorphisms. To be precise, denote by $K_X$ 
 the  {{canonical line bundle}} $K_X=\det(TX^\vee)$ and consider the 
 corresponding  {\emph{canonical class}} 
$k_X\in \NS(X;\Z)$; since $K_{\P^2}$ is $O_{\P^2}(-3)$ and 
since we blow up the nodes of $C_0$, we get 
\begin{equation}
-2k_X=[C] \quad {\text{and}} \quad k_X=-3\bfe_0+\sum_{i=1}^{10}\bfe(p_i)
\end{equation}
where $C$ is the strict transform of $C_0$ in $X$, and $(\bfe_0,\bfe(p_1), \ldots, \bfe(p_{10}))$ is the basis of $\NS(X)$ given by the classes of the total transform of a line and the exceptional divisors $E_i$ obtained by blowing up the ten double points $p_i$ of $C_0$.
The orthogonal complement $k_X^\perp$ is a lattice of dimension $10$, isomorphic
to $E_{10}$, and we define $W_X(2)$ exactly in the same way as for Enriques surfaces. 
Then, $\Aut(X)^*$ preserves the decomposition $k_X\oplus k_X^\perp$. 
As before we say that $X$ is unnodal when it does not contain any smooth rational curve of self-intersection $-2$. If  $X$ is unnodal, the 
 $\Aut(X)^*$ contains $W_X(2)$, in particular it is non-elementary
     (see~\cite{Cantat-Dolgachev}, Theorem~3.5). 

Let us explain this result in a more explicit way (see~\cite{Cantat-Dolgachev}). Let $C_0$ be 
a general rational sextic as before. 
Choose one of the double points of $C_0$, say $p_i$, and consider the cubic curve $C_i$ containing the remaining $9$ points. Then, $C_0$ and $2C_i$ 
generate a pencil of sextic curves, with base points at the $p_j$, $j\neq i$. Blowing up these $9$ points, we obtain a surface $X_i$ with a genus $1$ fibration $\pi_i\colon X_i\to \P^1_\C$; the sextic $C_0$ lifts
to  a nodal fiber of $\pi_i$ and $2C_i$ gives a smooth multiple fiber. The $9$ exceptional divisors $E_j$, $j\neq i$, determine $9$ multisections of $\pi_i$, each of degree $2$. Each pair $(j,k)$, $j,k\neq i$, defines an automorphism $g_{jk}$ of $X_i$, acting by translation along the fibers of $\pi_i$: the translation on the fiber $X_{i,b}:=\pi_i^{-1}(b)$ is by the divisor of degree $0$ defined by 
\begin{equation}
\tau_{jk}:=(E_k-E_j)_{X_{i,b}}\in \Pic^0(X_{i,b}).
\end{equation} 
For a general choice of $C_0$, the $g_{jk}$ generate a free Abelian group $A_i$ of rank $8$; its elements, except the identity, are parabolic automorphisms of $X_i$ preserving the fibration $\pi_i$ (hence also the singular point $p_i$ of the strict transform of $C_0$ in  $X_i$). Thus, blowing-up $p_i$, the group $A_i$ lifts to a subgroup of $\Aut(X)$. In this way we obtain 
 $10$ copies $A_i$ of $\Z^8$ in $\Aut(X)$, acting as parabolic groups with respect to distinct genus $1$ fibrations. In particular, $\Aut(X)$ is non-elementary. Note that the strict transform of $C_0$ is $\Aut(X)$-invariant and is contained in a fiber of $\pi_i$ for each index $i$.  

Also, Coble surfaces are degenerations of Enriques surfaces; thus, Coble surfaces share many features of Enriques surfaces, but there are also interesting differences. For instance, $\Aut(X)$ preserves the class $k_X$, and this class is non-trivial. Moreover, there is a holomorphic  section of $-2K_X$
vanishing exactly along the strict transform $C\subset X$ of the rational sextic curve $C_0$; this means that there is
a meromorphic section $\Omega_X=\xi(x,y) (dx\wedge dy)^2$ of $2K_X$ that does not vanish and has
a simple pole along $C$. Thus, the formula
\begin{equation}
\vol_X(U)=\int_U
 \abs{\xi(x,y)} dx\wedge dy\wedge d{\overline{x}}\wedge d{\overline{y}} =\int_U  \abs{\xi(x,y)} ({\mathsf{i}}dx\wedge  d{\overline{x}}) \wedge ({\mathsf{i}} dy\wedge d{\overline{y}})
\end{equation}
determines a measure  $\vol_X$ $={\text{``}}\,\Omega_X^{1/2}\wedge {\overline{\Omega_X^{1/2}}}\,{\text{''}}$. The total mass of this measure is finite.  Indeed, if locally $C=\{ x=0\}$  then $\xi(x,y)=\eta(x,y)/x$ where $\eta$ is regular; thus, $\abs{\xi} =\abs{\eta} \abs{x}^{-1}$ is locally integrable because $\frac{1}{r^\alpha}$ is integrable with respect to $rdrd\theta$ when $\alpha<2$. 
We may assume  that it is   a probability after multiplying $\Omega_X$ by some
adequate constant, and 
this measure is $\Aut(X)$-invariant, because $\vol_X$ is uniquely determined by the complex structure (see also Remark~\ref{rem:pm1}  below). In particular the ergodic theory of 
Coble examples can be studied with the techniques of~\cite{stiffness}.


\subsection{Blanc} \label{subs:blanc} 
Another family of examples was introduced by Blanc in~\cite{Blanc:Michigan}. Here we 
  describe them  informally, a more detailed presentation will be  given in Part~\ref{part:ergodic_blanc}. 

Start with a smooth cubic curve $C\subset \P^2$. 
If $q$ is a point of $C$, the Jonquières involution associated to $(C, q)$ is 
the birational involution $\sigma_q$ of $\P^2$ characterized by the following properties: it fixes $C$ pointwise
and it preserves the pencil of lines through $q$. The indeterminacy points of $\sigma_q$ are $q$ and the four tangency 
points of $C$ with this pencil, which may be ``infinitely near'' $q$.
Thus, the indeterminacies of $\sigma_q$ are resolved by blowing-up points of $C$, possibly several times. 
After such a sequence of blow-ups $\sigma_q$ 
becomes  an automorphism of a rational surface   that fixes pointwise 
the strict transform of $C$. In particular,  if we blow-up further points of this
curve, $\sigma_q$ lifts to an automorphism of the new surface. 

Pick a finite number of points $q_i\in C_0$, $i=1, \ldots, k$, and 
resolve simultaneously the indeterminacies of the Jonquières 
involutions $\sigma_i:=\sigma_{q_i}$ determined by the $q_i$. The result is a rational surface $X$, together
with a subgroup $\Gamma:=\langle s_1, \ldots, s_k\rangle$ of $\Aut(X)$.  
Blanc proves in~\cite{Blanc:Michigan, Blanc:Indiana} that: 
\emph{
\begin{enumerate}
\item there are no relations between these involutions, 
that is,  $\Gamma$ is a free product 
$$\langle s_1, \ldots, s_k\rangle \simeq \bigast_{i=1}^{k} \Z/2\Z;$$
\item  the composition of two distinct involutions $s_i\circ s_j$ is parabolic;
\item  the composition of three distinct involutions is loxodromic. 
\end{enumerate}}
In addition, there is a meromorphic section 
$\Omega_X$ of $K_X$ with a simple pole along the strict transform of $C_0$, however, 
contrary to Coble surfaces, 
the form $\vol_X:=\Omega_X\wedge{\overline{\Omega_X}}$ is  not integrable (its total mass is infinite). This observation will play a crucial role in Part~\ref{part:ergodic_blanc}.

\begin{rem}\label{rem:pm1}
If $\Gamma\subset \Aut(X)$ is generated by involutions and there is a meromorphic form $\Omega$ such that $f^*\Omega=\xi(f)\Omega$ 
for every $f\in \Gamma$, then $\xi(f)=\pm 1$: 
this is the case for Blanc's examples or general Coble surfaces, since $W_X(2)$ is also 
generated by involutions (see \cite{Dolgachev:Kyoto}).  
\end{rem}

\subsection{Lesieutre}\label{subs:lesieutre}
 In~\cite{lesieutre:tricoble}, Lesieutre constructs rational surfaces $X$ with the following properties: {\emph{$\Aut(X)$ contains three involutions $\tau_i$, $i=1,2,3$, such that the group $\Gamma:=\langle \tau_1, \tau_2, \tau_3\rangle\subset \Aut(X)$ and $f:=\tau_1\circ\tau_2\circ\tau_3$ satisfy 
\begin{enumerate}
\item $X$ and $\Gamma$ are defined over $\Q$;
\item $\Gamma$ is non-elementary, and isomorphic to $(\Z/2\Z)\star(\Z/2\Z)\star(\Z/2\Z)$; 
\item $\Gamma$ does not contain any parabolic element;
\item $f$ is loxodromic and does not preserve any curve.
\end{enumerate} 
In particular, $\Aut(X)$ does not preserve any rational section of $K_X$ (moreover, $-K_X$ is not pseudo-effective).}}

Thus, the situation is quite different from Enriques, Coble and Blanc surfaces, since there is neither parabolic automorphism nor invariant algebraic volume form. Our results are so far not powerful enough to describe the finite orbits or stationary measures for such an example.

\newpage
\part{Pentagon folding and dynamics on K3 surfaces} \label{part:pentagons}

\section{The space of pentagons and the folding groups} \label{sec:pentagons_geometry}

The automorphism groups  of Wehler surfaces were discussed at length in our previous papers. 
Here we describe another family of K3 surfaces with a 
non-elementary group action, coming from the geometry of pentagons in the euclidean plane.

\begin{rem}  
The ``folding'' terminology is borrowed from \cite{benoist-hulin, benoist-hulin2}, which was a source of motivation for these results. There is a vast amount of literature on length-preserving transformations 
in spaces of polygons, notably motivated by algorithmic questions: see for instance Section 5.3.2 in the monograph \cite{Demaine-ORourke}  where our folding transformations are referred to as ``flips''. 
\end{rem}

\subsection{Spaces of plane pentagons}\label{par:pentagons}
Let $\ell = (\ell_0, \ldots, \ell_4)\in \R_{>0}^5$ be a  5-tuple of  positive
 real numbers such that there exists a pentagon with side lengths $\ell_i$; 
 this imposes a  a condition on the $\ell_i$, defined by explicit inequalities, and we say 
 that $\ell$ is {\bf{admissible}} if this condition is satisfied.
 Here a pentagon 
is just an ordered set of points $(a_i)_{i=0, \ldots, 4}$ in the Euclidean plane $\R^2$,
such that $\dist(a_i,a_{i+1})=\ell_i$ for $i=0,\ldots, 4$ (with $a_5 = a_0$ by definition, in other words we consider indices modulo 5); pentagons are not assumed to be convex, and 
two distincts sides $[a_i, a_{i+1}]$ and $[a_j, a_{j+1}]$ may 
intersect at  a point which is not one of the $a_i$'s.

Let  $\Pent(\ell)$ be the set of pentagons with side lengths $(\ell_i)_{i=0}^4$. Note that 
$\Pent(\ell)$ is naturally a real  algebraic variety, defined by polynomial equations of the form 
$\dist(a_i, a_{i+1})^2 = \ell_i^2$. 

For every  $i$,  $a_i$ is one of the two intersection points $\set{a_i, a'_i}$ of the circles
of respective  centers $a_{i-1}$ and $a_{i+1}$ and   radii $\ell_{i-1}$ and $\ell_i$. 
The transformation exchanging  these two points $a_i$ and $a_i’$,  while keeping 
 the other vertices fixed, defines an involution   of $\Pent(\ell)$, that we denote by 
 $s_{i-1}$ (this choice for the index  will be  convenient later). 
 Geometrically, it corresponds to folding (or reflecting) 
 the pentagon along the diagonal $(a_{i-1} a_{i+1})$. 
 It commutes with the action of the group $\SO_2(\R)\ltimes \R^2$ of positive isometries of the plane, 
  hence, it induces an involution $\sigma_{i-1}$ on the quotient space 
\begin{equation}
\Pent^0(\ell)=\Pent(\ell)/(\SO_2(\R)\ltimes \R^2).
\end{equation}
Each element of $\Pent^0(\ell)$ admits a unique representative with 
$a_0=(0,0)$ and $a_1=(\ell_0, 0)$, so as before  
  $\Pent^0(\ell)$ is a real 
algebraic variety, which is easily seen to be of dimension 2 
(see \cite{Curtis-Steiner, Shimamoto-Vanderwaart}). 
We will see below that, when  it is smooth,  it is a real K3 surface. 
The five involutions $\sigma_i$ act by algebraic diffeomorphisms on this surface,  and for a general choice of lengths, the group generated by these involutions 
 is non-elementary.

\begin{rem}
If we consider quadrilateral instead of pentagons,
 the corresponding space 
 \begin{equation} 
 {\text{Quad}}^0(\ell_0, \ell_1,\ell_2, \ell_3)
 \end{equation}
 is a curve of genus $1$ and the involutions typically 
 generate an infinite dihedral group.  The corresponding dynamical system, both on the space of 
quadrilaterals and the space of quadrilaterals modulo isometries, 
was studied in depth in~\cite{esch-rogers, benoist-hulin, benoist-hulin2}. With $n$-gons, $n\geq 4$, one would get Calabi-Yau manifolds of dimension $n-3$
({\footnote{This follows from computations which are similar to the ones used to prove Lemma~\ref{lem:pentagon_smoothness} and Lemma~\ref{lem:pentagon-volume-form} below, the difference being that the singularities of $X$ are not isolated when $n\geq 6$.}}).
\end{rem}
 
\subsection{Algebraic geometry of $\Pent^0(\ell)$} To analyze the algebraic structure and geometry of  $\Pent^0(\ell)$,
we view a plane pentagon with side lengths $\ell_0, \ldots, \ell_4$ modulo translations  
 as the data of a 5-tuple of vectors $(v_i)_{i=0, \ldots , 4}$ in $\R^2$ (identified with $\C$)
of respective length $\ell_i$   such that $\sum_i v_i = 0$. 
Write $v_i   = \ell_i t_i$ with $\abs{t_i} = 1$. Then the action of $\SO_2(\R)$ can be identified to the diagonal multiplicative action of $\mathbb{U}=\{\alpha \in \C\; ; \; \abs{\alpha}=1\}$ on the $t_i$:
\begin{equation}
 \alpha\cdot (t_0, \ldots, t_4)=(\alpha t_0, \ldots \alpha t_4).
\end{equation} 
Now, following Darboux \cite{darboux}, we consider the surface $X$ in $\P^4_\C$  defined by the equations
\begin{equation}\label{eq:equation_pentagons}
\begin{cases} 
\ell_0z_0+ \ell_1 z_1+\ell_2z_2+\ell_3z_3+\ell_4 z_4 = 0\\
\ell_0/z_0+ \ell_1/ z_1+\ell_2/z_2+\ell_3/z_3+\ell_4/ z_4 = 0
\end{cases}
\end{equation}
where $[z_0:\ldots :z_4]$ is some fixed choice of homogeneous coordinates, and the second 
equation must be multiplied by $z_0z_1z_2z_3z_4$ to obtain a homogeneous equation of degree $4$.

\begin{rem}\label{rem:hessian} This surface is isomorphic to the Hessian of a cubic surface (see~\cite[\S 9]{Dolgachev:CAG}). More precisely, consider a cubic surface
$S\subset \P^3_\C$ whose equation  $F$ can be written in Sylvester's pentahedral form, that is, 
as a sum $F=\sum_{i=0}^4 \lambda_i F_i^3$ for some complex numbers $\lambda_i$ and linear forms $F_i$ with $\sum_{i=0}^4 F_i=0$.
By definition, its Hessian surface $H_F$ is defined by $\det(\partial_i\partial_j F)=0$. Then, using the linear forms
$F_i$ to embed $H_F$ in $\P^4_\C$,  we obtain the surface defined by the pair of
equations $\sum_{i=0}^4 z_i=0$ and $\sum_{i=0}^4\frac{1}{\lambda_iz_i}=0$. Thus, $H_F$ is our surface $X$, 
for $\ell_i^2=\lambda_i^{-1}$. 
We refer to \cite{Dolgachev-Keum, Dardanelli-vanGeemen, Dolgachev:Salem, Rosenberg} for an introduction to 
these surfaces and their birational transformations.\end{rem}

For completeness, let us directly prove some of its basic properties.  First, we introduce the ten points $q_{ij}$
determined by the system of equations 
\begin{equation}
\ell_i z_i+\ell_j z_j=0, \quad z_k=z_l=z_m=0
\end{equation}
with $i<j$ and $\{i,j,k,l,m\}=\{0,1,2,3,4\}$. 

\begin{lem}\label{lem:pentagon_smoothness}
Let $\ell=(\ell_0, \ldots, \ell_4)$ be an element of $(\C^*)^5$.
The surface $X_\ell \subset \P^4_\C$ defined by the system~\eqref{eq:equation_pentagons} has $10$ singularities 
at the points $q_{ij}$. 
In the complement of these ten isolated singularities, $X_\ell$ is smooth if and only if 
\begin{equation}\label{eq:pentagon_smoothness}
\sum_{i=0}^4 \e_i \ell_i \neq 0\quad \forall \e_i \in \set{\pm 1}.
\end{equation} 
\end{lem}

 Note that for positive $\ell_i$'s, violation of condition~\eqref{eq:pentagon_smoothness} 
 means that there exist a degenerate pentagon with lengths $\ell_i$. 
 
 \noindent {\bf{Notation.-- }} {\emph{We shall use the notation $X$ instead of $X_\ell$ when the dependence on $\ell$ is not crucial.}} 
 
\begin{proof}
We first look for singularities in the complement of the hyperplanes $z_i=0$, and work in the chart $z_0 = 1$. Then, if we substitute $z_4 = - (\ell_0+ \ell_1z_1+ \ell_2z_2 + \ell_3z_4)/\ell_4$  in the
second equation of \eqref{eq:equation_pentagons}, we obtain an affine equation of $X$ in the chart $z_0 = 1$, namely:
\begin{equation}\label{eq:equation_pentagons_affine}
\frac{\ell_1}{z_1} + \frac{\ell_2}{z_2} + \frac{\ell_3}{z_3} - \frac{\ell_4^2}{\ell_0+ \ell_1z_1+ \ell_2z_2+ \ell_3z_3} + \ell_0 = 0. 
\end{equation}
The singularities are determined by  the system of equations
$z_1^2 =z_2^2=z_3^2 = \ell_4^{-2}(\ell_0+ \ell_1z_1+ \ell_2z_2 + \ell_3z_3)^2$. So, by symmetry, at a singularity where none of the coordinates vanishes we must have 
$z_i=\varepsilon _i z$ for some $\varepsilon _i=\pm 1$ and a common factor $z\neq 0$; this is precisely Condition~\eqref{eq:pentagon_smoothness}.

Looking for singularities with one coordinate equal to $0$, say $z_1=0$ in the chart $z_0=1$, we obtain the system 
of equations 
\begin{equation}\label{eq:equation_pentagons_singularities}
\begin{cases} 
0= (\ell_0 z_2z_3+\ell_3 z_2+\ell_2z_3)(\ell_0 +\ell_2 z_2+\ell_3 z_3)+(\ell_1^2-\ell_4^2)z_2z_3\\
0=\ell_1 z_3(\ell_0+2\ell_2 z_2+\ell_3z_3)\\
0= \ell_1 z_2 (\ell_0+\ell_2z_2+2\ell_3 z_3)
\end{cases}
\end{equation}
together with $\ell_0+\ell_2 z_2+\ell_3z_3+\ell_4z_4=0$ and $\ell_1 z_2z_3z_4=0$ (in particular, $z_2$, $z_3$ or $z_4$
must vanish).
The solutions of this system are given by $z_1=z_2=z_3=0$, which gives the point $q_{04}=[\ell_4:0:0:0:-\ell_0]$,
or $z_1=z_2=0$ and $\ell_0+\ell_3z_3=0$, which corresponds to $q_{03}=[\ell_3:0:0:-\ell_0:0]$, or $z_1=z_3=0$
which gives $q_{02}$, or $z_1=z_4=0$ but then either $z_2=0$ or $z_3=0$ and we end up again with $q_{02}$ and $q_{03}$.
The result follows by symmetry.
\end{proof}
 
\begin{lem}\label{lem:pentagon-volume-form}
If $\ell\in (\C^*)^5$ satisfies Condition~\eqref{eq:pentagon_smoothness}, then the ten singularities are simple nodes (Morse singularities) and 
the surface $X$ is a (singular) K3 surface: a minimal resolution $\hat X$  of $X$ is a K3 surface, which is  
 obtained by blowing-up its ten nodes, thereby creating ten rational $(-2)$-curves.
\end{lem}

\begin{proof}
Working in the chart $z_0=1$ and replacing $z_4$ by $-(\ell_0+\ell_1z_1+\ell_2z_2+\ell_3z_3)/\ell_4$, 
the quadratic term 
of the equation of $X$ at the singularity $(z_1,z_2,z_3)=(0,0,0)$ is $(-\ell_0/\ell_4)Q$, where 
\begin{equation}\label{eq:pentagons_quadratic_Q}
Q(z_1,z_2,z_3)=\ell_1z_2z_3+\ell_2z_1z_3+\ell_3z_1z_2
\end{equation}
is a non-degenerate quadratic form (its determinant is $2\ell_1\ell_2\ell_3\neq 0$). So locally $X$ is holomorphically 
equivalent to the quadratic cone $\{Q=0\}$, hence to a quotient singularity $(\C^2,0)/\eta$ with $\eta(x,y)=(-x,-y)$. 
The minimal resolution of such a singularity is obtained by a simple blow-up of the ambient space, the exceptional 
divisor being a $(-2)$-curve in the smooth surface~$\hat{X}$.
The adjunction formula shows that there is a holomorphic $2$-form $\Omega_X$ on the regular part of $X$; locally, $\Omega_X$
lifts to an $\eta$-invariant form $\Omega_X'$ on $\C^2\setminus \{ 0\}$, which by 
the Hartogs theorem extends across  the origin
to a non-vanishing $2$-form. To recover $\hat{X}$, one can first blow-up $\C^2$ at the origin and then take the quotient 
by (the lift of) $\eta$: a simple calculation shows that $\Omega_X'$ determines a non-vanishing $2$-form on $\hat{X}$.
After such a surgery is done at the ten nodes, $\hat{X}$ is a smooth surface with a non-vanishing section of $K_{\hat{X}}$; 
since it contains at least ten rational curves, it cannot be an Abelian surface, so it must be a K3 surface. 
\end{proof}

\begin{rem}\label{rem:line}
Let $L_{ij}$ be the line defined by the  equations 
$z_i=0$, $z_j=0$, $\ell_0z_0+\cdots +\ell_4z_4=0$; each of these ten lines is contained in $X$, each of them contains
$3$ singularities of $X$ (namely $q_{kl}$, $q_{lm}$, $q_{km}$ with obvious notations), and each singularity is contained in 
three of these lines. If one projects them on a plane, the ten lines $L_{ij}$ form a Desargues configuration (see~\cite{Dolgachev:Salem, Dolgachev-Keum}).
\end{rem}

\subsection{The real part} All this works for any choice of complex numbers $\ell_i\neq 0$. When the $\ell_i$ are real, $X$ is endowed with two real structures.
First, one can consider the complex conjugation 
\begin{equation}
c\colon [z_i]\mapsto [\overline{z_i}]
\end{equation} 
on $\P^4(\C)$ and restrict it to $X$: this gives a first
antiholomorphic involution $c_X$. Another one is 
\begin{equation}\label{eq:sX}
s_X\colon [z_i]\mapsto [1/\overline{z_i}].
\end{equation}
To be more precise, consider first the quartic birational involution $J\in \Bir(\P^4_\C)$ defined by $J ([z_i]) =[1/z_i]$; 
$J$ preserves $X$, it determines a birational transformation $J_X\in \Bir(X)$, and on ${\hat{X}}$ it becomes an automorphism because every birational transformation of a K3 surface is regular. Moreover, $J$ commutes with $c$. 
Thus, $s_X=J_X\circ c_X$ determines a second antiholomorphic involution $s_{{\hat{X}}}$ of ${\hat{X}}$. In what follows, we denote by $(X,s_X)$ this real structure (even if it would be better to study it on ${\hat{X}}$); 
its real part is the fixed point set of $s_X$, i.e. the set of points in $X(\C)$ with coordinates of modulus $1$: the real part does not contain any of the singularities of $X$, this is why we prefer to stay in $X$ rather than lift everything to ${\hat{X}}$. In conclusion,  if $(\ell_i)\in (\R_+^*)^5$, then
{\emph{with the 
real structure defined by $s_X$, the real part of $X$  coincides with $\mathrm{Pent}^0(\ell_0, \ldots , \ell_4)$}}.

\begin{rem}\label{rem:topology}
When $\ell_i>0$ for every $i$,
a complete description of the possible homeomorphism types for  the real locus (in the smooth and singular cases) is given in~\cite{Curtis-Steiner}: 
{\emph{in the smooth case, it is an  orientable surface of genus $g = 0, \ldots, 4$ or  
the disjoint union of two tori; 
if one includes singular surfaces, one gets a total of $19$ topological types.}} In particular if $\mathrm{Pent}^0(\ell)$ is disconnected, it is the disjoint union of two tori. 
The space of possible side lengths can be tesselated in cells corresponding to smooth surfaces $\mathrm{Pent}^0(\ell)$, with walls corresponding to singular surfaces. Cells are encoded by a $2\times 2$  ``code-matrix'' in~\cite[Table 4]{Curtis-Steiner}. With this viewpoint, the disconnected surfaces correspond to exactly three cells (see Figure~\ref{fig:pentagons_disconnected} below).
\end{rem} 

\begin{rem}\label{rem:pentagon_enriques}
The involution $J$ preserves $X$ and the two real structures $(X,c_X)$ and $(X,s_X)$. 
It lifts to a fixed point free involution ${\hat{J}}_X$ on $\hat{X}$, and $\hat{X}/{\hat{J}}_X$ is an Enriques surface. 
On pentagons, $J$ corresponds to the symmetry $(x,y)\in \R^2\mapsto (x,-y)$ that reverses orientation. Thus we see that 
the space of pentagons modulo affine isometries is an Enriques surface. 
When $X$  acquires an eleventh singularity which is fixed by $J_X$, then $\hat{X}/{\hat{J}}_X$ becomes a Coble surface: 
see~\cite[\S 5]{Dolgachev:Salem}
for nice explicit examples. This happens 
 for instance when all lengths are $1$, except one which is equal to  $2$ 
 (this corresponds to $t=1/4$ in~\cite[\S 5.2]{Dolgachev:Salem}).
 \end{rem}

\subsection{Involution and the folding groups}\label{subs:involutions}
Let us express the folding transformations  
 in the coordinates $(\ell_it_i)$, where $t_i$ is a complex number with  $\vert t_i\vert=1$.
Given $i\neq j$ in $\set{0, \ldots, 4}$ (consecutive or not) we define an involution $(t_i, t_j)\mapsto (t_i', t_j')$ preserving the vector $\ell_i t_i+ \ell_j t_j$
by taking the  symmetric of  $t_i$ and $t_j$ 
with respect to the line directed by $\ell_i t_i+ \ell_j t_j$. In coordinates, 
$t'_k  = u / t_k$ for some $u$ of modulus 1, and equating   $\ell_i t_i+ \ell_j t_j =  \ell_i t'_i+ \ell_j t'_j$ we obtain  
\begin{equation}\label{eq:sigmaij}
(t_i',t_j') = \lrpar{\frac{u}{t_i}, \frac{u}{t_j}} \text{, with }  u  = \frac{\ell_i t_i + \ell_j t_j}{\ell_i t_i\inv + \ell_j t_j\inv}.
\end{equation}
Observe  that these computations also make sense when the $\ell_i$ are complex numbers, or when we
replace the $t_i$ by the complex numbers $z_i$.
This defines a birational involution $\sigma_{ij}: X\dasharrow X$, 
\begin{equation}
\sigma_{ij}[z_0: \ldots :z_4]=[z'_0:\ldots :z_4']
\end{equation}
with $z'_k=z_k$ if $k\neq i,j$, $z'_i=vz_j$, and $z_j'=vz_i$ with $v= (\ell_i z_i+\ell_jz_j)/(\ell_i z_j+\ell_j z_i)$. Again, since every birational self-map of a K3 surface is an automorphism, these involutions $\sigma_{ij}$ are elements of $\Aut({\hat{X}})$
that commute with the antiholomorphic involution $s_{{\hat{X}}}$; hence, they generate a subgroup of $\Aut({\hat{X}}; s_{{\hat{X}}})$.
Thus we have constructed a family of projective surfaces ${\hat{X}}$, depending on a parameter $\ell \in \P^4(\C)$, endowed with a 
group of automorphisms generated by involutions. 

 For consecutive sides, i.e.\   when $j=i+1$ modulo $5$ (in the next few lines, all indices are considered modulo 5) $\sigma_{i, i+1}$ 
corresponds to the folding transformation described in 
\S~\ref{par:pentagons} and denoted by  $\sigma_{i}$ there. We define the \textbf{folding group} 
$\Gamma$ (resp. the \textbf{extended folding group} $\Gamma^{\mathrm{ext}}$)
 to be the group generated by the 5 folding involutions $\sigma_{i, i+1}$ (resp. the group generated by 
 all 10 involutions $\sigma_{ij}$). Likewise, for given $m\in \set{0, \ldots , 4}$ we introduce the subgroup 
 $\Gamma_m$ (resp. $\Gamma_m^\mathrm{ext}$) stabilizing the side $m$, that is the group generated 
 by the 3 folding involutions    $\sigma_{i, i+1}$     such that  $m\notin\set{i,i+1}$ 
 (resp. by  the 6 involutions $\sigma_{ij}$ with   $m\notin\set{i,j}$). 

\begin{rem}\label{rem:sigma_geometric_pentagon}
Pick a singular point $q_{ij}$, and project $X$ from that point onto a plane, say the plane $\{z_i=0\}$ in the
hyperplane $P=\{\ell_0 z_0+ \cdots +\ell_4z_4=0\}$. One gets a $2$-to-$1$ cover $X\to \P^2_\C$, ramified along a sextic curve (this curve is the union of two cubics, 
see~\cite{Rosenberg}). 
The involution $\sigma_{ij}$ permutes the points in the fibers of this $2$ to $1$ cover: if $x$ is 
a point of $X$, the line joining $q_{ij}$ and $x$ intersects $X$ in the third point $\sigma_{ij}(x)$.
The singularity $q_{ij}$ is an indeterminacy point, mapped by $\sigma_{ij}$ to the opposite line $L_{ij}$.
\end{rem}

\begin{pro}\label{pro:pentagons}
For a  general parameter  $\ell \in (\C^*)^5$: 
\begin{enumerate}[\em (1)]
\item $X$ is a K3 surface with ten nodes, which admits  two real structures $c_X$ and $s_X$ when $\ell\in \P^4(\R)$;
\item if $i,j,k$ are  three distinct  indices (modulo $5$), then $\sigma_{ij}\circ\sigma_{jk}$ is a parabolic transformation on ${\hat{X}}$;
its invariant fibration is induced by $\pi_{lm}\colon [z_0:\ldots : z_4]\mapsto [z_l:z_m]$ where $l$ and $m$ are the complementary indices ({i.e.} $\set{i,j,k,l,m}=\set{0,1,2,3,4}$);  
\item if $i$, $j$, $k$, and $l$ are four distinct indices (modulo $5$), then $\sigma_{ij}$ commutes to $\sigma_{kl}$.
\item the folding group $\Gamma$  (resp. the extended folding group $\Gamma^{\mathrm{ext}}$) 
is a non-elementary 
subgroup of $\Aut({\hat{X}}; s_{\hat{X}})$ that does not preserve any algebraic curve;
\item likewise, the subgroup $\Gamma_m$  stabilizing the side $m$
 is non-elementary, and its invariant curves in 
$\hat X$ are contained in the total transform of
  the lines  $L_{ml}$ for $l\neq m$ (see Remark~\ref{rem:line}). 
\end{enumerate}
\end{pro}

\begin{rem} In~\cite{Dolgachev:Salem}, Dolgachev computes the action of $\sigma_{ij}$ on $\NS(\hat{X})$. This contains a proof of this 
proposition.  
The automorphism groups of 
$\hat{X}$ and of the Enriques surface ${\hat{X}}/{\hat{J}}_X$ are described in~\cite{Dolgachev-Keum} and \cite{Shimada}. 
\end{rem}

The next example shows that the folding groups can be elementary for certain parameters. 
\begin{eg}\label{eg:equilateral_pentagons}
Say that a pentagon is \emph{equilateral} if $\ell_0=\ell_1=\ell_2=\ell_3=\ell_4$. Let  $X_{\mathrm{eq}}(\R)$ be the surface of all equilateral pentagons, modulo 
rotations, translations, and dilatations. It is connected and of genus $4$. 
On $X_{\mathrm{eq}}$, the group generated by the involutions is finite and isomorphic to 
$\mathfrak S_5$,  because 
 $\sigma_{ij}(t_i, t_j) = (t_j, t_i)$ (see Equation~\ref{eq:sigmaij}). 
So, this highly symmetric  case is  also highly degenerate. 
\end{eg}

 \begin{proof}[Proof of Proposition~\ref{pro:pentagons}]
We already established   Assertion~(1) in the previous lemmas. For Assertion~(2), 
denote by $l,m$ the indices for which $\set{i,j,k,l,m} = \set{0, \ldots, 4}$, and consider the linear projection 
$\pi_{lm}\colon \P^4(\C)\dasharrow\P^1(\C)$ defined by 
$[z_0:\ldots :z_4]\mapsto [z_l:z_m]$. The fibers of $\pi_{lm}$  are the hyperplanes containing the plane
$\{z_l=z_m=0\}$, which intersects $X$ on the line $L_{l m}$. This line is a common component of the
pencil of curves cut out by the fibers of $\pi_{l m}$ on $X$, and the mobile part of this pencil determines 
a fibration $\pi_{l m}\rest{X}\colon X\to \P^1$ whose fibers are the plane cubics
\begin{equation}\label{eq:cubic_pilm}
(\ell_l z_l+\ell_m z_m)(\ell_m z_l+\ell_l z_m)z_i z_j z_k=z_l z_m (\ell_i z_jz_k+\ell_j z_iz_k+\ell_k z_iz_j)(\ell_i z_i+\ell_j z_j+\ell_k z_k),
\end{equation}
with $[z_l:z_m]$ fixed. The general member of this fibration is a smooth cubic, hence a curve of genus $1$. 

Then $\sigma_{ij}$ and $\sigma_{jk}$ preserve $\pi_{l m}\rest{X}$, and along the general fiber of 
$\pi_{l m}\rest{X}$ each of them is described by Remark~\ref{rem:sigma_geometric_pentagon}; for instance, $\sigma_{ij}(x)$
is the third  point of intersection of the cubic with the line $(q_{ij}, x)$. Thus, writing such a  cubic as $\C/\Lambda_{[z_l:z_m]}$, 
$\sigma_{ij}$ acts as $z\mapsto -z+b_{ij}$, for some $b_{ij}\in \C/\Lambda_{[z_l:z_m]}$ that depends on $[z_l:z_m]$ and the parameter $\ell$; it has four fixed points on the
cubic curve, which are the points of intersection of the cubic~\eqref{eq:cubic_pilm} with the hyperplanes $z_i=z_j$ and $z_i=-z_j$; equivalently, the
line $(q_{ij},x)$ is tangent to the cubic at these four points.

By Lemma~\ref{lem:charac_parabolic_on_K3} below, either $\sigma_{ij}\circ\sigma_{jk}$
is of order $\leq 66$ (in fact of order $\leq 12$ because it preserves $\pi_{l m}\rest{X}$ fiber-wise), or it is parabolic. In other words, 
 the locus in the parameter space 
 where $\sigma_{ij}\circ\sigma_{jk}$ is not parabolic is defined by the 
equation $(\sigma_{ij}\circ\sigma_{jk})^{12} = \id$. 
Since there 
 do exist pentagons for which $\sigma_{ij}\circ\sigma_{jk}$ is of infinite order
(indeed, this reduces to the corresponding fact for quadrilaterals, see   Example~\ref{ex:quadrilateral} below), we conclude that 
  $\sigma_{ij}\circ\sigma_{jk}$ is parabolic for general~$\ell$.   
  
  Assertion~(3) follows directly from the fact that $\sigma_{ij}$ changes the coordinates $z_i$ and $z_j$ but keeps the other three fixed. 

To prove Assertion~(4), we see that 
for a general parameter $\ell$, $\Gamma$ contains two such parabolic automorphisms 
associated to distinct fibrations 
$\pi_{lm}$ and $\pi_{l'm'}$ so it is non-elementary (this follows from Theorem 3.2 in~\cite{Cantat:Milnor}). To show that  $\Gamma$ does not preserve any curve 
in $\hat{X}$, assume by way of contradiction that 
 $E\subset {\hat{X}}$ be a $\Gamma$-periodic irreducible curve, and denote by $F$ its image in 
$\P^4_\C$ under the projection ${\hat{X}}\to X$. If $F$ is  a point, it is one of the singularities $q_{ij}$.  Note that $\Gamma$ acts transitively on the singularities of $X$: given any pair of singularities $(q,q')$, there is an element of $\Gamma$ which is well defined at $q$ and $q'$ and maps $q$ to $q'$. 
Thus, we can assume that $j=i+1$, and changing $E$ into its image
under (the lift of) $\sigma_{ij}\in \Gamma$ the curve $F$ becomes the line $L_{ij}$. So, we may assume that $F$ is an irreducible curve. 
Now, the orbit of $F$ is periodic under the action of the parabolic automorphisms 
$g_i=\sigma_{ij}\circ\sigma_{jk}$, with $j=i+1$ and $k = i+2$ modulo $5$.
Since the invariant curves of a parabolic automorphism
 are contained in the fibers of its invariant fibration, we deduce that $F$ is 
contained in the fibers of each of the projections $\pi_{lm}$ with $m=l+1$, which 
is   impossible. So there is no invariant curve.

The corresponding statement for $\Gamma^{\rm ext}$ follows immediately. 

The reasoning for~(5) is similar. Without loss of generality, assume $m=0$. Then again 
$\Gamma_m$ is non-elementary since it contains the parabolic elements $\sigma_{12}\circ\sigma_{23}$ and
$\sigma_{23}\circ\sigma_{34}$ (with distinct associated fibrations 
$\pi_{04}$ and $\pi_{01}$. 
 Reasoning as above shows that if $E\subset \hat X$ is any $\Gamma$-periodic irreducible curve projecting to a curve in $X$, then 
 its image $F$ in $X$ is contained in a fiber  of each of the  projections 
$\pi_{01}$, $\pi_{02}$, $\pi_{03}$ and $\pi_{04}$.  So we conclude that $F\subset \set{z_0 = 0}$, but then the equation of $X$ forces another coordinate to vanish, and we conclude that 
$F$ is one of the $L_{0l}$. 
\end{proof}

\begin{eg}\label{ex:quadrilateral}
 Let us give some  geometric explanations for 
  Assertion~(2) of Proposition~\ref{pro:pentagons}. Choose   
 $(l,m) = (1,2)$, and normalize the   pentagons so 
 that $a_0=(0,0)$ and $t_0=1$, which means that $a_1=(\ell_0,0)$. 
In homogeneous coordinates, this corresponds to the normalization 
$[1:z_1:z_2:z_3:z_4]$ with $z_i= t_i$. The pentagons contained 
 in a fiber of $\pi_{12}\rest{X}$ have three 
fixed vertices, namely $a_0$, $a_1$ and $a_2$. The remaining free vertices $a_3$ and $a_4$ 
move along  the circles centered at $a_2$ and $a_0$, of respective radii $\ell_2$ and $\ell_4$, with the 
constraint $a_3a_4=\ell_3$. 
These  circles are two conics,  the fiber is an elliptic curve which is a
$2$-to-$1$ cover of each of these two conics,
 the involutions  $\sigma_{23}$ and $\sigma_{34}$ preserve these fibers, and 
 $\sigma_{23}\circ \sigma_{34}$ is a translation on the elliptic curve. 
Forgetting the vertex $a_1$, 
we obtain a quadrilateral $(a_0, a_2, a_3, a_4)$,  and  
one recovers the transformations described in \cite{benoist-hulin}. 
The side lengths of this quadrilateral are
$\ell_2$, $\ell_3$, $\ell_4$ and $\norm{\overrightarrow{a_0a_2}}$, hence the translation vector 
(which only depends on these lengths) varies 
non-trivially when deforming the pentagon, which corresponds to the twisting property of parabolic transformations. 
\end{eg}

\begin{lem}\label{lem:charac_parabolic_on_K3}
 Let $X$ be a K3 or Enriques surface, and $\pi\colon X\to B$ be a genus $1$ fibration.
 If $g\in\Aut(X)$ maps some fiber $F$ of $\pi$ to a fiber of $\pi$, then $g$ preserves the fibration and
either $g$ is parabolic or it is periodic of order $\leq 66$. 
\end{lem}
\begin{proof}  
Since $g$ maps $F$ to some fiber $F'$,  it maps the complete linear system $\vert F\vert$ to $\vert F'\vert$, 
but both linear systems are made of the fibers of $\pi$. So $g$ preserves the fibration and is not loxodromic.  
If $g$ is not parabolic it is elliptic, and its action on cohomology has finite order since it 
preserves $H^2(X, \Z)$.
 On a K3 or Enriques 
surface every holomorphic vector field vanishes identically, so $\Aut(X)^0$ is trivial and the kernel of the homomorphism $ \Aut(X)\ni f\mapsto f^*$ is finite (see~\cite[Theorem 2.6]{Cantat:Milnor}); as a consequence, any elliptic automorphism has finite order.   The upper
bound on the order of $g$ was obtained in~\cite{Keum:2016}. \end{proof}

\section{Random   foldings and ergodic theory} \label{sec:pentagons_dynamics}

We have now gathered enough geometric information to draw some dynamical 
consequences on the dynamics of pentagon folding. 

\subsection{Dynamics on $\Pent^0(\ell)$}
Recall  that the folding groups $\Gamma$, $\Gamma^{\rm ext}$ and 
$\Gamma_m$ were defined in  \S~\ref{subs:involutions}. 
Recall also that a parameter $\ell \in \R^5_{>0}$ is admissible when it corresponds to at least one pentagon.

\begin{figure}[h]
\includegraphics[width=12cm]{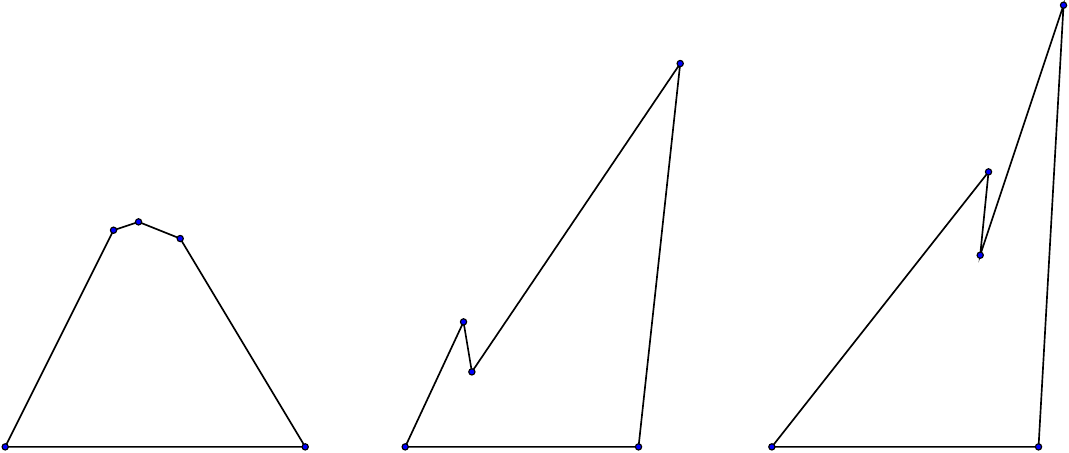}
\caption{{\small Pentagons with disconnected configuration spaces  
(respectively corresponding  to the codes $\bigl({\begin{smallmatrix} \cap &\cap   \\ \cap &\cap \end{smallmatrix}}\bigl)$, 
$\bigl({\begin{smallmatrix} \subset  &\emptyset   \\  \subset  &\emptyset  \end{smallmatrix}}\bigl)$
and $\bigl({\begin{smallmatrix} \subset  &\supset   \\  \subset  &\supset  \end{smallmatrix}}\bigl)$ in~\cite[Table 4]{Curtis-Steiner})
  }}\label{fig:pentagons_disconnected}
\end{figure}

Recall  from Remark~\ref{rem:topology} that $X(\R)\simeq \Pent^0(\ell)$ can be disconnected, in which case it is the disjoint union of two tori. 
This happens in $3$ of the $19$ possible configurations listed in~\cite{Curtis-Steiner}, the shapes of which are  sketched in Figure~\ref{fig:pentagons_disconnected}.
Such a pentagon cannot be deformed continuously to its reflection along the horizontal side, hence the configuration space is disconnected. On the other hand, 
\begin{enumerate}
\item folding it along its longest diagonal maps it into the other component, 
 so $\Gamma$ and $\Gamma^{\rm ext}$ act transitively on the set of components of $\Pent^0(\ell)$;
 
\item the involutions preserving the horizontal side preserve each component of $\Pent^0(\ell)$, so $\Gamma_m$ preserves each component of $\Pent^0(\ell)$ for $m=0$. 
\end{enumerate}

In the next two statements, ``general'' 
means that the conclusions of Proposition~\ref{pro:pentagons} are satisfied.

\begin{pro}\label{pro:ergodicity_pentagons}
For a general admissible  
parameter $\ell\in  \R_{>0}^5$, the action of $\Gamma$  (resp. of $\Gamma^{\rm ext}$) on $\Pent^0(\ell)$ 
is ergodic with respect to its natural  volume form.  
Likewise, for $m\in \set{0, \ldots, 4}$
 the action of $\Gamma_m$ is ergodic on each $\Gamma_m$-orbit of connected components of $\Pent^0(\ell)$.  
\end{pro}

\begin{proof}
Since $\Gamma$ has no invariant curve and acts transitively on the set of components of $\Pent^0(\ell)$, 
the result follows from case (c) of~\cite[Thm A]{invariant}. For $\Gamma_m$, the same argument applies, upon adding the observation 
 that the $\Gamma_m$-invariant curves do not intersect $\Pent^0(\ell)$.
\end{proof}

\begin{thm}\label{thm:dynamics_pentagons}
Let $\ell\in  \R_{>0}^5$ be a general admissible parameter.  Fix a   probability measure $\nu$ on $\Gamma$ (resp. $\Gamma^{\rm ext}$) satisfying the moment condition (M), 
and whose support generates $\Gamma$ (resp. $\Gamma^{\rm ext}$). Then all $\nu$-stationary 
measures on $\Pent^0(\ell)$ are invariant, and the  
ergodic invariant measures are given by:
\begin{itemize}
\item  finitely many  periodic orbits;
\item $\vol_{\Pent^0(\ell)}$. 
\end{itemize}
In particular, the set of $\nu$-stationary 
measures  is a finite dimensional simplex. 
\end{thm}

\begin{proof}
This  follows from Proposition~\ref{pro:ergodicity_pentagons}, 
the stiffness theorem of \cite{stiffness}, and the 
finiteness of the set of finite orbits established in~\cite[Thm C]{finite_orbits}. 
\end{proof}

\begin{rem}
The classification of stationary measures also holds for $\Gamma_m$, except that in this case we cannot   apply~\cite[Thm C]{finite_orbits} to infer
 the finiteness of the set of periodic orbits, because of the existence of 
 invariant curves.
\end{rem}

The random ergodic theorem then implies that for $\vol_{\Pent^0(\ell)}$-almost every pentagon $x\in \Pent^0(\ell)$, 
the sequence of  empirical measures 
$\unsur{n}\sum_{k=1}^n \delta_{f^k_\omega(x)}$ is  
almost surely equidistributed on $\Pent^0(\ell)$.  As explained in \S~\ref{subs:UE}, 
to deduce the more precise result that this random equidistribution holds for 
{every} pentagon  with infinite orbit
we need some information about  periodic orbits. This is where it is useful to work 
with the extended folding group. 

\begin{thm}\label{thm:finite_orbits_pentagons} 
For a very general admissible parameter $\ell\in \R_{>0}^5$, every orbit of $\Gamma^{\rm ext}$  in $\Pent^0(\ell)$ is infinite. Hence,  for such a parameter, the only ergodic $\Gamma^{\rm ext}$-invariant probability measure on $\Pent^0(\ell)$ is the natural volume.  
\end{thm}

We thus  obtain the following equidistribution result (the moment condition ($\mathrm M_+$) was defined in \S~\ref{subs:UE}): 
 
\begin{cor}\label{cor:pentagons_UE}
Fix a   probability measure $\nu$ on $\Gamma^{\rm ext}$ satisfying the moment condition 
($\mathrm M_+$)
and  generating $\Gamma^{\rm ext}$.
There is an open and  dense   subset of full measure in the set of admissible parameters $\ell$
such that  for any $x\in \Pent^0(\ell)$:
 either $\Gamma^{\rm ext}\cdot x$ is finite or for $\nu^\N$-almost every $\omega$, 
$$\unsur{n}\sum_{k=1}^n \delta_{f^k_\omega(x)} \to \vol_{\Pent^0(\ell)}$$ as $n\to +\infty$.
\end{cor}

\begin{proof}
By~\cite[Thm 1.5]{hyperbolic}, at a very general parameter the action of $\Gamma^{\rm ext}$ is uniformly expanding. 
Since uniform expansion is an open property, it holds on a dense open set.
Then the equidistribution result 
follows 
from~\cite[Thm 10.4]{hyperbolic}.  Since the complement of this dense open set is contained in a countable union of proper Zariski closed subsets,  
 it has zero  Lebesgue measure.
\end{proof}

\begin{proof}[Proof of Theorem~\ref{thm:finite_orbits_pentagons}]  
By Theorem~\ref{thm:dynamics_pentagons}, it suffices to prove the first assertion. If this assertion were not correct, then, arguing exactly as 
in~\cite[Thm A]{finite_orbits}, we would  find a finite index subgroup $\Gamma'\subset\Gamma^{\rm ext}$ such that the algebraic set 
\begin{equation}
\mathcal Z = \set{(\ell, x)\in \C^5\times X_\ell(\C), \ x\in X_\ell(\C), \forall f\in \Gamma', f(x)  = x}
\end{equation}
has a Zariski dense projection to $\C^5$. Since $\Gamma^{\rm ext}$ does not preserve any curve in $X_\ell$ (for a general $\ell$), then so does $\Gamma'$, and as a consequence,   the natural projection $\pi\colon \mathcal Z \to \C^5$ is generically  finite. Then, there exists a Zariski dense open subset $W\subset \C^5$ such that above $W$, $\pi$ is a finite unramified cover and the surfaces $\hat{X}_{\ell}$ are all smooth. 

Since $\R_{>0}^5$ is Zariski dense in $\C^5$, the intersection $W_{>0}(\R):=W\cap \R_{>0}^5$ is the complement of a proper Zariski closed subset of $\R_{>0}^5$. Reducing this open set slightly (by cutting out additional hypersurfaces if necessary), we may assume that every connected component of $W_{>0}(\R)$ is simply connected. If $W_i$ is such a connected component (for the euclidean topology), and if $(\ell, x)$ is a point of $\mathcal Z$ with projection $\ell\in W_i$, then there is a unique continuous (algebraic) section  
\begin{equation}\label{eq:section}
\ell'\in W_i\mapsto (\ell',x')\in \mathcal Z
\end{equation} 
of $\pi$ defined on $W_i$ that maps $\ell$ to $(\ell,x)$. This will be referred to as 
 the continuous continuation of the $\Gamma'$-fixed point $x$.
Moreover, under our contradiction hypothesis, we may choose $W_i$ and $(\ell, x)$ such that the section defined in~\eqref{eq:section} takes its values $x'$ in $\Pent^0(\ell')$ 
({i.e.} in the real part of the complex surface $X_{\ell'}$ for the real structure defined in~\eqref{eq:sX}). We fix such a pair $(W_i, (\ell,x))$, where $x$ corresponds to a normalized pentagon $(a_i)$ ({i.e.} $a_0=(0,0)$, $a_1=(\ell_0,0)$). 

Recall that a planar polygon is said to be degenerate, or flat, if it is contained in a line (ie. its vertices are collinear). 
Since we can choose $\ell$ as we wish in the open set $W_i$, 
we may assume that 
\begin{equation}
{\text{no line contains three of the vertices}}\;  a_i.
\end{equation}
In particular, $x$ itself is not flat. 

The triangle $(a_2a_3a_4)$ and the quadrilaterals $a_0a_2a_3a_4$
(with side lengths $\ell_5 := \norm{\overrightarrow{a_0 a_2}}$,
$\ell_2$, $\ell_3$, $\ell_4$) and $a_1a_2a_3a_4$ 
(with side lengths  $\ell_1$, $\ell_2$, $\ell_3$, and $\ell_6 := \norm{\overrightarrow{a_1 a_4}}$) 
 are non-degenerate. These quadrilaterals are periodic for the respective 
 transformations $\sigma_{23}\circ\sigma_{34}$ and $\sigma_{12}\circ\sigma_{23}$. Therefore, by the Darboux alternative for quadrilaterals (see~\cite{darboux, benoist-hulin, izmestiev2015} and also Example~\ref{ex:quadrilateral}), their side lengths satisfy a non-trivial relation. 
 More precisely, given the lengths $\ell_1, \ldots , \ell_4$,
  there exists a countable set $D_5 = D_5(\ell_2, \ell_3, \ell_4)$  and 
  $D_6 = D_6(\ell_1, \ell_2, \ell_3)$   such that 
  \begin{equation}\label{eq:lengths-conditions}
  \ell_5\in D_5(\ell_2, \ell_3, \ell_4) \quad \text{and} \quad \ell_6\in D_6(\ell_1, \ell_2, \ell_3).
  \end{equation} 
 
We  now deform the pentagon by varying $\ell_0$ while keeping the other side lengths fixed; this gives a variation $\ell'=(\ell'_0, \ell_1, \ldots , \ell_4)$ of $\ell$ in $W_i$, and a continuation $x'=(a'_i)\in \mathcal Z$ of $x$ (with the normalization $a'_0 = (0, 0)$ and $a'_1 = (\ell_0, 0)$). 
Let us show that for $\ell'$ near $\ell$, the continuation $x'$ of $x$ is geometrically determined, in a unique way, by the  length conditions~\eqref{eq:lengths-conditions}. 
Since  
 $\big\Vert{\overrightarrow{a'_0 a'_2}}\big\Vert$ varies continuously and must stay in $D_5$, it is constant, 
hence  the triangle $a'_0a'_1a'_2$  has side lengths $\ell'_0$, $\ell_1$ and $\ell_5$ which defines it 
 uniquely up to isometry. Since the initial triangle $a_0a_1a_2$ is not flat, then the point $a'_2$ is 
uniquely determined as a continuation of $a_2$. 
Similarly, the triangle 
$a'_0a'_1a'_4$ determines $a'_4$. 
Finally, since the lengths $\ell_2$ and $\ell_3$ are fixed and the triangle $a_2a_3a_4$ is not flat, the continuation 
$a'_3$ is uniquely determined. In conclusion, the periodicity of $x'$ under the action of $\sigma_{23}\circ\sigma_{34}$ and of $\sigma_{12}\circ\sigma_{23}$ determines $x'$  as 
$\ell_0$ varies; thus, in what follows we can forget about $\mathcal Z$: the section $(\ell',x')$ is, in fact, given by the geometric construction we have just described. 

To reach the desired contradiction, we now argue that the periodicity of $x'$ 
under the parabolic automorphism 
$\sigma_{13}\circ \sigma_{34}$  
creates an additional  rigidity that cannot be satisfied  (we might use 
$\sigma_{12}\circ \sigma_{24}$ instead). Indeed, $\sigma_{13}\circ \sigma_{34}$ 
acts on  the vectors $v_1 = \overrightarrow{a_1 a_2}$, $v_3 = \overrightarrow{a_3 a_4}$
and $v_4 = \overrightarrow{a_4 a_0}$, so it can be seen as 
a   transformation of the ``virtual'' quadrilateral $a_0b_1a_3a_4$, where $b_1$ is such that 
$\overrightarrow{a_3 b_1} = \overrightarrow{a_2 a_1}$. Since 
$x'$ is a $\Gamma^{\rm ext}$-periodic point of $\Pent^0(\ell')$ for all $\ell'$, the quantity
$\norm{v_1+v_3+v_4} = \big\Vert{\overrightarrow{a_0b_1}}\big\Vert$ must be constant; more precisely, it does not vary with $\ell'_0$ on a neighborhood of $\ell$. But this function depends algebraically on the parameters, and any neighborhood of $\ell$ in $W_i$ is Zariski dense in $\C^5$, so we conclude that 
   this function actually does not depend on $\ell'_0$.  
Thus, to reach the desired contradiction, one just needs to contemplate Figure~\ref{fig:pentagon_flex}. 
\end{proof}

\begin{figure}[h]
\begin{minipage}{6cm}\includegraphics[width=6cm]{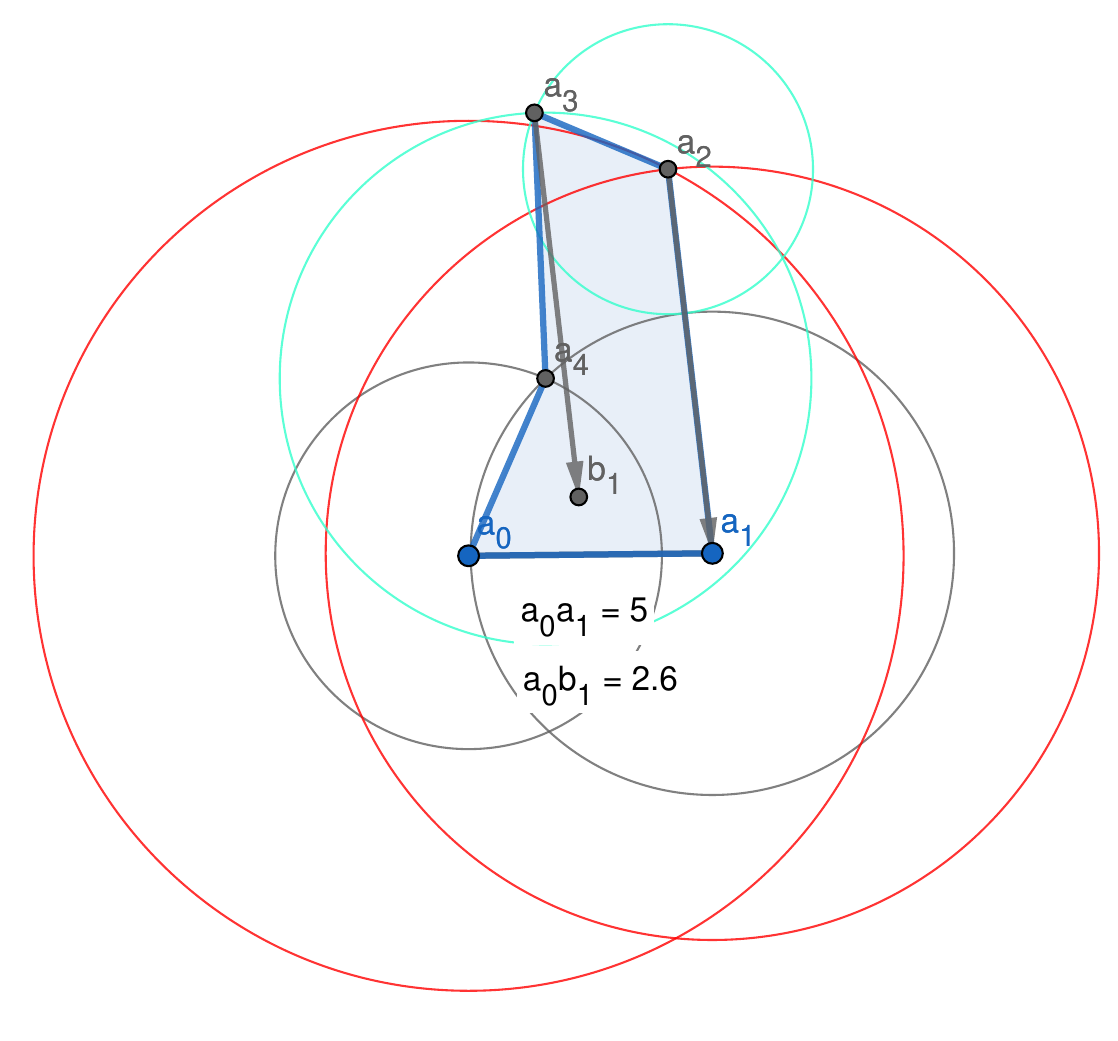}\end{minipage}  
\begin{minipage}{6cm}\includegraphics[width=7.4cm]{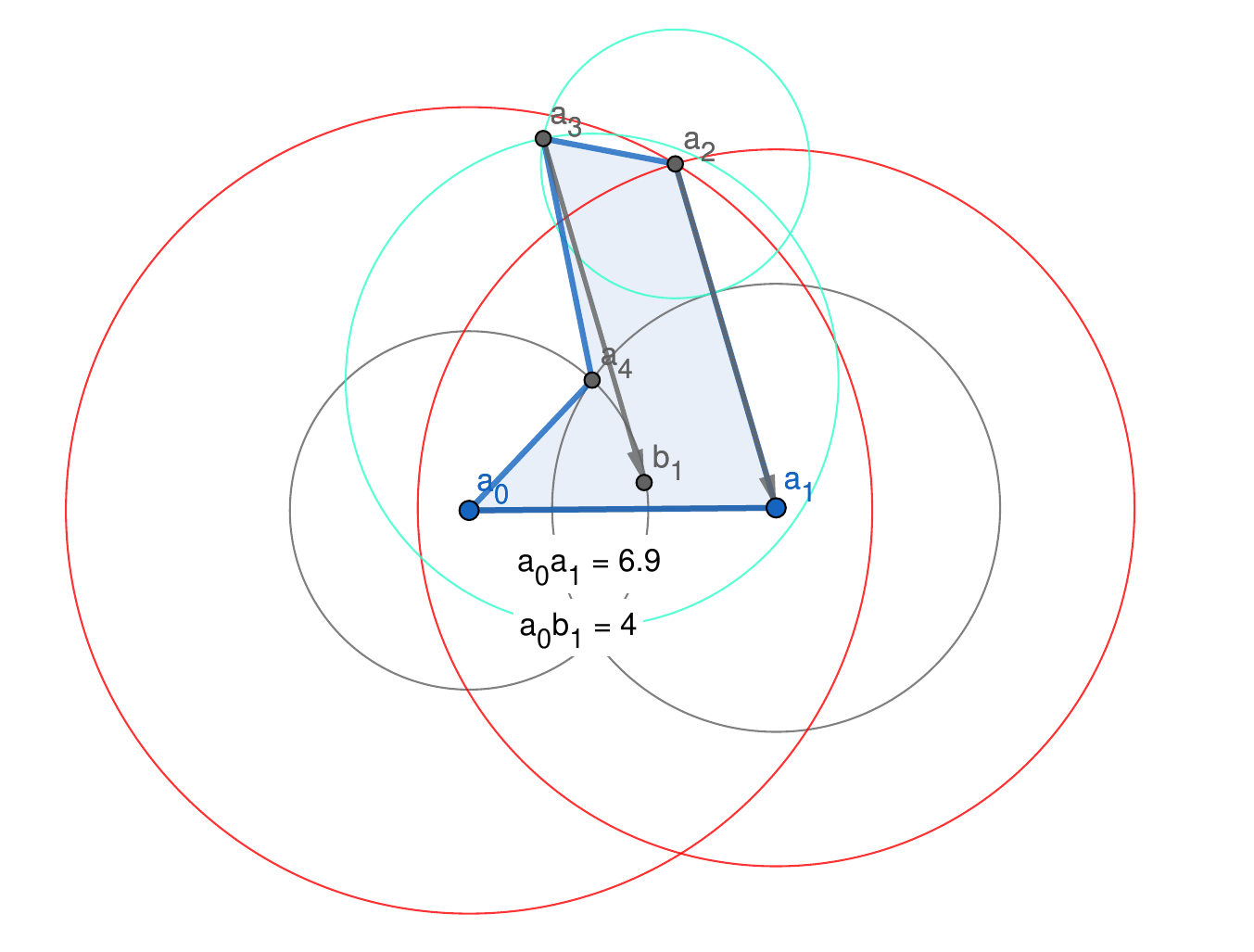}\end{minipage}
\caption{{\small Proof of Theorem~\ref{thm:finite_orbits_pentagons}. The  corresponding
circles have the same radius on the two figures, and serve for the construction of $a_2$, $a_3$ and $a_4$ (for instance on both pictures 
the   red circle on the left is centered at $a_0$ and has radius $\ell_5$ and the red circle on the right is centered at $a_1$ with radius $\ell_2$).  
When $\ell_0$ varies the distance $a_0b_1$ varies as well. (Realized with GeoGebra and available at 
\url{https://www.geogebra.org/m/edydwafb}). }}
\label{fig:pentagon_flex}
\end{figure}

Since we use transformations like 
$\sigma_{13}\circ\sigma_{34}$ in the proof, it is necessary to work with $\Gamma^{\rm ext}$. The analogous result for $\Gamma$ remains elusive (\footnote{We actually do not know any example of a ``general'' periodic pentagon, that is a pentagon $x\in \Pent^0(\ell)$ with a finite $\Gamma$-orbit, 
  for which $\ell$ satisfies the 
conditions of Proposition~\ref{pro:pentagons}.}).

\begin{que}\label{que:Gamma-vs-Gammaext}~
\begin{enumerate}
\item Is it true that for a very general set of lengths $\ell\in \P^4(\C)$, every orbit of $\Gamma$ 
(resp. $\Gamma_m$) in $\Pent^0(\ell)$ ({resp.} in $X_\ell$) is infinite? 
\item Is there a dense open set of parameters $\ell\in \P^4(\R)$ for which 
every finite orbit is uniformly expanding? 
\item Can we replace ``very general'' by ``general'' in Theorem~\ref{thm:finite_orbits_pentagons} and in 
Question (1)? 
\end{enumerate}
\end{que}
 
Let us expand a little bit on this. 
As seen above, a positive answer to the first question implies a positive answer to the second one. 
On the other hand, even  without excluding the existence of a persistent 
periodic orbit for $\Gamma$, it might still be possible to show that the action of $\Gamma$ on $\Pent^0(\ell)$ is 
uniformly expanding generically  (this point of view is 
developed in~\cite[\S 9]{hyperbolic}). Indeed, by~\cite{hyperbolic}, it is enough to show that 
at such a persistent periodic orbit, for a general parameter $\ell$, we can find  
parabolic elements in $\Gamma$ that exhibit a non-trivial twisting in different directions. 
Geometrically, this would require to understand how the rotation numbers 
associated to sub-quadrilaterals vary when deforming the pentagon (an explicit formula 
for this rotation number is given in~\cite{benoist-hulin, izmestiev2015}).

\subsection{Dynamics on $\Pent^1(\ell)$} Recall that $\mathbb{U}$ denotes the unit circle in $\R^2\simeq \C$. To an element $u\in \mathbb{U}$ 
corresponds a unique rotation $R_u$ (centered at the origin); in complex coordinates, it   just corresponds to  multiplication by $u$, and its inverse is $R_u^{-1}=R_{\overline{u}}$.

For $\ell = (\ell_0, \ldots , \ell_4) \in \R_{>0}^5$, let us introduce the space of pentagons modulo translations  
\begin{equation}
\Pent^1(\ell) = \Pent(\ell)/\R^2.
\end{equation}
With notation as in \S~\ref{par:pentagons}, 
every element $\tilde{x}$ of $\Pent^1(\ell)$ admits a unique representative with $a_0 = (0,0)$. 
Then,  $
\overrightarrow{a_0a_1} = \ell_0 t_0$, for some unit vector $t_0\in \mathbb{U}$ and if we apply the rotation $R_{\overline{t_0}}$, we get  
a normalized pentagon $x$ (with $a_0=(0,0)$ and $a_1=(\ell_0,0)$). This shows that
 $\Pent^1(\ell)$ is a trivial circle bundle over  $\Pent^0(\ell)$: 
 \begin{equation}
 \Pent^1(\ell) \simeq \Pent^0(\ell) \times \mathbb{U}.
 \end{equation}
The reciprocal diffeomorphism is obtained as follows. 
 Let $(x,u)$ be an element of $\Pent^0(\ell) \times \mathbb{U}$.
 \begin{itemize}
\item To  $x$, one associates its normalized pentagon, that is, the unique pentagon $(a_i)$  in its $\SO(2)\ltimes \R^2$ orbit corresponding to $x$ such that $a_0=(0,0)$ and $a_1=(\ell_0,0)$. 
 \item then one rotates it by $R_u$ to get an element of $\Pent^1(\ell)$.
 \end{itemize}
 As a real algebraic variety, $\Pent^1(\ell)$ is 
 the real locus, with respect to  the involution 
 \begin{equation}
 s\colon (z_i)\mapsto (1/\overline{z_i}),
 \end{equation}
 of the variety defined  in $\C^5$ by the system of Equations~\eqref{eq:equation_pentagons}. 
 
 The involutions 
 $\sigma_{ij}$ and the groups $\Gamma$ and $\Gamma_m$ are defined exactly as before, with the 
 same formulas; to avoid confusion with the action on $\Pent^0$, we may denote the involutions by   
 $\tilde{\sigma}_{ij}$ and the corresponding pentagons by $\tilde x$). 
 In the coordinates $(x,u)\in \Pent^0(\ell)\times {\mathbb{U}}$, 
 $\tilde{\sigma}_{ij}$ is given by  
  \begin{equation}\label{eq:skew}
\begin{cases}
\tilde{\sigma}_{ij}(x,u) = (\sigma_{ij}(x), u), \text{ if } 0\notin\set{i,j}\\
\tilde{\sigma}_{0j}(x,u)  = (\sigma_{0j}(x), h_{0j}(x)u)
\end{cases}
 \end{equation}
 where $x\mapsto h_{0j}(x)$ takes values in ${\mathbb{U}}$.
To justify the second line, simply observe that $\sigma_{ij}$ commutes with the rotation $R_v$, {i.e.} with $(x,u)\mapsto (x, uv)$; thus, $\tilde{\sigma}_{0j}(x,u)  = R_u \tilde{\sigma}_{0j}(x,1)$. 

Instead of choosing $a_0$ as a base point (which is translated to the origin), we could choose any of the five vertices $a_m$, $m=0, \ldots, 4$. 
This provides five  different 
 identifications   $\Pent^1(\ell) \simeq \Pent^0(\ell) \times \mathbb{U}$; for each of them, we denote by 
 $\vartheta_m\colon \Pent^1(\ell)\to \mathbb{U}$ 
  the projection onto the second factor (in other words, for a pentagon $\tilde{x}$, 
 $\vartheta_m(\tilde{x})$ is the angle $\angle(\overrightarrow{a_ma_{m+1}}, (1, 0))$).
Changing from the  basepoint $a_i$ to the basepoint $a_j$ yields a change of coordinates of the form $(x,u)\mapsto (x, R_{\alpha_{ij}(x)}u)$, where $\alpha_{i,j}(x)$ is the unit vector with  angle
 $\angle(\overrightarrow{a_ia_{i+1}}, \overrightarrow{a_ja_{j+1}})$. So, 
 \begin{equation}
 \vol_{\Pent^1(\ell)}  := \vol_{\Pent^0(\ell)} \times \leb_{\mathbb{U}}
 \end{equation} 
 defines a  $\Gamma$-invariant volume. 
 
\begin{pro}\label{pro:ergodic_pent1}
The action of $\Gamma$ is ergodic on  $\Pent^1(\ell)$ with respect to the natural volume $\vol_{\Pent^1(\ell)}$. 
\end{pro}

\begin{proof}
The argument is borrowed from Chivet's master's thesis~\cite{chivet}.  
Let $B$ be a Borel set of positive volume such that for every $\gamma\in \Gamma$, 
\begin{equation}
\vol_{\Pent^1(\ell)}(B\Delta\gamma\inv(B))=0.
\end{equation} 

Pick an index $m\in \set{0, \ldots, 4}$, and note that $\Gamma_m$ preserves the fibers of $\vartheta_m$.
Since $B$ is almost $\Gamma_m$-invariant 
 and the $\Gamma_m$-action on $\Pent^0(\ell)$ is ergodic, we get that $B$ is $\vartheta_m$-saturated, which means that every fiber of 
 $\vartheta_m$ intersects $B$ on a set of  zero or full measure for $\vol_{\Pent^0(\ell)}$. 
 
 Let us fix a value  of $m$, say $m=0$ and work in the system of 
 coordinates $(x,u)$ associated to this choice; in these coordinates,  $R_v$ is of the form 
 $(x, u)\mapsto (x,uv)$. 
 Set 
 \begin{equation}
 B_0 = \set{u  \in \mathbb{U};\; \vol_{\Pent^0(\ell)}(B\cap \vartheta_0\inv(u)) = 1}. 
 \end{equation}
 If we push $\vol_{\Pent^1(\ell)}$ onto $ \mathbb{U}$ by $\vartheta_0$, one gets the Lebesgue measure, and $B_0$ has a positive Lebesgue measure;
below, we shall consider Lebesgue density points $u_0$ of $B_0$.
 
 Now fix another value of $m$, say $m=1$. Let $F$ be any fiber of $\vartheta_1$. Since 
 $\vartheta_0\rest{F}: F\to \mathbb{U}$ admits a regular value, there exists $\delta>0$  
 and an open set $W\subset F$ such that $\vartheta_0\rest{W}$ is a submersion onto an interval of length $2\delta$. In addition, since any other fiber  of $\vartheta_1$ 
 is of the form $R_v(F)$, we deduce that 
 there exists an open set $V\in \Pent^0(\ell)$  such that, 
 for any $u_0\in  \mathbb{U}$, there exists a neighborhood $N(u_0)\subset \mathbb {U}$  such that,
  for every $u\in   N(u_0)$
and every $x\in V$, there exists 
 $W_x\in \vartheta_1\inv(\vartheta_1(x, u))$ such that $\vartheta_0\rest{W_x}$ 
 realizes a submersion $W_x\to ]u-\delta, u+\delta[$. 
 
 Fix  any density point $u_0$ of  $B_0$, so that  there exists 
  $A\subset B_0\cap N(u_0)$ of positive Lebesgue measure. 
  For every $u\in A$, $B\cap (V\times\set{u})$ 
  is of full $\vol_{\Pent^0(\ell)}$ mass in $V\times\set{u}$. By Fubini's theorem, and the fact that the $B$ is $\vartheta_1$-saturated,  there is a set of positive measure $B'\subset A\times V$ such that for 
 any $(x,u)\in B'$, the $\vartheta_1$-fiber of $(x,u)$ is contained in $B$ (modulo a set of measure zero). 
 Projecting back by $\vartheta_0$ and    using   the fact that 
 $B$ is $\vartheta_0$-saturated, 
 we conclude that $]u_0 - \delta, u_0+\delta[\subset B_0$  (modulo a set of measure zero). From this we easily conclude that 
 $B_0  = \mathbb{U}$ (modulo a set of measure zero) and,  applying the    Fubini theorem again
 completes the proof. 
\end{proof}

\begin{thm}\label{thm:dynamics_pentagons2}
Let $\ell\in \R^5_{>0}$ be a general parameter. 
Fix a   probability measure $\nu$ on $\Gamma$   satisfying the moment condition (M) and 
 whose support generates $\Gamma$. 
 Then, the ergodic, $\nu$-stationary 
measures  are  
\begin{itemize}
\item atomic measures on $\Gamma$-periodic orbits(\footnote{
Such periodic orbits are even rarer than periodic orbits on $\Pent^0(\ell)$ (we do not know any example). So 
we strongly believe that there exists a dense, Zariski open subset $W\subset \R^5_{>0}$  
such that for $\ell\in W$,  $\Gamma$ does not have any finite orbit in $\Pent^1(\ell)$  (see Question~\ref{que:Gamma-vs-Gammaext}).});
\item measures of the form $$\left(\sum_{x\in F} \delta_x\right)\times \leb_{\mathbb{U}}$$ where $F$ is a finite orbit for the action of $\Gamma$ on $\Pent^0(\ell)$;
\item the measure $\vol_{\Pent^1(\ell)}$. 
\end{itemize}
In particular, every  $\nu$-stationary 
measure on $\Pent^1(\ell)$ is invariant.

If $\nu$ is a probability measure on $\Gamma^{\rm ext}$ satisfying condition (M) and generating $\Gamma^{\rm ext}$, and if $\ell$ is very general, then the only ergodic 
$\nu$-stationary measure on $\Pent^1(\ell)$ is $\vol_{\Pent^1(\ell)}$. 
\end{thm}

In the latter case, we could also state an equidistribution result in the spirit of Corollary~\ref{cor:pentagons_UE}.
The core of the proof is the following 
stationary version of a celebrated argument due to Furstenberg \cite{furstenberg:1961}, 
which we state here in a general form.  

\begin{lem}\label{lem:furstenberg}
Consider a random dynamical system $(X, (f_\omega), \nu, \mu)$, where $X$ is a compact metric space, $\nu$ is a Borel probability measure on $\mathrm{Homeo}(X)$, and $\mu$ is an ergodic $\nu$-stationary measure on $X$.
Let $G$ be a compact group, with Haar measure $\lambda$. 
Consider a $G$-extension of this random dynamical system, by transformations of $X\times G$ of the form 
$$F_\omega:(x,g)\mapsto (f_\omega(x), h_\omega(x)\cdot g)$$
where $h_\omega$ is a $G$-valued continuous cocycle that is, a continuous $G$-valued map depending measurably on $\omega$.  Then  for this extension,   $\mu\times \lambda$ is $\nu$-stationary, and if it is  ergodic, it is the unique 
 stationary measure projecting down to $\mu$. 
\end{lem}

\begin{proof}
The stationarity of $\mu\times \lambda$ is an easy exercise. Let $\tilde \mu$ be an  ergodic, 
$\nu$-stationary probability measure on $X\times G$ with marginal $(\pi_X)_*\tilde{\mu}=\mu$. By the random ergodic theorem,  
$\tilde \mu$-almost every $(x,g)$  is $\tilde \mu$-generic: this means that for
 $\nu^\N$-almost every $\omega$, 
\begin{equation}
\unsur{n}\sum_{k=0}^{n-1} \delta_{F_\omega^k(x,g)}\underset{n\to\infty}\longrightarrow\tilde \mu.
\end{equation}
Let $R_h: (x,g)\mapsto (x, g\cdot h)$ denote the right translation induced by $h\in G$. Note that $F_\omega$ commutes with $R_h$ and $\mu\times \lambda$ is $R_h$-invariant. From this we infer that  if 
$(x,g)$ is $\mu\times \lambda$-generic, then all points in the fiber $\set{x}\times G$ are 
$\mu\times \lambda$-generic. By Fubini's theorem, it follows 
 that $\tilde\mu$-almost every point is in fact $\mu\times \lambda$-generic, hence 
 $\tilde \mu = \mu\times \lambda$, as asserted. 
\end{proof}

We shall also need the following well known lemma.

\begin{lem}\label{lem:circle-dynamics-obvious}
Let $\Gamma$ be a  group of   rotations of the unit circle and $\nu$ be 
a measure such that 
$\bra{\supp(\nu)} = \Gamma$. 
\begin{itemize}
\item If $\Gamma$ is infinite,  
the Lebesgue measure is the only $\nu$-stationary measure.  
\item If $\Gamma$ is finite, the  ergodic stationary measures are supported by its   finite orbits.
\end{itemize} \end{lem}

\begin{proof} Let us prove the first assertion (the second one is obvious). Since $\Gamma$ is infinite,  it is a dense subgroup of $\SO(2)$.
Since $\SO(2)$ is Abelian, by the Choquet-Deny theorem, every stationary measure is invariant.
Now, if $\mu$ is $\Gamma$-invariant, it  is $\SO(2)$-invariant too, by density and dominated convergence. Thus, 
$\mu$ is the Lebesgue measure. 
\end{proof}

\begin{proof}[Proof of Theorem~\ref{thm:dynamics_pentagons2}]
By Equation \eqref{eq:skew}, the  dynamics on $\Pent^1(\ell)$ is a $\mathbb{U}$-extension of the one on
$\Pent^0(\ell)$. 
 Theorem~\ref{thm:dynamics_pentagons} implies that if $\nu$ generates $\Gamma$, 
  any ergodic $\nu$-stationary measure 
 $\tilde \mu$ on 
$\Pent^1(\ell)$ either projects to a finite orbit or to $\vol_{\Pent^0(\ell)}$. 
In the latter case, Proposition~\ref{pro:ergodic_pent1} and Lemma~\ref{lem:furstenberg} imply that 
$\tilde\mu = \vol_{\Pent^1(\ell)}$. 
Assume now that   
\begin{equation}
(\pi_{\Pent^0(\ell)})_*\tilde \mu= \frac{1}{\vert F\vert}\sum_{x\in F}\delta_x
\end{equation} 
for some  finite orbit $F$. In this situation we can also use 
   Lemma~\ref{lem:furstenberg}   to classify stationary measures. 
 Restricting to a finite index subgroup, endowed with the induced measure 
(see~\cite[Chap. 5]{benoist-quint_book}) reduces  the problem to the case  
where $F = \set{x_0}$. Then 
the classification follows from Lemma~\ref{lem:circle-dynamics-obvious}.
 
The second statement of the theorem follows similarly, using 
Theorem~\ref{thm:finite_orbits_pentagons} instead of~\ref{thm:dynamics_pentagons}.   
\end{proof}

\subsection{Dynamics on $\Pent(\ell)$}
We can finally derive some information about the dynamics of random foldings on $\Pent(\ell)$, which 
is a $\R^2$-extension of the dynamics on $\Pent^1(\ell)$. 
Indeed $\Pent(\ell)$ can be identified to $\Pent^1(\ell)\times \R^2$ 
by choosing a preferred vertex, say $a_0$, and translating it to the origin.
Doing so, we obtain a diffeomorphism 
\begin{equation}\label{eq:pent}
\Pent(\ell)\ni (a_i) \longmapsto (\tilde{x}, a_0)=((v_i), a_0)\in \Pent^1(\ell)\times \R^2, 
\text{ where } v_i  = \overrightarrow{a_ia_{i+1}}.
\end{equation}

We already introduced the involution  $s_i$ in \S~\ref{par:pentagons}, which descends to 
${\tilde{\sigma}}_{i, i+1}$ (indices modulo $5$) on $\Pent^1(\ell)$ and to $\sigma_{i, i+1}$ on $\Pent^0(\ell)$. 
With the identification given by Equation~\eqref{eq:pent}, we obtain:
\begin{equation}\label{eq:si}
\begin{cases}
s_i((a_i)) = ({\tilde{\sigma}}_{i, i+1}({\tilde{x}}), a_0)  \text{ for } i \neq 4\\
s_4((a_i)) = \lrpar{{\tilde{\sigma}}_{40}({\tilde{x}}), a_0 + {\tilde{\sigma}}_{40}(v_4) - v_4)} 
\end{cases}
\end{equation}
By definition the group $\Gamma$ is the group generated by the involutions $s_i$ (we are slightly abusing notation here). 

To find an involution descending to ${\tilde{\sigma}}_{ij}$ when $j\neq i+1$, different options can be chosen, depending on the vertices that remain  fixed under the involution. 
First, observe that if $j\neq i+1$ we can always choose $i,j$ so that $j = i+2$ modulo $5$ 
(so the vectors $v_i$ and $v_{i+2}$ will change while the others remain fixed). 
In analogy with the case of $s_i$, we define   $r_i$ such that $r_i(a_i) = a_i$ and 
$[r_i(P)] = \sigma_{i, i+2}([P])$, so that $r_i$ fixes $a_i$, $a_{i+3}$ and $a_{i+4}$ and moves 
$a_{i+1}$ and $a_{i+2}$ (another option would have been to leave   $a_{i+2}$ fixed). It can be expressed  in  coordinates in $\Pent(\ell)$ as in~\eqref{eq:si}. We define $\Gamma^{\rm ext}$ acting 
on $\Pent(\ell)$ by
$\Gamma^{\rm ext} := \bra{s_i, \ r_i, \ i=0\ldots 5}$. 
 
\begin{figure}[h] 
\includegraphics[width=10cm]{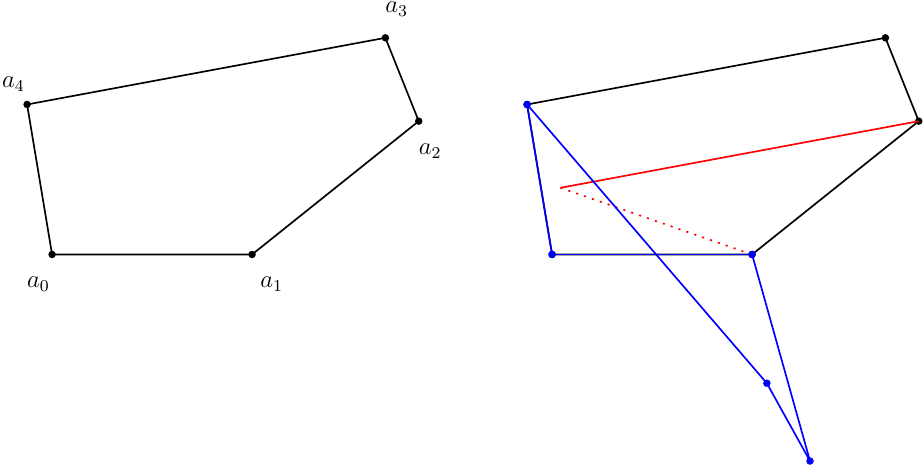}
\caption{
{\small{The black pentagon  
on the left  is folded to the blue pentagon   
on the right by the lift $r_1$ of $\sigma_{1,3}$. The red segment is parallel to $[a_3,a_4]$ and the red dotted segment gives the direction along which the vectors $v_1$ and $v_3$ are reflected.}}}
\end{figure}
 
In \cite{benoist-hulin, benoist-hulin2}, the authors study in detail how a quadrilateral drifts in the plane
under successive foldings. 
Fix a   probability measure $\nu$ on $\Gamma^{\rm ext}$. From a pentagon 
$P_0=(a_i)$ with $a_0 = (0,0)$ and a sequence $\omega = (f_n) \in (\Gamma^{\rm ext})^\N$ of folding instructions, we obtain the following random sequence 
of pentagons $P_n\in \Pent(\ell)$:
\begin{equation}
P_n(\omega) = f_{n-1}\cdots f_0(P_0).
\end{equation} 
Note that the parameter space for the starting point is $\Pent^1(\ell)$;  
 in other words, we identify the set of pentagons in $\Pent (\ell)$
 whose first vertex is at the origin $(0,0)$ with $\Pent^1(\ell)$. 
 In the next lemma, we also use $(0,0)$ as a reference point to study the drift of $P_n(\omega)$ in $\R^2$.

\begin{pro}~
\begin{enumerate}
\item Let $\ell\in \R^5$ be a general parameter. Assume that $\nu$ satisfies the moment condition (M) and generates $\Gamma$. 
Then, for $\vol_{\Pent^1(\ell)}$-almost every $P_0$, the linear drift of the sequence $(P_n)$ vanishes, that is
$$\lim_{n\to\infty} \unsur{n} \dist( P_n(\omega), (0,0))   = 0$$
for $\nu^\N$-almost every $\omega$.
\item Assume now that $\ell$ is very general, $\nu$ satisfies the moment condition ($\mathrm{M}_+$), and $\nu$ generates $\Gamma^{\rm ext}$. Then the same conclusion holds for 
every $P_0\in \Pent^1(\ell)$.
\end{enumerate}
\end{pro}

\begin{proof}
To study the drift it is enough to study the location of the point $a_0$. 
For this we use the coordinates given by the identification 
$\Pent(\ell) \simeq\Pent^1(\ell)\times \R^2$    and 
write 
$P_n(\omega) = ([P_n(\omega)], a_0(P_n(\omega)))$. From Equation~\eqref{eq:si} and its analogue for the 
$r_i$,  we see that there exists a function 
$w:\Gamma^{\rm ext}\times \Pent^1(\ell)\to \R^2$   such that 
$a_0(f(P)) = a_0(P) + w(f, [P])$. This function is continuous and is a cocycle: $w(fg, [P])=w(f, g[P])+w(g,[P])$. We then obtain 
\begin{equation}\label{eq:cocycle1}
a_0(P_n(\omega)) = \sum_{k=0}^{n-1} w(f_k, [P_k(\omega)]). 
\end{equation}
For Assertion~(1), we apply the Birkhoff ergodic theorem to the skew product 
\begin{equation}
F:(\omega, x)\mapsto (\sigma(\omega), f_\omega(x))
\end{equation}
(where $\sigma$ is the shift on $(\Gamma^{\rm ext})^\N$) and to  
 the ergodic measure $\nu^\N\times \vol_{\Pent^1(\ell)} $. This shows that for 
 $\vol_{\Pent^1(\ell)}$-almost every $P_0$ and  $\nu^\N$-almost every $\omega$, 
 \begin{equation}\label{eq:cocycle2}
 \unsur{n} \sum_{k=0}^{n-1} w(f_k, [P_k(\omega)]) \underset{n\to\infty}{\longrightarrow} 
 \int w(\gamma,P) \; \d\nu(\gamma)  \d\vol_{\Pent^1(\ell)} (P).
 \end{equation}
Now for every $\gamma\in \Gamma^{\rm ext}$, the   invariance of $\vol_{\Pent^1(\ell)}$ under rotations implies that 
 \begin{equation}
\int w(\gamma,P)   \d\vol_{\Pent^1(\ell)} (P) = 0,
\end{equation}
 and the result follows. 

For Assertion~(2), $\ell$ being very general, we can assume that $\vol_{\Pent^1(\ell)}$ is the unique stationary measure on $\Pent^1(\ell)$ (see Theorem~\ref{thm:dynamics_pentagons2}). The argument will be the same as for Assertion~(1), except that we shall need a more sophisticated limit theorem that makes use of this unique ergodicity.
As already observed,  $w$ is a cocycle and, by unique ergodicity, it has a unique 
average  in the sense of \cite[\S 3.3.2]{benoist-quint_book}. The analogue of 
 formula~\eqref{eq:si} for extended foldings  implies that there exists $C(\ell)\leq 2 \sum \ell_i$ such that 
  that for any $\gamma\in \Gamma^{\rm ext}$
\begin{equation}
w_{\sup} (\gamma):= \sup_{[P]\in \Pent^1(\ell)} \norm {w(\gamma, [P])} \leq C(\ell) \length(\gamma),
\end{equation}
where the length $\length(\gamma)$ is relative to the given generators of $\Gamma^{\rm ext}$. The moment condition  implies that  
\begin{equation}
\int_{\Gamma^{\rm ext}} w_{\sup} (\gamma) \d\nu(\gamma) <\infty, 
\end{equation}
thus we can apply the Law of Large Cocycles \cite[Thm 3.9]{benoist-quint_book}, which 
 completes the proof. 
\end{proof}

It is natural to ask for a better understanding of the asymptotic behavior of $\dist( P_n(\omega), (0,0))$. 
Intuitively, by Equation~\eqref{eq:cocycle1}, the first vertex $a_0(P_n(\omega))$ 
should behave like a random walk in the plane, the steps of which are random vectors distributed according to some explicit measure supported by a bounded disk. 
Under appropriate non-degeneracy properties of this measure,  
such a random walk escapes   
 with speed $O(\sqrt{n})$ and, after rescaling, converges to a Brownian motion (for some positive definite covariance form). 
Numerical experiments indicate that $\dist(P_n(\omega),(0,0))$ indeed behaves like  $\sqrt{n}$.

\begin{que}
Does $a_0(P_n(\omega))$ satisfy a central limit theorem?  Does $\unsur{\sqrt{n}} a_0(P_n(\omega))$ converge to a Brownian motion? 
\end{que}

As far as we know, this question is already open for random foldings of quadrilaterals.  
An even simpler variant is to take a triangle and 
reflect it randomly along one of its sides. This last problem 
falls into the setting of random iteration of Euclidean isometries 
in which case the result is known (see e.g. \cite{adahl-melbourne-nicol} and references therein).  
To establish such a result, it is likely that some estimates would be needed for the speed 
of the ergodicity  of the base dynamics in $\Pent^0$, as discussed in \S~\ref{subs:UE} (cf. \cite[Chap. 11 and 12]{benoist-quint_book}).

\newpage
\part{Ergodic theory of Blanc's examples}\label{part:ergodic_blanc}

\section{Action on cohomology in the complex surface}\label{sec:invariant_curves_parabolic}

\subsection{Setting and volume form}\label{par:setting_and_omega} Consider a smooth cubic curve $C\subset \P^2_\C$ and a point $q\in C$. 
Denote by  $\sigma_q\colon \P^2_\C\dasharrow \P^2_\C$ the birational involution defined by the following properties:  it fixes $C$ pointwise
and it preserves the pencil of lines through $q$. If $L$ is a general line through $q$, the restriction of $\sigma_q$ to $L$ is defined as follows: besides $q$, $L$  intersects $C$ in two other points $p$ and $p'$; identifying $L\setminus \set{p, p'}$ to $\P^1\setminus \set{0, \infty}$, with affine coordinate $z$, $\sigma_{q\, \vert L}$ corresponds to $z\mapsto -z$. 
By definition, $\sigma_q$ is the {\bf{Jonquières involution}} associated to the pencil of lines through $q$ and fixing $C$ pointwise. 

Now, fix  a positive integer $k$, and $k$ distinct points $q_i$ on $C$. Each  $q_i$ gives rise to a 
Jonquières involution $\sigma_i:=\sigma_{q_i}$. It admits five base points, namely $q_i$ and the four 
points $p_{i,j}\in C$ such that the line $(q_ip_{i,j})$ is tangent to $C$ at $p_{i,j}$. One of the $p_{i,j}$ 
may be infinitely near (above $q_i$) if $q_i$ is an inflexion point, but for simplicity we shall assume that: 
\begin{itemize}
\item[(Hyp1)] none of the $q_i$ is an inflexion point of $C$;
\item[(Hyp2)] the points $q_i$ and $p_{i,j}$ for $1\leq i \leq k$ and $1\leq j\leq 4$ are pairwise distinct.
\end{itemize}
In fact, (Hyp2) implies (Hyp1) since otherwise one of the $p_{i,j}$ would coincide with $q_i$ (as a point of the plane, {i.e.}\  $p_{i,j}$ would be infinitely near $q_i$). We shall denote by $X$ the rational surface obtained by blowing up the $5k$ points $q_i$ and $p_{i,j}$ and by $\pi\colon X\to \P^2_\C$ the natural projection.   

We also consider the surface $X_i$ obtained by fixing $i$ and blowing 
 up the base points $q_i$ and $p_{i,j}$ of $\sigma_i$, creating five exceptional divisors $E(q_i)$, and $E(p_{i,j})$ for $1\leq j\leq 4$. Doing so, we get a birational morphism $\pi_i\colon X_i\to \P^2_\C$, together with a distinguished basis for its N\'eron-Severi group: 
\begin{itemize}
\item $\bfe_0$ will denote the class of the total transform of a line $L \subset \P^2_\C$;
\item $\bfe_{q_i}$ the class of $E(q_i)$;
\item and $\bfe_{p_{i,j}}$ the class of $E(p_{i,j})$. 
\end{itemize}
This basis is orthogonal for the intersection form,  $\bfe_0^2 =1$, and 
  $\bfe_{q_i}^2  = \bfe_{p_{i,j}}^2 = -1$. 
Abusing slightly,  the same notation will be used for 
 the classes of (the total transform of) these curves in the surface $X$. 
 
The involution $\sigma_i$ lifts to an involutive automorphism $\tilde{\sigma_i}$ of $X_i$, and then all the 
  $\tilde{\sigma_i}$    lift to automorphisms of $X$, which we still denote by  $\tilde{\sigma_i}$. 
Indeed, since ${\tilde{\sigma_i}}$ fixes the strict transform of $C$ in $X_i$ pointwise, it preserves the exceptional divisors obtained by blowing up the $q_\ell$ and $p_{\ell,j}$ for $\ell\neq i$.  Acording to~\cite{Blanc:Michigan}, the subgroup $\Gamma\subset \Aut(X)$ generated by the $\tilde{\sigma_i}$ is a free product of $k$ copies of $\Z/2\Z$, {i.e.}\  
\begin{equation}
\Gamma\simeq \Z/2\Z\ast \cdots \ast \Z/2\Z,  \end{equation}
acting faithfully on the Néron-Severi group $\NS(X)$. 
Since the canonical bundle of $\P^2(\C)$ is ${\mathcal{O}}(-3)$ and we only blow-up points of $C$, we obtain the following properties.

\smallskip

\noindent{(1)} The strict transform $C_X$ of $C$ in $X$ is fixed pointwise by the group $\Gamma$. Its self-intersection is equal to 
\begin{equation}
C_X^2=9-5k\, ,
\end{equation}
so it is negative when $k\geq 2$. The strict transform $C_{X_i}$ of $C$ in $X_i$ has self-intersection $C_{X_i}^2=4$.

\smallskip

\noindent{(2)}  Assume $k\geq 2$. The curve $C_X$ is the unique member of the linear system $\vert C_X\vert$, 
and its cohomology class is equal to the anticanonical class $-k_X$.   In particular, $C_X$ is invariant under the action of $\Aut(X)$.
More precisely, $C$ being a cubic curve, there is a meromorphic $2$-form  
$\Omega_0$ on $\P^2_\C$ that does not vanish 
and has a simple pole along $C$. Now,  since we only blow up smooth points of $C$, an easy local computation shows that
 the pull back 
\begin{equation}
\Omega_X:=\pi^*\Omega_0
\end{equation} is a meromorphic $2$-form that does not vanish and has a simple pole along $C_X$. 
We shall fix such a form $\Omega_X$; it is almost unique: on $X$, any non-vanishing meromorphic $2$-form is proportional to $\Omega_X$.

\smallskip

\noindent{(3)}  The total volume of the singular volume form $\Omega_X\wedge\overline{\Omega_X}$ is infinite: 
\begin{equation}
\int_X \Omega_X\wedge\overline{\Omega_X} = +\infty. 
\end{equation}
We set 
\begin{equation}
\vol_X^\infty=\Omega_X\wedge \overline{\Omega_X}, 
\end{equation}
which we identify to a positive measure on $X$. 

\smallskip

\noindent{(4)}  Let $\Jac_\Omega \colon \Aut(X)\to \C^*$ denote the homomorphism such that  
\begin{equation}
\forall f\in \Gamma, \ 
f^*\Omega_X=\Jac_\Omega(f)\Omega_X 
\end{equation} 
(cf. Remark~\ref{rem:pm1}). 
We have 
\begin{equation}\label{eq:jacobian-is-1}
\Jac_\Omega(\Gamma)=\set{1}.
\end{equation} 
To see this, consider the involution $\sigma_1$ and the pencil of lines through $q_1$. Let $r$ be a point of $C$ such that $(q_1r)$ intersects $C$ transversely; then, $\sigma_1$ determines a germ of diffeomorphism fixing $r$. Since this germ has order $2$ and fixes $C$ pointwise, we can linearize it(\footnote{Indeed, its differential $L$ at the origin
 in the system of coordinates mentioned just below is $L(x,y) = (x,-y)$, and 
 $\psi \circ \sigma_1 = L\circ\psi$, where $\psi = \unsur{2}(\id + L\circ \sigma_1)$}). 
 Since $C$ is  pointwise fixed, there is a small euclidean neighborhood of $r$ and holomorphic coordinates $(x,y)$ on it, in which $r=(0,0)$, 
$\sigma_1(x,y)=(x,-y)$ and $C=\set{y=0}$. 
In these coordinates, the form $\Omega_X$ is equal to 
$$
\varphi(x,y)\frac{dx\wedge dy}{y}
$$
for some non-vanishing holomorphic function $\varphi$. Then, $\Jac_\Omega(\sigma_1)=\varphi(x,-y)/\varphi(x,y)$. Evaluating this ratio at $r$, we get $\Jac_\Omega(\sigma_1)=1$. Thus, $\Jac_\Omega(\sigma_i)=1$ for $i=1,\ldots k$, and $\Jac_\Omega(\Gamma)=\set{1}$.

\smallskip

\noindent{(5)}  If $C$ and the $q_i$ are defined over $\overline{\Q}$, then $X$ and $\Gamma$ are also defined over $\overline{\Q}$. Indeed, if $K$ is the number field over which $C$ and the $q_i$ are defined then, for each $i$, the $p_{i,j}$ are defined over some quartic extension of $K$; thus, $X$ and $\Gamma$ are defined over some extension $L$ of $K$ of degree $[L: K]\leq 4 k$.

\subsection{Finer description of the involution $\sigma_i$} \label{par:description_of_sigma_i}
\subsubsection{Invariant fibration}\label{par:rational fibration}
For each $i=1, \ldots, k$, the pencil of lines through $q_i$ gives rise to a rational fibration 
\begin{equation}
\eta_{q_i}\colon X_i\to \P^1_\C.
\end{equation}
The singular fibers of $\eta_{q_i}$ correspond to the total transform of the lines $(q_ip_{i,j})$ (each of them is made of two rational curves).
In $X$, we obtain $k$ fibrations $\eta_{q_i}\colon X\to \P^1_\C$, each of them with 
$5k-1 = 4+5(k-1)$ 
singular fibers.

Let $L_{i,j}\subset X_i$ be the strict transform of the line $(q_ip_{i,j})$. 
Since $(q_ip_{i,j})$ is tangent to $C$ at $p_{i,j}$, the three curves $C_{X_i}$, 
$E({p_{i,j}})$, and $L_{i,j}$ have a unique common point 
\begin{equation}
{\tilde p_{i,j}}:=C_{X_i}\cap E({p_{i,j}}) =L_{i,j}\cap E({p_{i,j}})
\end{equation} 
in $X_i$. 
The involution  $\tilde\sigma_i$ permutes $E({p_{i,j}})$ and $L_{i,j}$ and fixes this intersection point ${\tilde p_{i,j}}$.  
This can be seen as follows: if $L$ is the strict transform of a general line though $q$, then $\tilde \sigma_i^*([L]) = [L]$. It follows that  $\tilde \sigma_i^*( [L_{i,j}]+ [E(p_{i,j})]) = 
[L_{i,j}]+ [E(p_{i,j})]$. Now,  $\sigma_i$ contracts $\pi_i(L_{i,j})$ onto $p_{i,j}$, so 
$ {\tilde \sigma_i}( L_{i,j}) = c  E(p_{i,j})$ for some integer $c\geq 1$;  and we infer that $c=1$ because $\tilde \sigma_i^*$ is an involution. 
Thus, $ {\tilde \sigma_i}( L_{i,j}) = E(p_{i,j})$ and $ {\tilde \sigma_i}( E(p_{i,j}))=L_{i,j}$

\subsubsection{Action on $E_{p_{i,j}}$}\label{par:action_of_sigma_I} Let $x$ be a point of $E({p_{i,j}})$. A first possibility is that  ${\tilde \sigma_i}(x)$ belongs to 
$ \Exc(\pi_i)$ (more precisely  ${\tilde \sigma_i}(x)\in E(q_i)\cup E(p_{i,j})$). This happens precisely when
\begin{itemize}
\item  $\tilde\sigma_i (x)=x$, and in that case $x={\tilde p_{i,j}}=L_{i,j}\cap E({p_{i,j}})$;
\item or $\tilde\sigma_i (x)$ is the intersection point of $L_{i,j}$ and $E(q_i)$, and in that case $x$ is the intersection of $E(p_{i,j})$ with the conic $D_i$ (see below \S~\ref{par:action_of_sigma_on_NS}).
\end{itemize}
Otherwise,  $\tilde\sigma_i (x)$ is a point of $L_{i,j}$ that does not belong to 
 the exceptional set $ \Exc(\pi_i)$.
  
\subsubsection{Blow-ups of points of $C$}\label{par:action_of_sigma_II} Now, let $r\in C$ be distinct from the base points of $\sigma_i$. The line $(qr)$ is ${\sigma_i}$-invariant, and ${\sigma_i}$ acts as $z\mapsto -z$ on it, by fixing $0 \simeq r$ and $\infty \simeq r'$, where $r'$ is the third point of intersection of $(qr)$ with $C$. Thus, when   $r$ is blown up, 
 the action of $\tilde\sigma_i$ on the exceptional divisor $E(r)$ is an involution with exactly two fixed points: a fixed point corresponding to the tangent line $T_rC$, and a fixed point corresponding to the line $(qr)$ (or more precisely to $T_r(qr)$). 

\begin{rem} By (Hyp2), $q_i$ does not coïncide with $q_l$ or a $p_{l,j}$ when $l\neq i$. Thus, 
the intersection point between $E({q_l})$ and 
the strict transform of $(q_lp_{l,j})$ is never fixed by $\tilde\sigma_i$: it is mapped to another point of $E({q_l})$. \end{rem}

 \subsubsection{Action on Néron-Severi, and the conic $D_i$}\label{par:action_of_sigma_on_NS} Acording to~\cite{Blanc:Michigan}, 
the action of $\tilde{\sigma}_i$ on the Néron-Severi group of $X_i$ is given, in the basis $(\bfe_0, \bfe_{q}, \bfe_{p_{1}}, \ldots, \bfe_{p_{4}})$, by the matrix 
\begin{equation}\label{eq:matrix_involution}
\left( \begin{array}{rrrrrr} 3 & 2 & 1 & 1 & 1 & 1  \\ 
-2 & -1 & -1  & -1  & -1  & -1  \\ 
-1  & -1 & -1 & 0 & 0 & 0 \\
-1  & -1 & 0 & -1 & 0 & 0 \\ 
-1  & -1 & 0 & 0 & -1 & 0 \\ 
-1  & -1 & 0 & 0 & 0 & -1   \end{array} \right).
\end{equation}
For instance, the second column means that the exceptional divisor $E(q_i)$ is mapped to the strict transform ${\tilde D_i}$ of a plane conic $D_i$ which goes through the points $q_i$ and $p_{i,j}$ with multiplicity $1$; its class in $\NS(X_i;\Z)$ is  $2\bfe_0-\bfe_{q}-\bfe_{p_1}-\ldots - \bfe_{p_4}$. This conic intersects $C$ in $6$ points, counted with multiplicity, and ${\tilde D_i}$ intersects $C_{X_i}$ in exactly one point, which must be 
 fixed under the action of~$\tilde\sigma_i$. On the other hand, since $\tilde \sigma_i$ fixes 
 $C_{X_i}$ pointwise, the point 
\begin{equation}
{\tilde q_i}:= E(q_i)\cap C_{X_i} 
\end{equation} must be  fixed under the action of $\tilde\sigma_i$ 
(\footnote{We abuse notation for convenience and 
   use similar notations (like $L_{i,j}$, $\tilde p_{i,j}$, $\tilde q_i$) for objects defined in $X$ and  $X_i$. For instance, $\tilde q_i$ also stands for ${\tilde q_i}:= E(q_i)\cap C_{X}$.}). Since $\tilde\sigma_i(E(q_i))={\tilde D_i}$, we conclude that ${\tilde D_i}\cap C_{X_i}={\tilde q_i}$, thus $D_i$ is tangent to $C$ at $q_i$. 
This argument  also shows that the only points of $E(q_i)$ which are 
mapped into the exceptional set of $\pi_i$ (resp.\ $\pi$) 
are  ${\tilde q_i}$ and the intersection points of   $E(q_i)$ with the  $L_{i,j}$: the point $E(q_i)\cap L_{i,j}$ is mapped to $L_{i,j}\cap E(p_{i,j})$, a point which corresponds to the tangent direction of $D_i$ at $p_{i,j}$ (cf. \S~\ref{par:action_of_sigma_I}).  

We already stated without proof the fact 
that $\Gamma$ is isomorphic to the free product of the $\tilde \sigma_i$.  Blanc obtains this result by 
using~\eqref{eq:matrix_involution} and proving that  the action of $\Gamma$ on $\NS(X;\Z)$ is faithful and its image is such a free product $\Z/2\Z\ast \cdots \ast \Z/2\Z$. Thus we get:

\begin{lem}
When $k\geq 3$, the group $\Gamma\subset \Aut(X)$ is non-elementary.
\end{lem}

\subsection{Invariant curves ($k\geq 3$)} 
With the matrix from Equation~\eqref{eq:matrix_involution} at hand, we deduce that all
 fixed points of
 ${\tilde{\sigma_i}}^*$ acting on $\NS(X_i)$ (resp.  $\NS(X_i;\R)$) are of the form 
\begin{equation}\label{eq:u_fixed}
u=d\bfe_0 -(d-2m) \bfe_{q_i} - m \sum_{j=1}^4 \bfe_{p_{i,j}}
\end{equation}
for some pair of integers $(d,m)$ (resp. pair of real numbers $(d,m)$). From now on 
we set
\begin{equation}
 \Sigma_i  := \sum_{j=1}^4 \bfe_{p_{i,j}} \in \NS(X_i;\Z) \quad ({\text{resp.}}\;  \in\NS(X;\Z)).
\end{equation}

\begin{lem}\label{lem:no_invariant_curve}
 If $k\geq 3$ and (Hyp1) and (Hyp2) are satisfied, the only  (reduced) effective curve $U\subset X$ which is invariant under the action of $\Gamma$ is the curve $C_X$.
\end{lem}

\begin{proof} 
Let  $u\in \NS(X)$ be the class of an invariant curve $U$. Then by Equation~\eqref{eq:u_fixed}
  its expression in the basis $(\bfe_0, \bfe_{q_1},  \ldots,  \bfe_{q_k}, \bfe_{p_{1,1}}, \ldots, \bfe_{p_{1,4}}, \bfe_{p_{2,1}}, \ldots \bfe_{p_{k,4}})$ is of the form 
\begin{equation}
u= d\bfe_0- (d-2m_1) \bfe_{q_1} - \cdots - (d-2m_k) \bfe_{q_k} - m_1 \Sigma_1 - \cdots -m_k\Sigma_k
\end{equation}
for some integers $d$ and $m_i$.
The strict transform of a general line through $q_i$, whose class is 
$\bfe_0 - \bfe_{q_i}$,  must intersect $U$ non-negatively. This implies  that 
$m_i \geq 0$. 
Similarly, $d\geq 0$ because a general line intersects $U$ non-negatively.

The equality $[C_X] = 3\bfe_0 - \bfe_{q_1} - \Sigma_1-\cdots - \bfe_{q_k} - \Sigma_k$ implies 
\begin{equation}
U\cdot C_X= -(k-3)d - 2(m_1+\cdots +m_k).
\end{equation} 
Now, assuming that $C_X$ is not an irreducible component of 
$U$, we obtain 
\begin{equation}
0\leq  -(k-3)d - 2(m_1+\cdots +m_k).
\end{equation}
Since $k\geq 3$, we infer that $k=3$ and $ m_i=0$ for all $i$, therefore 
 $u=d\bfe_0-d(\bfe_{q_1}+\cdots+\bfe_{q_k})$. From this it follows that  the intersection of $U$ 
 with the strict transform of a general line through $q_i$ is 
  $0$ for all $i$. This means that $U$ is the strict transform of  
  a plane curve of degree $d$ and is 
  mapped to a point by each of the fibrations $\eta_{q_i}$ (see \S \ref{par:description_of_sigma_i}). 
  This is a contradiction, so $C_X$ must be a component of $U$; in particular, $\overline{U\setminus C_X}$ is also invariant.  
  Repeating this argument with $U\setminus C_X$ finishes the proof. 
\end{proof}

\subsection{Parabolic automorphisms}\label{subs:parabolic_blanc}
 From now on, let us add the following hypothesis
\begin{itemize}
\item[(Hyp3)] For any $i\neq j$, the line $(q_iq_j)$ does not contain any of the points $p_{l,m}$  for $1\leq l\leq k$, $1\leq m\leq 4$, nor   $q_n$ for $1\leq n \leq k$ and $n\neq i,j$.
\end{itemize}
The strict transform $M_{i,j}$ of the line $(q_iq_j)$ is invariant under $\Gamma_{i,j}:=\langle \sigma_i,\sigma_j\rangle$. Furthermore 
 this line $(q_iq_j)$ intersects $C$ in a third point and (Hyp3) assures that   this point is none of the points that we blow-up. 

Assume for a few lines that $k=2$, so that $X$ is obtained by blowing-up the points $q_r$ and $p_{r,j}$ for $r=1,2$. Then 
$[M_{1,2}]=\bfe_0-\bfe_{q_1} - \bfe_{q_2}$ and $M_{1,2}\cdot C_{X}=1$. The class $u_{1,2}$ of $C_X+ M_{1,2}$ satisfies 
\begin{equation}
u_{1,2}=4\bfe_0-2\bfe_{q_1}-2\bfe_{q_2}-\Sigma_1-\Sigma_2;
\end{equation}
 it is invariant under $\Gamma_{1,2}$, {i.e.}\ under $\tilde{\sigma_1}^*$ and $\tilde{\sigma_2}^*$, and it is 
isotropic. 

\begin{lem}\label{lem:parabolic_gij}
Assume $k=2$ and the three hypotheses (Hyp1-Hyp3) are satisfied.
\begin{enumerate}
\item  If $v$ is the class of an $\R$-divisor with non-negative self intersection and $v$ is invariant under $\tilde{\sigma_1}^*$ and $\tilde{\sigma_2}^*$ then $v$ is a multiple of $u_{1,2}$. 
\item The composition $g_{1,2}=\sigma_1\circ \sigma_2$ is a parabolic automorphism of $X$ that preserves the isotropic class $u_{1,2}$.
\item The invariant genus $1$ pencil of $g_{1,2}$ is given by the pencil of plane quartic curves going 
through the $p_{i,j}$ with multiplicity $1$ and through $q_i$ with multiplicity $2$, for $i=1,2$. 
\end{enumerate}
\end{lem}

Before proving this lemma, let us state the following immediate corollary. 

\begin{cor}
If $k\geq 2$ and the hypotheses (Hyp1-Hyp3) are satisfied, then
\begin{enumerate}
\item  each of the automorphisms $g_{i,j}=\tilde\sigma_i\circ \tilde\sigma_j$, $i\neq j$, is parabolic;
\item  the invariant genus $1$ fibration of $g_{i,j}$ corresponds to the linear system of plane quartics going through $q_i$ and $q_j$ with multiplicity $2$ and through the $p_{i,l}$ and $p_{j,l}$ with multiplicity $1$ (for $l=1, \ldots, 4)$. 
\end{enumerate}
\end{cor}

\begin{proof}[Proof of Lemma~\ref{lem:parabolic_gij}]
Every invariant class can be written 	as 
$v= d\bfe_0-(d-2m_1)\bfe_{q_1}-m_1\Sigma_1-(d-2m_2)\bfe_{q_2}- m_2\Sigma_2$. These classes form 
a $3$-dimensional
subspace of $\NS(X;\R)$ on which the intersection form is non-positive and degenerate: its kernel is 
generated by $u_{1,2}$. So, if $v^2\geq 0$, $v$ is proportional to $u_{1,2}$; more precisely, $m_1=m_2$ and $v=m_1u_{1,2}$.

In the group $\langle \tilde{\sigma}_1, \tilde{\sigma}_2\rangle$, $g_{1,2}$ generates a cyclic, normal subgroup of index $2$. 
Thus, the fixed point set of $g_{1,2}^*$ in $\NS(X;\R)$ is invariant under $\langle \tilde{\sigma}_1, \tilde{\sigma}_2\rangle$. 
If $g_{1,2}$ were elliptic, this set of fixed points would intersect the set of  classes $v\in \Hyp_X\subset \NS(X;\R)$ on a non-empty convex subset $F$ of  the hyperbolic space $\Hyp_X$. 
This convex set $F$ would be invariant under the action of $\tilde\sigma_1^*$, and the involution $\tilde\sigma_1^*$ would have a fixed point in $F$; 
such a fixed point would also be fixed by ${\tilde \sigma_2}^*$ because ${\tilde \sigma_2}={\tilde \sigma_1}\circ g_{1,2}$; thus, it is fixed by $\langle \tilde{\sigma}_1^*, \tilde{\sigma}_2^*\rangle$, in contradiction with the first assertion. We conclude that $g_{1,2}$ is not elliptic. Since it preserves the isotropic class $u_{1,2}$, it is parabolic  \footnote{Alternatively, one can compute the product of the matrices for $\tilde\sigma_1^*$ and  $\tilde\sigma_2^*$, and check that some power of it is unipotent, but not the identity.}. 

Let us now prove the third property.
Set $\mathcal{M}_{1,2}=\mathcal{O}(C_X+M_{1,2})$ and $\mathcal{L}_{1,2}={\mathcal{O}}(M_{1,2})$ (viewed as line bundles or as sheaves on $X$). Since $X$ is rational we have $\chi (\mathcal{O}_X)=1$. 
Since $M_{1,2}$ is a smooth rational curve with $M_{1,2}^2=-1$, the Riemann-Roch theorem gives
\begin{equation}
h^0(X,\mathcal{L}_{1,2})-h^1(X,\mathcal{L}_{1,2})+h^2(X,\mathcal{L}_{1,2})=1.
\end{equation}
We have $h^0(X,\mathcal{L}_{1,2})=1$ because $M_{1,2}$ is irreducible and has negative self-intersection, and $h^2(X,\mathcal{L}_{1,2})=0$ by Serre duality; thus, $h^1(X,\mathcal{L}_{1,2})=0$.
Now, looking at the restriction of $\mathcal{M}_{1,2}$  to $C_X$, we get the long exact sequence 
\begin{equation}
H^0(X,\mathcal{L}_{1,2})\to H^0(X,\mathcal{M}_{1,2})\to H^0(C_X, \mathcal{M}_{1,2\vert C_X})\to H^1(X,\mathcal{L}_{1,2})\to \cdots
\end{equation}
By what we know of the $H^i(X,\mathcal{L}_{1,2})$ this gives 
\begin{equation}\label{eq:long_exact_quartic}
\C \to H^0(X,\mathcal{M}_{1,2})\to H^0(C_X, \mathcal{M}_{1,2\vert C_X})\to 0.
\end{equation}
Let us identify  $C_X$ to  $C$ via the projection $X\to \P^2_\C$ and fix  an inflexion point $o$ of $C$; the divisor obtained by intersecting a line with $C$ is equivalent to $3o$. 
In restriction to $C_X\simeq C$, $\mathcal{M}_{1,2}$ is given by the divisor $12o - 2q_1-2q_2-p_{1,1}- \cdots -p_{2,4}$, of degree~$0$. 
According to \S~\ref{par:action_of_sigma_on_NS}, there is a conic  $D_1$ tangent to $C$ at $q_1$ and passing through the $p_{1,j}$; thus, $2q_1+p_{1,1}+ \cdots +p_{1,4}=6o$ on $C$; similarly, 
$2q_2+p_{2,1}+ \cdots +p_{2,4}=6o$. Hence $12o - 2q_1-2q_2-p_{1,1}- \cdots -p_{2,4}=0$, which means that $\mathcal{M}_{1,2\vert C_X}$ is the trivial line bundle ${\mathcal{O}}(C_X)$. Thus,  $h^0(C_X, \mathcal{M}_{1,2\vert C_X})=1$ and Equation~\eqref{eq:long_exact_quartic} gives $h^0(X, \mathcal{M}_{1,2})=2$. 

In other words, the linear system $\vert C_X+M_{1,2}\vert$ is a pencil of curves. The general member $D$ of this linear system is irreducible, because 
otherwise $C_X$ or $M_{1,2}$ would be a fixed component, then we could write $D=D_0+C_X$ (resp. $D_0+M_{1,2}$) for some movable curve $D_0$, but since $C_X^2 = M_{1,2}^2=-1$,
this curve $D_0$ would simultaneously satisfy $D_0^2=-1$, a contradiction. Now, since the self-intersection of $C_X+M_{1,2}$ is $0$, the elements of $\vert C_X+M_{1,2}\vert$ are disjoint and form a fibration. Finally, 
  since they intersect trivially $C_X$ and $C_X  = -k_X$, the genus formula shows that they have genus $1$. \end{proof}
 
\section{Finite orbits and invariant measures}

We keep notation as in the previous sections. 
In \S~\ref{subs:finite_orbits},  we  prove that when $k\geq 4$, under general assumptions on the the $q_i$,
every orbit of $\Gamma$ outside $C_X$ is infinite. 
Then we discard the possibility of  $\Gamma$-invariant measures  in \S~\ref{subs:invariant_measures}. 
This relies heavily on the results of \cite{finite_orbits, invariant}.

\subsection{Finite orbits outside $C_X$: a finiteness result}  
By a {\bf{generalized Kummer surface}}, we mean a desingularization of a quotient $A/G$ where $A$ is an Abelian surface, $G$ is a subgroup of $\Aut(A)$, and the set of points $x\in A$ with a non-trivial stabilizer in $G$ is a finite subset $F_G$ of $A$ (see~\cite{invariant, cantat-dupont}).

 \begin{pro}
 The surface $X$ is not a generalized Kummer surface.  \end{pro}
  
 \begin{proof} Assume, by contradiction, that $X$ is a generalized Kummer surface. Then,  there exists a birational morphism $\e\colon X\to X_0$ onto a singular surface $X_0$, an Abelian surface $A$, and a finite subgroup $G$ of $\Aut(A)$ as above, such that $X_0$ is isomorphic to $A/G$; the singularities of $X_0$ correspond to the finite set $F_G\subset A$ (see~\cite[\S 4]{finite_orbits}). 
 In particular, the singularities of $X_0$ are of quotient type, and the genus~$1$ curve $C_X$ cannot be contracted by $\e$; set $C_{X_0}=\e(C_X)$ and let $C_A$ be the preimage of $C_{X_0}$ in $A$ under the quotient map $\eta\colon A\to A/G=X_0$. As a consequence, the meromorphic $2$-form $\Omega_X$ induces a meromorphic $2$-form $\Omega_{X_0}$ on the regular part of $X_0$, hence a meromorphic $2$-form $\eta^*\Omega_{X_0}$ on $A\setminus F_G$; this $2$-form has poles along $C_A$ because the quotient map $\eta$ is a local isomorphism on $C_A\setminus F_G$.
 Let $\Omega_A$ be a holomorphic $2$-form on $A$ with $\int_A \Omega_A\wedge \overline{\Omega_A} = 1$. Then, there is a holomorphic function $\varphi\colon  A\setminus F_G\to \C$ such that $\Omega_A=\varphi \eta^* \Omega_{X_0}$. By the Hartogs theorem, this function $\varphi$ extends holomorphically to $A$; and because it vanishes along $C_A$, it is identically zero. This contradiction shows that $X$ is not a Kummer surface.\end{proof}
 
 \begin{pro}\label{pro:finitely_many_finite_orbits_Blanc}
Suppose $k\geq 3$. If $C$ and the $q_i$ are defined over $\overline{\Q}$, then $\Gamma$ has at most finitely many finite orbits in $X(\C)\setminus C_X$. 
 \end{pro}

 \begin{proof}
By Lemma~\ref{lem:parabolic_gij}, the group $\Gamma$ contains parabolic elements. If the curve $C$ and the points $q_i$ are defined over $\overline{\Q}$, then $X$ and the ${\tilde{\sigma_i}}$ are all defined over some number field $K$. Since $X$ is not a Kummer surface, $(X,\Gamma)$ is not a Kummer group and Theorem~B of~\cite{finite_orbits} implies that the union $\Per(\Gamma)$ of all finite orbits of $\Gamma$ is not Zariski dense. Let $\overline{\Per(\Gamma)}$ be its Zariski closure: it is made of a one-dimensional part $\Per^1(\Gamma)$ that contains $C_X$, together with a sporadic finite set $\Per^0(\Gamma)$. Since $\Per^1(\Gamma)$ is $\Gamma$-invariant, Lemma~\ref{lem:no_invariant_curve} shows that $\Per^1(\Gamma)= C_X$. This concludes the proof.  \end{proof}

\subsection{Finite orbits outside $C_X$: non-existence} \label{subs:finite_orbits}

By (Hyp2), $q_i$ does not coincide with $q_l$ or a $p_{l,j}$ with $l\neq i$. Thus, the points of intersection between $E({q_l})$ and 
the strict transform $L_{l,j}$ of $(q_lp_{l,j})$ is never fixed by $\tilde\sigma_i$: it is mapped to another point of $E({q_l})$ (see \S~\ref{par:action_of_sigma_II}). To get more rigidity, let us add a stronger  hypothesis 
\begin{itemize}
\item[(Hyp4)] For $i\neq l$, the involution $\tilde \sigma_i$ acts on 
  $E({q_{l}})$, by mapping 
   the four points corresponding to the directions $(q_lp_{l,j})$ to four other points (i.e. $(\tilde \sigma_i)_* T_{q_{l}}(q_lp_{l,j})\neq T_{q_{l}}(q_lp_{l,j'})$ for any $i$, $l\neq i$, and $1\leq j,j'\leq 4$). 
\end{itemize}

To check that this condition is satisfied for a general choice of points, note that for fixed $q_l$ and varying $q_i$, $\tilde \sigma_i$ acts upon $E(q_l)$ as an involution fixing a fixed point (corresponding to the line $T_{q_l}C$) and a mobile one (corresponding to the line $(q_lq_i)$). Fix a coordinate on $E(q_l)$ in which $T_{q_l}C = \infty$. Then the induced  involution is of the 
form $z\mapsto -z+ 2c(q_i)$, where $c(q_i)$ is the coordinate of the mobile point. 
Given $a$ and $b$ in $E(q_l)$, taken among the points $L_{l,j}\cap E(q_l)$, the relation  $\tilde\sigma_i(a)=b$ reads $ a+b=2c(q_i)$; such a relation is not satisfied for a general choice of $q_i$. 

\begin{pro}\label{pro:no_finite_orbit_strong_version}
Let $C\subset\P^2$ be a smooth cubic curve defined over ${\overline{\Q}}$ and let $k\geq 4$ be an integer. Fix a $(k-1)$-tuple  $(q_1, q_2, \ldots, q_{k-1})$ of points in $C(\overline{\Q})$ satisfying hypotheses (Hyp1) to (Hyp4). Then for $q_k\in C(\overline{\Q})$ outside a finite set, the pair $(X,\Gamma)$ determined by $(q_1, q_2, \ldots, q_{k-1}, q_k)$ 
satisfies  (Hyp1) to (Hyp4)
and 
does not have any finite orbit, except for  the points of $C_X$, which are fixed.
\end{pro}

\begin{lem}  
Assume $k\geq 2$. If $x$ does not belong to $C_X$, there exists an element $f$ of $\Gamma$ such that 
$\pi(f(x))\notin C$.
\end{lem}

 \begin{proof} 
Observe that if  $\pi(\Gamma(x))\subset C$, then $\Gamma(x)\subset \Exc(\pi) \setminus C_X$. If 
$x\notin  \Exc(\pi)$, we are done. 
Otherwise $x$ belongs to some $E({p_{i,j}})$, or to some $E(q_i)$. 

In the first case, assume for concreteness that $x\in E({p_{1,1}})$ and set $y={\tilde \sigma_1}(x)$. By \S~\ref{par:action_of_sigma_I}, 
if  $y\in \Exc(\pi)$, then $y$ is   
 the intersection point of the strict transform $L_{1,1}$ of the line $(p_{1,1}q_1)$ either with 
 $E(q_1)$ or with $E(p_{1,1})$. The first possibility does not happen because   
 $L_{1,1}\cap E(q_1)$ is in $C_X$ and is fixed by ${\tilde \sigma_1}$ (see   \S~\ref{par:rational fibration}). 
In the second possibility, by (Hyp4), ${\tilde \sigma_2}(y)$ 
 is not on one of the strict transforms of the lines $(q_1p_{1,l})$, thus by 
 \S~\ref{par:action_of_sigma_I} again, $\tilde\sigma_1({\tilde \sigma_2}(y)) \notin \Exc(\pi)$ and we are done.

In the second case, assume for concreteness  that $x\in E(q_1)$ and, again, set $y={\tilde \sigma_1}(x)$. By  
 \S~\ref{par:action_of_sigma_on_NS}, $y$ belongs to  ${\tilde D_1}$. If  
 $y\in \Exc(\pi)$, then either
  $y\in {\tilde D_1}\cap E(p_{1, j})$ for some $j$, and we are back to the previous case, or 
   $y\in {\tilde D_1}\cap E(q_1)$ and then $x =   {\tilde q_1}\in C_X$ is fixed,   contradicting our assumptions. 
\end{proof}

\begin{proof}[Proof of  Proposition~\ref{pro:no_finite_orbit_strong_version}] 
We start by fixing a $k$-tuple of points $(q_1, \ldots, q_{k-1}, q_k')$ satisfying (Hyp1) to (Hyp4) and then deform $q_{k}'$ (hence the surface $X$) 
to achieve the desired finiteness.  
Consider the subgroup $\Gamma_3$ of $\Gamma$ generated by $\tilde\sigma_1$, $\tilde\sigma_2$, and $\tilde\sigma_3$. 
By Proposition~\ref{pro:finitely_many_finite_orbits_Blanc}, $\Gamma_3$ has only finitely many finite orbits outside $C_X$. 
By the previous lemma, each of these is the orbit $\Gamma(x_j)$ of some $x_j\notin C_X\cup \Exc(\pi)$. 
Let $F$ be the union of these finite orbits $\Gamma(x_j)$, and set $F'=F\setminus (C_X\cup \Exc(\pi))$. 
For a general choice of $q_k\in C$, (Hyp1-Hyp4) are satisfied and 
the lines $(q_kx)$ for $x\in \pi(F')$ are pairwise distinct, 
and $x$ does not belong to the lines $(q_kp_{k,j})$, for $1\leq j\leq 4$, nor to 
the conic $D_{q_k}$ from  \S~\ref{par:action_of_sigma_on_NS} (see Lemma \ref{lem:general_conic} below). 
Thus, $\sigma_k(\pi(F'))\cap \pi(F')=\emptyset$, and the orbit of $\tilde\sigma_k(z)$ under $\Gamma_3$ is infinite for any $z$ in $F'$. 
This shows that $\Gamma$ does not have any finite orbit, except for its fixed point set, which coincides with $C_X$. \end{proof}

\begin{lem}\label{lem:general_conic}
Fix $x\in \P^2\setminus C$. Then for general $q \in C$, $x$ does not belong to the conic $D_{q}$.
\end{lem}

 \begin{proof}  
 Suppose that  $x$ is in  $D_q$ for every $q$
 in some Zariski dense open subset $U$ of $C$. 
By continuity, the same holds for $U=C$. In other words, $D_q$ should contain $x$ for every $q\in C$. Let us derive a contradiction.

One can choose a point\footnote{Such a point always exists. Otherwise, looking at the linear projection from $C$ to $\P^1\simeq\P(T_x\P^2)$, which is a ramified cover of degree $3$, one sees that there would be three inflexion lines meeting in $x$. 
So, assume $p$, $q$, $r$ are inflexion points with $T_pC\cap T_qC\cap T_rC=\set{x}$. Put $C$ in Weierstrass form $x_2^2x_3=x_1^3+bx_1x_3^2+cx_3^3$, with $r=[0:1:0]$, $T_rC=\set{x_3=0}$, and $x=[1:d:0]$ for some $d\in \C$. The lines passing through $x$ have equations $x_2=dx_1+ex_3$. Such a line intersects $C$ when $$ x_1^3-d^2x_1^2x_3+(b-2de)x_1x_3^2+(c-e^2)x_3^3=0,$$
and it is an inflexion line when  this cubic equation coincides with $(x_1-\tau x_3)^3$ for some triple root $\tau$. This implies that 
$3\tau=d^2$, $3\tau^2=b-2de$, and $e=(d^4-3b)/6d$, which means that $e$ is determined by $d$; in other words,  there is at most one inflexion line of slope $d$ (plus the line at infinity), and it is impossible, for a smooth cubic curve $C$, to have three inflexion lines with a common point. 
}   $q\in C$ such that 
$(xq)$ is tangent to $C$ at $q$ and $q$ is not an inflexion point of $C$. By  \S~\ref{par:action_of_sigma_on_NS}, 
  $D_q$ is tangent to $C$, hence to 
$(xq)$, at $q$,   so $D_q$ intersects $(xq)$ at $q$ with multiplicity $2$. Since $D_q$ is a conic, we conclude that  $D_q\cap (xq) = \set{q}$ (with multiplicity $2$). But then, $D_q$ does not contain $x$. 
\end{proof}

\subsection{Invariant measures} \label{subs:invariant_measures}

\begin{thm}\label{thm:blanc_no_invariant_measure} 
 Fix  a smooth cubic curve $C\subset \P^2$ defined over ${\overline{\Q}}$.  
Assume that $k\geq 4$, and consider a $(k-1)$-tuple of points $(q_1, q_2, \ldots, q_{k-1})$ in $C(\overline{\Q})$ 
satisfying hypotheses (Hyp1) to (Hyp4). Then, for $q_k\in C(\overline{\Q})$ outside a proper real analytic curve,
the pair $(X,\Gamma)$ determined by $(q_1, q_2, \ldots, q_{k})$ does not have any invariant probability measure   except for the probability measures supported on the fixed point set $C_X$. 
\end{thm}

\begin{proof}   
As in Proposition~\ref{pro:no_finite_orbit_strong_version} we
 start with a $k$-tuple of points $(q_1, q_2, \ldots, q_{k-1}, q_k')$ satisfying (Hyp1) to (Hyp4) and   deform $q_k'$ into $q_k$ so that every ergodic, invariant probability measure is supported on $C_X$.
According to Lemma~\ref{lem:no_invariant_curve}, there is no invariant curve except $C_X$. Thus, from Proposition~\ref{pro:no_finite_orbit_strong_version}, it suffices to 
exclude the existence of invariant measures  giving no mass to proper Zariski closed subsets. 
Let $\mu$ be such a measure. 
By Theorem~A  of \cite{invariant}, there are two possibilities:
\begin{enumerate}
\item either  $\mu$ is supported on a totally real, real analytic subset $\Sigma\subset X$, and $\mu$ is absolutely continuous with respect to the $2$-dimensional Lebesgue measure on $\Sigma$, with a real analytic density along the smooth locus of $\Sigma$;
\item or the support of $\mu$ is the whole surface $X$, and $\mu$ is absolutely continuous with respect to any smooth volume form, with a positive and real analytic density outside some invariant algebraic subset $Z$.
\end{enumerate}

Let us exclude the second case. We already know that $C_X$ is the only invariant algebraic subset, hence $\mu=
\varphi \; \vol_X^\infty$ for some real analytic function 
$\varphi\colon X\setminus C_X\to \R_+$. The ergodicity of $\mu$ and the invariance of $\vol_X^\infty$ imply that 
$\varphi$ is  constant. But then,  by Property (3) of \S~\ref{par:setting_and_omega}, 
$\mu(X)   =\infty$, which is a contradiction. 

To rule out the first case, we argue as follows. The real (singular) surface $\Sigma$ is 
 invariant under the action of $\Gamma_{k-1}=\langle \tilde\sigma_1, \tilde\sigma_2, \ldots, \tilde\sigma_{k-1}\rangle$, and   supports an invariant probability measure with a smooth density. According to Theorem~C of \cite{invariant}, there are only finitely many surfaces of this type. 
 We denote by $\Sigma_{k-1}\subset X$ the union of these real analytic subsets: it is the maximal, $2$-dimensional, real analytic subset of $X$ that supports a $\Gamma_{k-1}$-invariant probability measure with full support. To conclude, it suffices to show that, after perturbation of $q_k'$, the surface $\Sigma_{k-1}$ is not $\tilde\sigma_k$-invariant. 

For this, denote by $z$ a smooth point of $\Sigma_{k-1}$ with $\pi(z)\notin C$. As $r$ varies along $C$, the point $\sigma_r(\pi(z))$ describes a complex algebraic curve. This curve cannot be contained in $\pi(\Sigma_{k-1})$, because $\Sigma_{k-1}$ is totally real. Thus, the set 
$B_{k-1}=\{ r \in C; \; \sigma_r(z)\in \Sigma_{k-1}\}$ is a proper real analytic subset of $C$. 
Then, we choose a point $q_k\in C\setminus B_{k-1}$ 
such that $(q_1, \ldots, q_k) $ satisfies 
the conclusion of Proposition~\ref{pro:no_finite_orbit_strong_version}, and we are done. 
\end{proof}

\begin{rem}
Since the proof of the theorem goes by breaking down all possible invariant totally real surfaces, the argument  does not apply to the real case. Another argument will be given for the real setting in Theorem~\ref{thm:blanc_stiffness_real} below. 
\end{rem}

We believe that under the assumptions of Theorem~\ref{thm:blanc_no_invariant_measure}, 
every stationary measure is invariant. More precisely, every 
ergodic stationary measure should
have both Lyapunov exponents zero, therefore be invariant.  In the next sections, we establish this 
result for some real examples. 

\section{Real construction} \label{sec:real_construction}

In this section, we construct  examples  for which $X$ and the $\sigma_i$ are defined over the real numbers, $X(\R)$ is obtained by blowing-up $12$ points of $\P^2(\R)$ (while $X(\C)$ is obtained by blowing up $20$ points), and  the action of $\Gamma$ 
on  the 1-dimensional  homology of $X(\R)$ admits a positive Lyapunov exponent. 

The computations done in this section, in particular in \S~\ref{par:real_action_of_Jonq_involutions} and~\ref{par:non_elementary_real} follow closely the strategy developed by Bedford, Diller, and Kim in~\cite{bedford-diller:amer_j_math, diller-kim:experimental_math} to describe similar examples. 

\subsection{Topology of real rational surfaces}\label{subs:topology} Let $X_\R$ be a real projective surface obtained by blowing up $a$ points $r_i$ of $\P^2(\R)$, as well as $b$ pairs of complex conjugate points $\{s_j, \overline{s_j}\}\subset \P^2(\C)\setminus \P^2(\R)$. We shall look at the first homology group of $X(\R)$
with integral or rational  coefficients. For simplicity, consider some homogeneous coordinate system $[x:y:z]$ on $\P^2_\R$, and suppose that the points $r_i$ are not contained in the line at infinity $L_\infty=\{z=0\}$: they are contained in the affine plane $\mathbb{A}^2(\R)$ of points $[x:y:1]$ with $(x,y)\in \R^2$. We endow this plane $\R^2$ with the usual, counterclockwise, orientation. By convention, 
if $U\subset \R^2$ is a domain with piecewise smooth boundary $\partial U$, we orient $\partial U$ in such a way that  
at any point $m\in \partial U$, 
  $(u_m,v_m)$ forms a positively oriented basis   of $T_m\R^2$, where $u_m$ is 
 the unit normal vector   to $\partial U$   pointing outside $U$; then $v_m$ is 
   the unit tangent vector  to $\partial U$ that is compatible with the orientation of $\partial U$.

For each index $i$, we denote by $U_i$ a small disk centered at $r_i$, and we orient its boundary $C_i:=\partial U_i$ in the clockwise direction. We suppose that the $U_i$ are pairwise disjoint. The exceptional divisor $E_i$ obtained by blowing up $r_i$ is a projective line. 
Then  $E_i(\R)$ is a circle, that we orient in such a way that $C_i=2E_i$ modulo homotopy (see Figure~\ref{fig:blowup}).    Let $V$ be the complement of the $U_i$ in $\R^2$, with  the 
 orientation induced by the orientation of~$\R^2$. We orient $L_\infty$ in such a way that 
$\partial V = 2L_\infty+ \sum_i C_i.$ Then in $H^1(X(\R), \Z)$,
\[
2\ell_\infty + 2\sum_i \bfe_i =0
\]
where $\bfe_i=[E_i(\R)]$ is the homology class of $E_i(\R)$ and $\ell_\infty=[L_\infty(\R)]$ is the  homology class of $L_\infty(\R)$.
Then, $H^1(X(\R), \Z)$ is isomorphic to $\Z^a\oplus\Z/2\Z$, and $H^1(X(\R),\Q)$ is isomorphic to $\Q^a$. More precisely, a basis of  $H^1(X(\R),\Q)$ is provided by the classes $(\bfe_1, \ldots, \bfe_a)$. 

If $L$ is any line in $\P^2(\R)$ which is not vertical, we orient $L$ from left to right; in other words, one can parametrize $L$ by $x\in \R\mapsto [x: \alpha x+\beta:1]$ for some $\alpha, \beta$ in $\R$, and this parametrization is compatible with orientations. Letting 
 $U^+_L$ be the open half-plane above $L$,   its boundary in $\P^2(\R)$ is made of $L$ and the line at infinity. In $X(\R)$, this gives
 \begin{equation}
 \ell_\infty + \ell + 2\sum_{r_i\in U^+_L}\bfe_i+\sum_{r_j\in L}\bfe_j=0,
 \end{equation} 
(see Figure~\ref{fig:orientation} for a local picture of a blow-up at a boundary of a domain) hence
\begin{equation}\label{eq:formula_homology_line}
\ell = \sum_{r_i \; \text{strictly below} \;  L} \bfe_i -\sum_{r_i \; \text{strictly above} \;  L} \bfe_i
\end{equation}
in $H^1(X(\R);\Q)$. This formula works even if some of the $r_i$ are contained in $L$.
 
\subsection{Action of one Jonquières involution}\label{par:real_action_of_Jonq_involutions}  Consider a smooth real cubic curve $C\subset \P^2_\R$ such that $C(\R)$ is connected. We assume that $C$ is in Weierstrass form $y^2=x^3+ ux^2+vx+w$ and we 
orient $C(\R)$ from bottom to top, {i.e.}\ from negative values of $y$ to positive ones.

Let $q$ be a point of $C(\R)$ which is not an inflexion point. Let $\sigma$ be the Jonquières involution associated to $(C,q)$, as in Sections~\ref{par:setting_and_omega} and~\ref{par:description_of_sigma_i} above. Besides $q$, the four remaining base points $p_{j}$ of $\sigma$ are made of two real points $p_1$, $p_2$ and 
two complex conjugate points $\{ p_3, p_4=\overline{p_3}\}$; here, we shall assume that the position of $q$, $p_1$, and $p_2$ are as
on Figure~\ref{fig:points_on_cubic} below. 

\begin{figure}[h] 
\includegraphics[width=7cm]{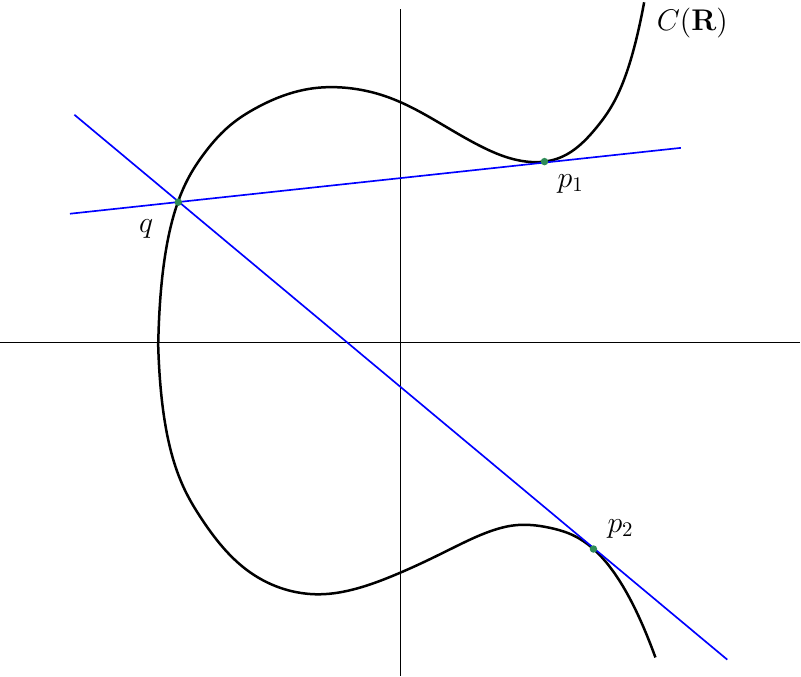}
\caption{}\label{fig:points_on_cubic}
\end{figure}

Denote by $X$ the surface obtained by blowing up the base points of $\sigma$, by $\pi\colon X\to \P^2$ the natural morphism, and by $\tilde{\sigma}$ the automorphism $\pi^{-1}\circ \sigma\circ \pi$ of $X$.
To describe the action of $\tilde{\sigma}$ on $H^1(X(\R);\Q)$, we use the following facts:

\smallskip

 \noindent--- the involution $\tilde{\sigma}$ permutes $E_{p_1}$ and the strict transform $L_{qp_1}\subset X$ of the line $(qp_1)\subset \P^2(\R)$.
More precisely, $L_{qp_1}$, $E_{p_1}$, and $C_X$ have a unique common point $\tilde{p}_1$. This point is fixed by $\tilde{\sigma}$ and 
the differential $D\tilde{\sigma}_{\tilde{p}_1}$ can be described with the help of Figure~\ref{fig:blowup}. There is a basis of $T_{\tilde{p}_1}X$ given by vectors $u$ and $v$ which are respectively tangent to $E_{p_1}$ and $L_{qp_1}$ and are compatible with their orientations; moreover, after scaling $v$ by some positive factor, we may assume that $-u+v$ is tangent to $C_X(\R)$ (and is compatible with its orientation). Then, the matrix of $D\tilde{\sigma}_{\tilde{p}_1}$ in this basis $(u,v)$ is  \begin{equation}
D\tilde{\sigma}_{\tilde{p}_1}=\left(\begin{array}{cc} 0 & -1 \\ -1 & 0\end{array}\right).
\end{equation}
Thus, $\tilde{\sigma}(E_1(\R))=-L_{qp_1}(\R)$, where the minus sign means that the orientation is reversed.  In homology, this gives 
\begin{equation}
\tilde{\sigma}^*\bfe_{p_1} = -\ell_{qp_1}=-\bfe_{p_2}.
\end{equation}
Here, for the second equality 
we used formula~\eqref{eq:formula_homology_line} and the fact that the unique point outside
  $(qp_1)$ is the point $p_2$ which is below it.
When we shall blow up more points, extra terms will be added.

\begin{figure}[h]
\begin{minipage}{6cm}\includegraphics[width=7cm]{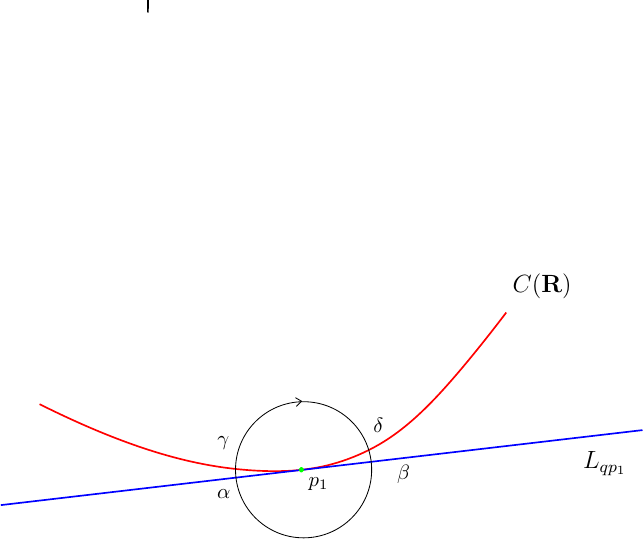}\end{minipage} \hspace{1cm}
\begin{minipage}{6cm}\includegraphics[width=5cm]{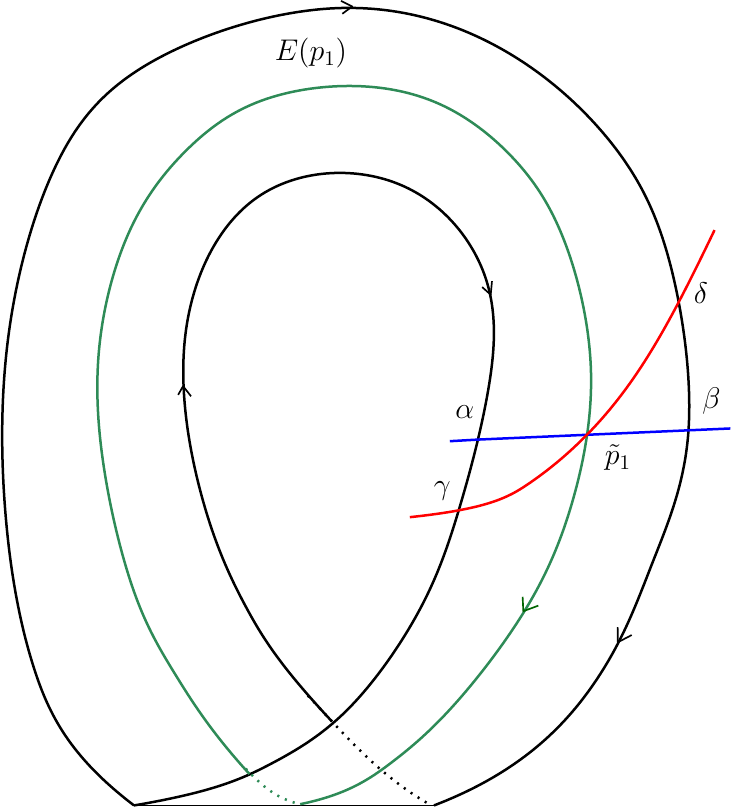}\end{minipage}
\caption{
{\small{This Möbius band is obtained from a blow-up of a small disk around $p_1$; its boundary is the preimage of the circle bounding this disk. On the right, the green curve is the exceptional divisor $E(p_1)$; the blue line is the strict transform of $(qp_1)$; the red curve is the strict transform of $C$. The involution $\tilde{\sigma}$ fixes the point of intersection of these three curves, permuting the green and blue curves.  The orientations are the ones defined previously. }}}\label{fig:blowup}
\end{figure}
 \noindent--- the picture is different at $p_2$ (because the concavity of $C$ is reversed), and we obtain 
\begin{equation}
\tilde{\sigma}^*\bfe_{p_2} = \ell_{qp_2}=-\bfe_{p_1}.
\end{equation}

\smallskip

 \noindent--- the ``image'' of $q$ by $\sigma$ is the conic $D_q$ that goes through $q$, $p_1$, $p_2$, 
 and the points $\{ p_3, p_4=\overline{p_3}\}$. It is tangent to $C$ at $q$; its real part is an ellipse, which we orient in the clockwise direction. Now, $D_q(\R)$
  bounds a disk $\Omega$; more precisely, $\Omega$ is the bounded domain of $\R^2$ whose boundary is (with our  convention on orientation from \S~\ref{subs:topology}) 
  $-D_q(\R)$ ({i.e.}\ $D_q(\R)$ but with the anti-clockwise direction).  Taking the preimage of $\Omega$ in $X(\R)$, this gives 
\begin{align}
\partial \Omega & = -D_q(\R) + E_{p_1}(\R) + E_{p_2}(\R) + E_q(\R) \\ 
[D_q(\R)] & = \bfe_q + \bfe_{p_1} + \bfe_{p_2}
\end{align}
and then we obtain 
\begin{equation}
\tilde{\sigma}^*\bfe_q = \bfe_q + \bfe_{p_1} + \bfe_{p_2}.
\end{equation}

Altogether, in the basis $(\bfe_q, \bfe_{p_1}, \bfe_{p_2})$ of $H^1(X(\R);\Q)$, the matrix for $\tilde{\sigma}^*$ is
\begin{equation}
\tilde{\sigma}^* = \left( \begin{array}{ccc} 1 & 0 & 0 \\ 1 & 0 & -1 \\ 1 & -1  & 0 \end{array} \right).
\end{equation}

\subsection{Action of three Jonquières involutions}\label{par:non_elementary_real} We now move on to the case when three involutions $\sigma_i$ are considered, each of them 
attached to a point $q_i$ of $C(\R)$. We suppose that the relative position of the points $q_i$ and $p_{i,j}$ are in the following order along $C$ (from bottom to top):  
\begin{equation}
 p_{1,2}, \; p_{2,2},  \; p_{3,2},  \; q_3,  \; q_2,  \; q_1,  \; p_{1,1},  \; p_{2,1},  \; p_{3,1}; 
\end{equation}
to obtain such a configuration, start with $q_1$ as in Figure~\ref{fig:points_on_cubic}, 
then choose $q_2\in C$ slightly below $q_1$ and $q_3$ slightly below $q_2$. Now, $X$ is the blow up of the plane at the fifteen points $q_i$, $p_{i,j}$, $1\leq i \leq 3$, $1\leq j \leq 4$, and the lifts of the $\sigma_i$ to $X$ are denoted ${\tilde \sigma_i}$.

To compute the action of $\tilde{\sigma}_1$ on $H^1(X(\R); \Q)$, we remark that the relative positions of the $q_i$ and $p_{i,j}$, impose the following properties:

\smallskip

 \noindent-- the points $p_{2,1}$ and $p_{3,1}$ are above the line $(q_1p_{1,1})$, and the points  $p_{1,2}$, $p_{2,2}$, $p_{3,2}$, $q_3$,   and  $q_2$ are below it. Thus, 
\begin{equation}
\tilde{\sigma}_1^*\bfe_{p_{1,1}}= - \bfe_{p_{1,2}} - \bfe_{q_2} + \bfe_{p_{2,1}} - \bfe_{p_{2,2}} - \bfe_{q_{3}} + \bfe_{p_{3,1}} - \bfe_{p_{3,2}}.
\end{equation}

\smallskip

 \noindent-- similarly, 
\begin{equation}
\tilde{\sigma}_1^*\bfe_{p_{1,2}}= - \bfe_{p_{1,1}} + \bfe_{q_2} - \bfe_{p_{2,1}} + \bfe_{p_{2,2}} + \bfe_{q_{3}} - \bfe_{p_{3,1}} + \bfe_{p_{3,2}}.
\end{equation}

\smallskip

 \noindent--- the ellipse $D_{q_1}(\R)$ bounds an open set $V_1$ that contains $q_2$, $q_3$, $p_{2,2}$, and $p_{3,2}$ in its interior. Thus, 
if we cut out small disks centered at those four points from $V_1$, and take its boundary in $X(\R)$, we obtain the equality 
\begin{equation}
[D_{q_1}(\R)]= \bfe_{q_1} + \bfe_{p_{1,1}}+ \bfe_{p_{1,2}} + 2\bfe_{q_2} + 2\bfe_{p_{2,2}} + 2\bfe_{q_3} + 2\bfe_{p_{3,2}}.
\end{equation}

\smallskip

 \noindent-- if $r$ is a point from $C(\R)$ that is not one of $q_1$, $p_{1,1}$, or $p_{1,2}$, and if one blows up that point, the curve $E_{r}$ is fixed by the lift $\tilde{\sigma}_1$. The line $(q_1r)$  and its strict transform $L_{q_1,r}$ are also invariant. Since $r$ is not one of the $p_{1,j}$, 
$(q_1r)$ and $C$ intersect transversely at~$r$: along $C$, $\sigma_1$ is the identity, and along $(q_1r)$, $\sigma_1$ is conjugate to $z\mapsto -z$, 
fixing  $r$ and another point $r'$ of $C$. Thus, on the blow up $E_{r}(\R)$, we see that $\tilde{\sigma}_1$ reverses the orientation (see Figure~\ref{fig:orientation}). 
\begin{figure}[h]
\begin{minipage}{6cm}\includegraphics[width=7cm]{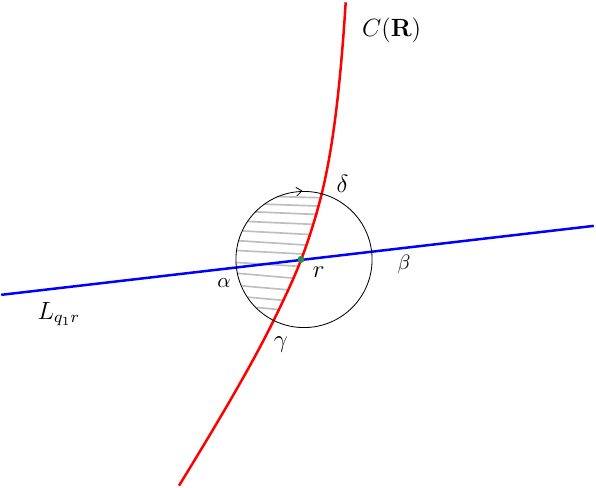}\end{minipage} \hspace{1cm}
\begin{minipage}{6cm}\includegraphics[width=5cm]{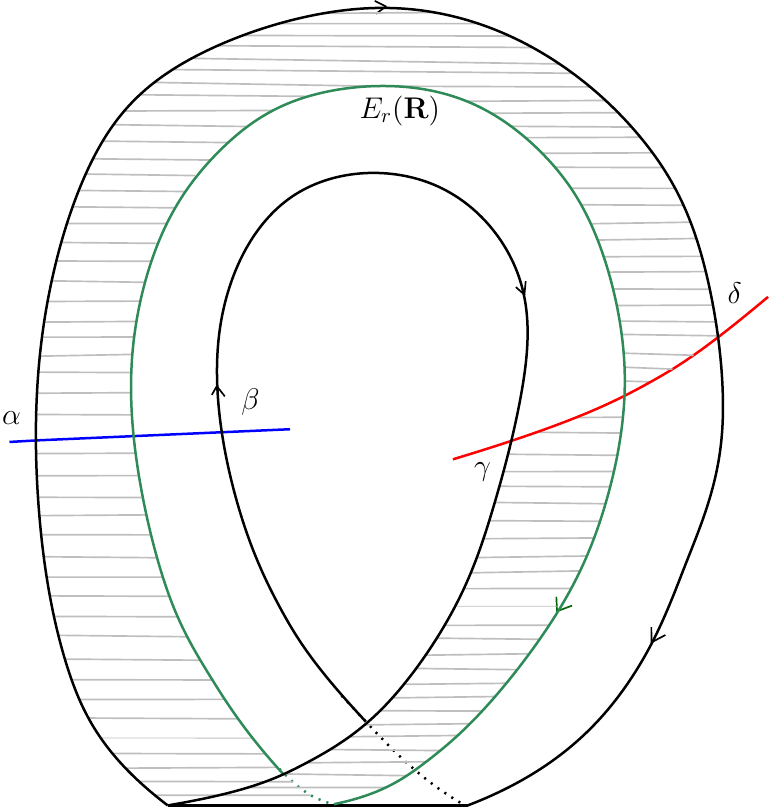}\end{minipage}
\caption{
{\small{The involution exchanges the hatched and the  plain sides of $C$, so in the blow-up it reverses the orientation of $E_r(\R)$}}}\label{fig:orientation}
\end{figure}
This gives  
\begin{equation}
\tilde{\sigma}_1^*(\bfe_{r})=-\bfe_{r}
\end{equation} 
 for $r\in \{q_2,p_{2,1}, p_{2,2}, q_3, p_{3,1}, p_{3,2}\}$.
 
\smallskip

Altogether, in the basis $(\bfe_{q_1}, \bfe_{p_{1,1}}, \bfe_{p_{1,2}} , \bfe_{q_2} , \bfe_{p_{2,1}} , \bfe_{p_{2,2}} , \bfe_{q_3} , \bfe_{p_{3,1}} , \bfe_{p_{3,2}} )$ of $H^1(X(\R);\Q)$, we obtain the following matrices for  $\tilde{\sigma}_1^*$ 
\begin{equation}
\tilde{\sigma}_1^*=\left(
\begin{array}{ccccccccc} 
1 & 0 & 0 & 0 & 0 & 0 & 0 & 0 & 0 \\
1 & 0 & -1 & 0 & 0 & 0 & 0 & 0 & 0 \\
1 & -1 & 0 & 0 & 0 & 0 & 0 & 0 & 0 \\
2 & -1 & 1 & -1 & 0 & 0 & 0 & 0 & 0 \\
0 & 1 & -1 & 0 & -1 & 0 & 0 & 0 & 0 \\
2 & -1 & 1 & 0 & 0 & -1 & 0 & 0 & 0 \\
2 & -1 & 1 & 0 & 0 & 0 & -1 & 0 & 0 \\
0 & 1 & -1 & 0 & 0 & 0 & 0 & -1 & 0 \\
2 & -1 & 1 & 0 & 0 & 0 & 0 & 0 & -1 
\end{array}
\right)
\end{equation}
Similarly, we obtain

\begin{equation}
\tilde{\sigma}_2^*=\left(
\begin{array}{ccccccccc} 
-1 & 0 & 0 & 2 & 1 & -1 & 0 & 0 & 0   \\
0 & -1 & 0 & 2 & 1 & -1 & 0 & 0 & 0   \\
0 & 0 & -1 & 0 & -1 & 1 & 0 & 0 & 0   \\
0 & 0 & 0 & 1 & 0 & 0 & 0 & 0 & 0  \\
0 & 0 & 0 & 1 & 0 & -1 & 0 & 0 & 0  \\
0 & 0 & 0 & 1 & -1 & 0 & 0 & 0 & 0  \\
0 & 0 & 0 & 2 & -1 & 1 & -1 & 0 & 0 \\
0 & 0 & 0 & 0 & 1 & -1 & 0 & -1 & 0 \\
0 & 0 & 0 & 2 & -1 & 1 & 0 & 0 & -1 
\end{array}
\right),
\end{equation}

\begin{equation}
\tilde{\sigma}_3^*=\left(
\begin{array}{ccccccccc} 
-1 & 0 & 0 & 0 & 0 & 0 & 2 & 1 & -1 \\
0 & -1 & 0 & 0 & 0 & 0 & 2 & 1 & -1 \\
0 & 0 & -1 & 0 & 0 & 0 & 0 & -1 & 1 \\
0 & 0 & 0 & -1 & 0 & 0 & 2 & 1 & -1 \\
0 & 0 & 0 & 0 & -1 & 0 & 2 & 1 & -1 \\
0 & 0 & 0 & 0 & 0 & -1 & 0 & -1 & 1 \\
0 & 0 & 0 & 0 & 0 & 0 & 1 & 0 & 0 \\
0 & 0 & 0 & 0 & 0 & 0 & 1 & 0 & -1 \\
0 & 0 & 0 & 0 & 0 & 0 & 1 & -1 & 0 
\end{array}
\right).
\end{equation}

\subsection{Positive Lyapunov exponent} 

Let $\Gamma^*$ be the image of $\Gamma=\langle \sigma_1, \sigma_2,\sigma_3\rangle $ in $\GL(H^1(X;\Z))$. This group preserves a three dimensional subspace, on which it acts by multiplication by $\pm 1$. The quotient is given by the linear map 
\begin{equation}\label{eq:quotient_map}
(x_1,x_2,x_3,x_4,x_5,x_6,x_7,x_8,x_9)\mapsto (x_1,x_2-x_3,x_4,x_5-x_6,x_7,x_8-x_9).
\end{equation}
On the quotient space, which we denote by $V$, the involutions act by the following matrices. 

\begin{equation}
A_1=\left(
\begin{array}{cccccc} 
1 & 0 & 0 & 0 & 0 & 0  \\
0 & 1 & 0 & 0 & 0 & 0  \\
2 & -1 & -1 & 0 & 0 & 0  \\
-2 & 2 & -1 & -1 & 0 & 0  \\
2 & -1 & 0 & 0 & -1 & 0 \\
-2 & 2 & 0 & 0 & 0 & -1 
\end{array}
\right).
\end{equation}

\begin{equation}
A_2=\left(
\begin{array}{cccccc} 
-1 & 0 & 2 & 1 & 0 & 0  \\
0 & -1 & 2 & 2 & 0 & 0  \\
0 & 0 & 1 & 0 & 0 & 0  \\
0 & 0 & 0 & 1 & 0 & 0  \\
0 & 0 & 2 & -1 & -1 & 0 \\
0 & 0 & -2 & 2 & 0 & -1 
\end{array}
\right).
\end{equation}

\begin{equation}
A_3=\left(
\begin{array}{cccccc} 
-1 & 0 & 0 & 0 & 2 & 1  \\
0 & -1 & 0 & 0 & 2 & 2  \\
0 & 0 & -1 & 0 & 2 & 1  \\
0 & 0 & 0 & -1 & 2 & 2  \\
0 & 0 & 0 & 0 & 1 & 0 \\
0 & 0 & 0 & 0 & 0 & 1 
\end{array}
\right).
\end{equation}

We shall denote by $\overline\Gamma^*$ the subgroup of $\SL_6(\Z)$ generated by 
 these involutions, and by $G$ the Zariski closure of $\overline\Gamma^*$ in $\SL_6$.
\begin{lem}\label{lem:representation_theory_Blanc} The following properties are satisfied:
\begin{enumerate}
\item The group $\overline\Gamma^*$ contains a non-Abelian free group.

\item The group $G$ is semi-simple. 

\item The action of $\overline\Gamma^*$ on $V\otimes \C$ is strongly irreducible. 

\item If $\nu$ is a probability measure on $\Gamma$ that satisfies Conditions~(S) and (M), then $\nu$ 
has a positive Lyapunov exponent on $H^1(X(\R);\R)$.
\end{enumerate}
\end{lem}

From this lemma and the calculation of the spectral radius 

\begin{proof}
We shall need the following facts (computations of characteristic polynomials, eigenvectors, and 
Galois groups were done with {sagemath}):

\smallskip

\noindent{(a)} The element $f=A_1A_2A_3$ has characteristic polynomial 
\begin{equation}
P_f(t)=t^6 - 4t^5 - 3t^4 - 2t^3 + 5t^2 + 2t + 1
\end{equation}
 whose factorization in $\Q[t]$ is 
$P(t)= (t - 1) \times (t^5 - 3t^4 - 6t^3 - 8t^2 - 3t - 1)$. In particular, $f$ has six distinct eigenvalues, and only one of them is rational, namely $1$. By a direct computation we see that the only other real eigenvalue is $\lambda^+_f\simeq  4.679$, 
then there are two complex conjugate eigenvalues of modulus strictly between $1$ and $\lambda^+_f$, and two of modulus $<1$.
The  eigenvector corresponding to $\lambda^+_f$ is not mapped to any other eigenvector by $A_1$. And for every $k\geq 1$,  $f^k$ has also $6$ distinct eigenvalues (see below an argument for the similar case  of $g$).

\smallskip

\noindent{(b)} The element $g=A_1A_2A_3A_2A_1A_3A_1A_3A_1A_2A_3A_2A_1A_2A_3$ has characteristic polynomial 
\begin{equation}
P_g(t)=t^6 - 24 t^5 - 83 t^4 - 122 t^3 - 35 t^2 + 22 t + 1
\end{equation}
which  is irreducible in $\Q[t]$.
In particular, $g$ has six  distinct eigenvalues. Four  of them 
are real, with a single  one of maximal absolute value $\lambda^+_g>1$, and there are two complex conjugate eigenvalues of modulus $>1$. Moreover, for every $k\geq 1$,  $g^k$ has also $6$ distinct eigenvalues.

To prove this last point, it suffices to show that $(\lambda^+_g)^k$ is an algebraic number of
 degree $6$ for every $k\geq 1$. For this, one computes the Galois group of the splitting field $F$ of $P_g$: it is isomorphic to the symmetric group $\mathfrak S_6$, so $[F:\Q]= 6!$. The degree 6 extension $\Q(\lambda^+_g)$ of $\Q$
 is the subfield of $F$ fixed by a subgroup $H\subset \mathfrak S_6$ of index $6$. Note that $(\lambda^+_g)^k$ can not be rational for any $k\geq 1$, because since $\lambda^+_g$ is a unit, $(\lambda^+_g)^k$ would be equal to $1$, in contradiction with $\lambda^+_g>1$. Thus, if $(\lambda^+_g)^k$ had degree $<6$, there would be an intermediate extension $K\subset \Q(\lambda^+_g)$ of degree $d=2$ or $3$. 
 This extension would be the fixed field of a group $G\subset \mathfrak S_6$ of index $d$. 
For $d=3$ such a group does not exist. For $d=2$ we get $G=   \mathfrak A_6$ 
 and $\mathfrak A_6$ would contain $H$ as a subgroup of index $3$; this is a contradiction, 
 since the largest maximal subgroup of $\mathfrak A_6$ has index $6$.

\smallskip

\noindent{(c)}  The eigenvector of $f$ corresponding to the eigenvalue $1$ is defined over $\Q$. Thus, it is in general position with respect to the eigenvectors of $g$ (i.e.\ it is not contained in any proper $g$-invariant  subspace of $V$, because such a subspace
would be rational,  thereby producing  a rational factor of $P_g$) 
 
\smallskip

To show that $\overline\Gamma^*$ contains a non-Abelian free group, we use the fact 
that the eigenvector of $f$ for the leading eigenvalue $\lambda^+_f $ is not 
mapped to another eigenvector of $f$ by $A_1$. Thus  by the ping-pong lemma,
 if we set $h= A_1  \circ f \circ A_1^{-1}= A_2A_3A_1$, then the group generated by $f$ and $h$ contains a non-Abelian free  subgroup of $\GL(V)$ all of whose elements $\neq \mathrm{Id}$ have an eigenvalue $>1$ (see Lemma~\ref{lem:non-elementary_free_groups}). So, at this stage we know that $\Gamma$ is non-elementary. 

Consider the connected component of the identity $G^o\subset G$. The intersection $\overline\Gamma^*_0:=\overline\Gamma^*\cap G^o(\C)$ is a finite index subgroup of $\overline\Gamma^*$ and is a Zariski dense subgroup of $G^o$. Since $\overline\Gamma^*$ is contained in $\SL_m(\Z)$, the linear algebraic groups $G$ and $G^o$ are both defined over $\Q$. 
Let $R$ be the solvable radical of $G^o$ and let $U$ be its unipotent radical (see~\cite{Milne:book}, Chap. 6.h, page 135); they are defined over $\Q$ and are characteristic subgroups of $G^o$; in particular, they are normal in $G$. 
Let $F\subset V$ be the fixed point set of $U$. This vector subspace is defined over $\Q$, its dimension is positive, and it is $g$-invariant. Since the characteristic polynomial of $g$ is irreducible over $\Q$, we infer that $F=V$ and $U=\set{\id_V}$. 
This implies that $R$ is a torus; over $\C$, $R$ is diagonalizable and $R(\C)$ isomorphic to $(\C^\times)^r$ for some $r\leq 5$. The group $G$ acts by conjugacy on $R$; since the automorphism group of $R$ is discrete (isomorphic to $\GL_r(\Z)$) and $G^o$ is connected, we deduce that $R$ is central in $G^o$. In particular, $R$ commutes with $f^k$ and with $g^k$ if $k$ is chosen to insure that $f^k$ and $g^k$ are in $G^o$. Thus, $R$ being connected, each of the sixth complex eigenlines of $f^k$ is $R$-invariant, and the same holds for the eigenlines of $g^k$. By Property~(c) above, this implies that $R$ is made of homotheties, and since $R\subset \SL_6$ we deduce that $R$ is trivial. 
Thus $G^o$ is semi-simple, and so is $G$. 

By the first property, $G^o(\R)$ is not bounded, so at least one of its semisimple factors is a (non-Abelian) non-compact almost simple real Lie group.

Now, $V\otimes \C$ is a direct sum of irreducible representations $V_i$ of $G^o$. Choose $k\geq 1$ such that $f^k\in G^o$. One of the $V_i$, say $V_1$, must contain the eigenline of $f^k$ corresponding to the eigenvalue $1$; thus, $V_1$ contains an element of $V(\Q)$, and since $\overline\Gamma^*_0$ 
is Zariski dense in $G^o$ and is defined over $\Z$, we deduce that, in fact, $V_1$ is defined over $\Q$. By Property~(b), $V_1=V$, and $G^o$ acts irreducibly on $V$. Thus, $\overline\Gamma^*_0$ acts strongly irreducibly on $V$, and so does $\overline\Gamma^*$.

Finally, since the action of $\overline\Gamma^*$ on $V$ is strongly irreducible and unbounded, 
Theorem~3.31 of~\cite{benoist-quint_book} shows the positivity of the first Lyapunov exponent in $V$, hence 
in  $H^1(X(\R);\R)$, and the proof is complete. 
\end{proof}

Putting together all the results in this section, we have established the following result.  

\begin{thm}\label{thm:lyapunov_real}
Let $C$ be a smooth, real, plane cubic whose real part $C(\R)$ is connected. There exists a non-empty 
open subset $U\subset C(\R)^3$ such that if $(q_1, q_2, q_3)$ belongs to $U$, and $\nu$ is a probability measure on the corresponding group
$\Gamma = \bra{\tilde \sigma_1, \tilde \sigma_2, \tilde \sigma_3}$, 
generating $\Gamma$ and   with a finite first moment,  
 then the top Lyapunov exponent of the action of $\Gamma$ on $H^1(X(\R), \R)$ is positive. 
\end{thm}

We now extend  this theorem to the case of $4$ points instead of $3$. So, we blow-up one more point $q_4\in C(\R)$, with $q_4$  between $q_1$ and $q_3$, and we still denote by $X$  the surface and by $\Gamma$ the group generated by the $4$ involutions. 
The cohomology group $H^1(X(\R);\Q)$ of the new surface $X(\R)$ has dimension $12$, 
with coordinates $(x_1, \ldots, x_{12})$; 
the subspace defined by 
\begin{equation}
x_2=x_3, x_5=x_6, x_8=x_9,  x_{11}=x_{12}, x_1   = x_4= x_7 = x_{10} =0
\end{equation} 
is invariant, and the restriction of $\Gamma$ to this subspace factors through a finite group. 
Let $W$ be the quotient space, which  is of dimension~$8$.

Choose an index $j\in \set{1, 2, 3, 4}$ and consider the subgroup $\Gamma_j$ of $\Gamma$ generated by the $\tilde{\sigma_i}$ with $i\neq j$. 
Then, there is a finite index subgroup $\Gamma_j^o$ of $\Gamma_j$ whose action on $W\otimes \C$ is reducible: there is an invariant subspace  $T_j$ of dimension $2$ (on which $\Gamma_j$ acts diagonally with eigenvalues equal to $\pm 1$), 
and the quotient space $W/T_j$ is a strongly irreducible representation $V_j$, defined over $\Q$, and of dimension $6$ 
(it is isomorphic to the representation $V$ studied in Lemma~\ref{lem:representation_theory_Blanc}). Pick $g_j$ in $\Gamma_j^o$   satisfying  Property~(b) of the proof of Lemma~\ref{lem:representation_theory_Blanc}. 
Then $g_j$ preserves a unique subspace $W_j$ of $W$ of dimension $6$, defined over $\Q$, which projects surjectively onto $V_j$.

Now, take   a finite index subgroup $\Gamma'$ of $\Gamma$   and let
 $K\subset W$ be a $\Gamma'$-invariant subspace. Changing the $\Gamma_j^o$ into 
  finite index subgroups, we may assume that they preserve $K$. Projecting $K$ to $V_j$, Lemma~\ref{lem:representation_theory_Blanc} implies that either $K$ is contained in $T_j$ 
 (and then $\dim(K)\leq 2$), or that $K$ is mapped onto $V_j$ (and then $\dim(K)\geq 6$). If $K$ is contained in one of the $T_j$, one checks that its projection onto some other $V_i$ is non-trivial, which gives a contradiction. It follows that 
  the projection of $K$ onto each $V_j$ is surjective. Since   $K$ is  $g_j$-invariant, 
  it contains $W_j$, and finally we obtain that $K=W$. 
Thus, the  action of $\Gamma$ on $W$ is strongly irreducible, and we get:

\begin{thm}\label{thm:lyapunov_real2}
Let $C$ be a smooth, real, plane cubic whose real part $C(\R)$ is connected. There exists a non-empty 
open subset $U\subset C(\R)^4$ such that if $(q_1, q_2, q_3, q_4)$ belongs to $U$, and $\nu$ is a probability measure on the corresponding group
$\Gamma = \bra{\tilde \sigma_1, \tilde \sigma_2, \tilde \sigma_3, \tilde \sigma_4}$, 
generating $\Gamma$ and  with a finite first moment,  
 then the top Lyapunov exponent of the action of $\Gamma$ on $H^1(X(\R), \R)$ is positive. 
\end{thm}

\section{Dynamics on the real surface} \label{sec:stiffness_blanc}

In this section we complete the proof of Theorem~\ref{thm:blanc_examples}.

 \subsection{Preliminaries from ergodic theory}

\begin{pro}\label{prop:ergodicity_volume}
Let  $X$ be  a complex projective surface (resp. a real projective surface). Let $\Gamma$ be a non-elementary subgroup of 
$\Aut(X)$ (resp. of $\Aut(X_\R)$) containing a parabolic element. Then the action of $\Gamma$ is ergodic with respect to the 
 Lebesgue measure on $X$ (resp. on $X(\R)$). 
\end{pro}

\begin{proof}
This result is contained in \cite{cantat_groupes, invariant} but it was not explicitly stated there, so we provide details in the case of complex surfaces. We only deal with the case where $X$ is not Abelian, in which case 
$\Gamma$ is automatically countable.   We closely follow \cite[\S 4]{invariant} and freely use the vocabulary from that paper. Let $A\subset X$ be a 
measurable subset such that $\vol_X(A\Delta\gamma\inv A) = 0$, where $\vol_X$ is  the probability 
 measure associated to some Kähler form on $X$. Replacing $A$ by $\bigcap_{\gamma\in \Gamma}\gamma\inv A$, we may assume that $A$ is $\Gamma$-invariant. Assuming that $\vol_X(A)>0$, we have to show that $\vol_X(A)=1$. As in \cite[\S 4.3.1]{invariant},   $\Gamma$ contains a 
  ``special'' subgroup $\langle{g, h\rangle}$, freely generated by two independent Halphen twists. Denote by 
 $\pi_g$, $\pi_h:X\to \P^1$ the associated invariant fibrations. Recall that for almost every $w\in \P^1$, the action 
 of $g$ (resp. $h$) on $\pi_g\inv(w)$ (resp. $\pi_h\inv(w)$) is uniquely ergodic, and the unique invariant measure is the Haar measure. Since $A$ is $g$-invariant, there is a subset $B_g\in \P^1$ of positive measure such that $\vol_X(A\Delta \pi_g^{-1}(B_g))=0$. Consequently $A$ intersects any $\pi_h$-fiber on a set of positive Haar measure, and then using the $h$-invariance,  $A = X$ up to a set of volume~$0$. 
\end{proof}

\begin{pro}\label{pro:sum_exponents}
Let $X$ be a compact Kähler surface. Let $\nu$ be a probability measure on $\Aut(X)$ that satisfies the moment condition (M) and let $\mu$ be an ergodic $\nu$-stationary measure.  Let $\eta$ be a non-trivial meromorphic $2$-form on $X$ such that
\begin{enumerate}[\em (i)]
\item $\displaystyle \int \log^+\vert \jac_\eta(f)(x)\vert \d\mu(x) \d\nu(f)<+\infty$;
\item $\mu$ gives zero mass to the set of  zeroes and poles of $\eta$.
\end{enumerate}
Then the Lyapunov exponents of $\mu$ satisfy:
\[
\lambda^{-}+\lambda^{+} = \int \log(\abs{\jac_\eta f (x)}^2) \d\mu(x)\d\nu(f).
\]
In particular,  if $\eta$  is  
invariant, $\lambda^{-} + \lambda^{+} = 0$. 
\end{pro}

Before starting the proof, let us recall some notation from \S~\ref{subs:intro_stiffness}: we denote 
$\Omega=\Aut(X)^\N$ and $F_+$   the skew-product transformation of $\Omega\times X$ associated to 
the random dynamics, 
acting as the one-sided shift on $\Omega$ and by automorphisms on $X$. 
The measure 
 $\nu^\N\times \mu$ is $F_+$ invariant. One may also consider the 2-sided 
shift $\vartheta\colon \Aut(X)^\Z\to  \Aut(X)^\Z$ and the corresponding 
 skew-product $F$  on $\Aut(X)^\Z\times X$, defined by   $F (\xi,x)= (\vartheta \xi, f_0(x))$, 
 where $\xi = (f_i)_{i\in \Z}$.   
 The natural extension of $\nu^\N\times \mu$ will be denoted by $\m$: it is invariant and ergodic and 
 its projection on $\Aut(X)^\Z$ is $\nu^\Z$. Beware that 
 $\m$ differs from $\nu^\Z\times \mu$ unless 
  $\mu$ is   invariant.

\begin{proof}   Fix a Kähler metric $\kappa_0$ on $X$. Fix a trivialization of the tangent bundle $TX$, given by a measurable family of linear isomorphisms 
$L(x)\colon T_xX\to \C^2$ such that (a) $\det(L(x))=1$ and (b) $1/c\leq \norm{L(x)}+\norm{L(x)^{-1}}\leq c$, for some constant $c>1$; here, 
the determinant is relative to a volume form $\vol_X$ on $X$ and the standard volume form on $\C^2$, and  the 
norm is with respect to $(\kappa_0)_x$ 
on $T_xX$ and the standard euclidean metric on $\C^2$.  

For $(\xi, x)\in \Aut(X)^\Z\times X$ and $n\geq 0$, the differential $D_xf^n_\xi$ is expressed in this trivialization as a
matrix  
\begin{equation}
A^{(n)}(\xi, x)=L(f^n_\xi(x)) \circ D_xf^n_\xi\circ L(x)^{-1}.
\end{equation} 
Denote by $\chi^{-}_n(\xi, x)\leq \chi^{+}_n(\xi, x)$ the  singular values of   
$A^{(n)}(\xi, x)$. Then $\m$-almost surely, $\unsur{n}\log \chi^{\pm}_n(\xi, x)$ converges towards $\lambda^{\pm}$ as $n$ goes to $+\infty$. 

The form $\eta\wedge{\overline{\eta}}$ can be written $\eta\wedge {\overline{\eta}}=\varphi(x) \vol_X$ for some
function  $\varphi\colon X\to [0,+\infty]$. Locally, one can write $\eta=h(x) dx_1\wedge dx_2$ where $(x_1,x_2)$ 
are local holomorphic coordinates  and $h$ is a meromorphic function; then $\varphi(x)\vol_X= \abs{h(x)}^2 dx_1\wedge dx_2\wedge d{\overline{x_1}}\wedge d\overline{x_2}$. The jacobian $\jac_\eta$ satisfies 
\begin{equation}
\vert \jac_\eta(f)(x)\vert^2=\frac{\varphi(f(x))}{\varphi(x)}\jac_\vol(f)(x)
\end{equation}
for every $f\in \Aut(X)$ and $x\in X$.
Using $\det(L(x))=1$, we get 
\begin{equation}\label{eq:determinant}
\det (A^{(n)}(\xi, x)) = \jac_\vol(f^n_\xi)(x),
\end{equation}
and then  
\begin{equation}\label{eq:jacobian}
\unsur{n}\log \chi^{-}_n(\xi, x) +  \unsur{n}\log \chi^{+}_n(\xi, x)  =  \frac{2}{n}\log \abs{\jac_\eta f^n_\xi(x)} - \unsur{n}\log(\varphi(f^n_\xi(x))/\varphi(x)).
\end{equation}
By the Oseledets theorem, the left hand side of \eqref{eq:jacobian} converges almost surely to $\lambda^{-}+\lambda^{+}$.
Since the Jacobian $\jac_\eta$ is multiplicative along orbits, {i.e.}\   $\jac_\eta f^n_\xi(x)  = \prod_{k=0}^{n-1} 
\jac_\eta f_{\vartheta^k\xi}(f^k_\xi x)$, the integrability condition and the ergodic theorem imply that, almost surely,    
\begin{align}
 \lim_{n\to\infty} \frac{1}{n} \log\abs{\jac_\eta f_\xi^n(x)}   
&=  \int \log\abs{\jac_\eta f_\xi^1 (x)} d\m(\xi, x) \\
&\notag =  \int \log\abs{\jac_\eta f_\omega^1 (x)} d(\nu^\N\times \mu)(\omega, x)  \\
&\notag  =  \int \log\abs{\jac_\eta f(x)} \d\mu(x)\d\nu(f).
\end{align}
Let ${\mathrm{div}}(\eta)$ be the set of zeroes and poles of $\eta$. Since $\mu$ is ergodic and does not charge ${\mathrm{div}}(\eta)$, we deduce that for $\m$-almost every $(\xi, x)$, 
there is a sequence $(n_j)$  such that $f^{n_j}_\xi(x)$ stays at  
positive distance from ${\mathrm{div}}(\eta)$; along such a sequence, $\log\vert \varphi(f^{n_j}_\xi(x))/\varphi(x)\vert$ 
stays bounded, and the right hand side of \eqref{eq:jacobian} tends to 
$2\int \log\abs{\jac_\eta f(x)} \d\mu(x)\d\nu(f)$. This concludes the proof. 
\end{proof}
 
 \subsection{Stiffness} 
 
\begin{thm}\label{thm:blanc_stiffness_real}  
Let $C$ be a smooth, real, plane cubic which is defined over $\overline{\Q}$ and whose real part 
$C(\R)$ is connected. One can find four points $q_1$, $q_2$, $q_3$, $q_4$ in $C(\R)$ such 
that  the following properties hold.  
Let $X$ be the Blanc surface constructed  in~\S~\ref{sec:invariant_curves_parabolic}
  and let  $\Gamma= \langle  \tilde{\sigma}_1, \tilde{\sigma}_2, \tilde{\sigma}_3, 
  \tilde{\sigma}_4\rangle\subset \Aut(X_\R)$ be the  group generated by the four Jonqui\`eres involutions.

Let $\nu$ be any probability measure on $\Gamma$ which satisfies Conditions (S) and (M). 
Then, any ergodic  $\nu$-stationary  measure on $X(\R)$ 
 is  the Dirac mass $\delta_x$ at 
some point $x\in C_X(\R)$.  Such measures have vanishing Lyapunov exponents. 
\end{thm}

In fact, the $q_i$ can be chosen in some explicit open subset of $C(\R\cap\overline{\Q})^4$.

\begin{proof} 
We   first  choose the points $(q_1, q_2, q_3, q_4')$ in $C(\R\cap \overline \Q)$ 
such that the properties (Hyp1-4) hold, as well as the conclusions 
of Theorem~\ref{thm:lyapunov_real2}. Then we move 
$q_4'$ to $q_4$ exactly as in   Proposition~\ref{pro:no_finite_orbit_strong_version} to make sure 
 that  the conclusions of this proposition are satisfied.

\noindent{\bf{Step 1.--}}  
Here we show that every 
\emph{invariant} ergodic probability
measure is a Dirac mass $\delta_x$ for some $x\in C_X(\R)$. 
The argument is similar to that of Theorem~\ref{thm:blanc_no_invariant_measure}. 
First, the invariant meromorphic 2-form is real and induces  a volume form $\vol_{X(\R)}^\infty$, which is not locally integrable along $C_X(\R)$. 
By  Lemma~\ref{lem:no_invariant_curve} and Proposition~\ref{pro:no_finite_orbit_strong_version} every proper Zariski closed invariant subset is contained in $C_X$ (hence   fixed pointwise). Thus, by~\cite{invariant}, if $\mu$ is any ergodic invariant measure on 
$X(\R)$ then either $\mu$ is a Dirac mass on $C_X(\R)$ or $\mu=\varphi \; \vol_{X(\R)}^\infty$, 
for some real analytic function  $\varphi$ on $X(\R)\setminus C_X(\R)$. To exclude this second possibility, 
note that $C(\R)$ is connected and has a unique point at infinity, so $\P^2(\R)\setminus C(\R)$ is  connected, and so is $X(\R)\setminus C_X(\R)$. 
Then, the ergodicity of $\mu$ and the  invariance of $\vol_{X(\R)}^\infty$ imply that $\varphi$ is 
constant $\mu$-almost everywhere; since $\varphi$ is analytic, it must be locally constant, and since $X(\R)\setminus C_X(\R)$ is connected, $\varphi$ is constant;   but  then $\mu(X(\R)) = \infty$, which is absurd.

\noindent{\bf{Step 2.--}} Now, pick   a probability measure  $\nu$ on $\Gamma$ satisfying (S) and (M). Let $\mu$ be an ergodic, $\nu$-stationary measure on  $X(\R)$.

\begin{lem}\label{lem:absolute_continuity}
 The measure $\mu$ cannot be absolutely continuous with respect to the Lebesgue measure on $X(\R)$. 
 \end{lem}
 
 \begin{proof}
Equivalently, let us show that $\mu$ is not absolutely continuous with respect to  
the invariant infinite volume $\vol_{X(\R)}^\infty$.  For notational ease, in this proof we write $X$ for $X(\R)$ 
(actually the  result   holds  for  both the complex and the real variety). 
Reasoning by contradiction, we assume that there is a function 
$\xi\colon X\to \R_+$ such that $\xi \in L^1(X, \vol_X^\infty)$ and 
$\mu=\xi \;  \vol_X^\infty$. 
The stationarity of $\mu$ and the invariance of $\vol_X^\infty$ under the action of $\Gamma$ give
\begin{align}
\xi & =  \int_\Gamma  \xi\circ f^{-1}    {\mathrm{d}}\nu(f)
\end{align}
$\vol_X^\infty$-almost everywhere. For $M\geq 0$, we  set $\xi_M=\min(\xi, M)$ and 
     obtain
\begin{align}\label{eq:harmonic_inequality}
\xi_M & \geq  \int_\Gamma  \xi_M \circ f^{-1}    {\mathrm{d}}\nu(f)
\end{align}
applying, say, Jensen's inequality. 
Then   
\begin{align}
\int_{X } \xi_M \d\vol_X^\infty & \geq   \int_\Gamma  \int_{X } \xi_M\circ f^{-1}   \d\vol_X^\infty  {\mathrm{d}}\nu(f)  \\
&= \int_{X } \xi_M \d\vol_X^\infty
\end{align}
 because $\vol_X^\infty$ is invariant and $\xi_M$ is integrable. 
This shows that the Inequality~\eqref{eq:harmonic_inequality} is in fact an equality  $\vol_X^\infty$-almost everywhere. Thus, $\xi$ and $\xi_M$ are both $\nu$-harmonic. 

By Markov's inequality, there exists $M>0$ such that 
\begin{equation}
0 < \vol_X^\infty(\{\xi >M\} )< +\infty.
\end{equation}
Given such an $M$, the real number
\begin{equation}
\alpha:=\int_{\set{\xi>M}} (\xi-M)\d\vol_X^\infty = \int_{\set{\xi>M}} (\xi-\xi_M)\d\vol_X^\infty 
\end{equation}
 satisfies $0 <\alpha\leq 1$ because $\mu$ is a probability measure. 
If $\alpha<1$,  $\mu$ could be decomposed as a convex combination
$\mu = \alpha \mu_M^++(1-\alpha)\mu_M^-$, with 
\begin{align}
\mu_M^+=\alpha^{-1}(\xi-\xi_M) \vol_X^\infty  \; \text{ and } \; 
\mu_M^-=(1-\alpha)^{-1} \xi_M\vol_X^\infty .
\end{align}
Such a decomposition must 
  be trivial, because the harmonicity of $\xi_M$ and of $\xi$ imply that $\mu_M^-$ and $\mu_M^+$ are stationary measures, while $\mu$ is assumed to be ergodic. 
  Thus, $\alpha=1$, for any such~$M$. This shows that there is measurable subset $A\subset X $    such that 
\begin{enumerate}
\item[(a)] $\vol_X^\infty(A)\in\,  ]0,\infty[$;
\item[(b)] $\xi = \vol_{X}^\infty(A)\inv 1_A$, where $1_A$ is the characteristic function of $A$. 
\end{enumerate}
The stationarity relation for $\xi$ implies that $A$ is $\supp(\nu)$-almost invariant, hence 
$\Gamma$-invariant, up to a subset  of zero volume. Proposition~\ref{prop:ergodicity_volume}   then implies that $A=X(\R)$ up to a subset of zero volume. This is a contradiction because $\mu$ is a probability measure and $\vol_{X}^\infty(X)=+\infty$.
 \end{proof}

\noindent{\bf{Step 3.--}}  
Let $\mu$ be an ergodic $\nu$-stationary measure as in Step~2. 
Assume by way of contradiction
 that $\mu$ is not invariant.  

A first observation is that $\mu(C_X(\R)=0$, so $\mu$ is Zariski diffuse. 
By Proposition~\ref{pro:sum_exponents},  $\lambda^++\lambda^- = 0$. Therefore,  
Crauel's invariance principle implies that   $\mu$ must be hyperbolic. 
Thus, by Theorem~\ref{thm:non_random} and~\cite[Thm. 3.4]{br}, $\mu$ is a fiberwise  SRB measure. Specifically, this means that working in 
$\Aut(X(\R))^\Z\times X(\R)$, and considering a measurable partition 
$\mathcal P^u$ subordinate to the family of local Pesin unstable manifolds, then 
for $\m$-almost every $(\xi, x)$, 
the conditional $\m^u_{(\xi, x)}:= \m (\cdot \vert \mathcal P^u(\xi, x))$ is absolutely continuous 
with respect to the Lebesgue measure on $W^u(\xi, x)$ (see~\cite{br} and~\cite[\S 7]{stiffness} for details on these concepts, and the paragraph following Proposition~\ref{pro:sum_exponents} for the notation).

The same construction can be done with a local stable partition $\mathcal P^s$ to get stable conditional measures $\m^s_{(\xi, x)}$. 
These conditional measures admit a pointwise Hausdorff dimension at $(\xi, x)$ which is defined almost everywhere and is constant by ergodicity. 
We denote these dimensions by $\dim(\mu^{u/s})$.
 In this context, the analogue of the Ledrappier-Young formula holds\footnote{The first equality in formula~\eqref{eq:LY} is proven in~\cite{qian_xie}. For 
random dynamical systems there is a dissymmetry between the future and the past, because
$F\inv$ is not the skew product map associated to a independent, identically distributed, random dynamical system. So we can not just 
consider $F\inv$ to get the second inequality. Fortunately, only minor adaptations are required for this: they
are described in the paragraphs following Theorem 2.1 in~\cite{liu_xie} (see also \cite{liu-kifer} for a unified discussion, with additional pointers to the literature} and asserts that for $\m$-almost every $(\xi, x)$, 
\begin{equation}\label{eq:LY}
h_\mu(X, \nu) = \lambda^+ \dim^u(\mu) = \abs{\lambda^-} \dim^s(\mu),
\end{equation}
where $h_\mu(X, \nu)$ is the fiber entropy (see Section~\ref{sec:classification} for a brief account on 
this notion). Since 
$ \dim^u(\mu)  =1$ and $\lambda^+ = \abs{\lambda^-}$, we conclude that 
$\dim^s(\mu)=1$ as well.  As in deterministic dynamics, 
this implies that $\mu^s$ is absolutely continuous with 
respect to the Lebesgue measure along stable manifolds\footnote{This is proven for  the unstable direction in 
\cite{liu-qian}, and the adaptation to the stable direction is explained in~\cite{bahnmuller-liu}.}.  
Thus,  $\m$ has absolutely continuous conditionals along both stable and unstable manifolds.  

The next important property is the absolute continuity of the local stable and unstable laminations, which follows the lines of the classical deterministic case. 
A detailed treatment for the stable lamination is given in~\cite[Chap. III]{liu-qian}, and a unified treatment for the stable and unstable laminations (with less details) is in~\cite[Thm 2.2.12]{liu-kifer}. 

At this stage we can directly adapt Theorems 5.1 and 5.5 of~\cite{ledrappier_sinai}, which implies that the conditionals of $\m$ 
along the fibers $\set{\xi}\times X(\R)$ are almost surely absolutely continuous with respect to the Lebesgue measure, 
and we conclude that $\mu$ itself is absolutely continuous. 
But   Lemma~\ref{lem:absolute_continuity} asserts that this is impossible. 
This contradiction shows that $\mu$ is invariant. Thus, applying Step~2,  we conclude that 
every ergodic stationary measure is a Dirac mass $\delta_x$ at some point $x\in C_X(\R)$.

\noindent{\bf{Step 4.--}} It  remains to show that the point masses on $C_X$, viewed as stationary measures, have zero 
Lyapunov exponents. This is elementary: let $x\in X(\R)$ (possibly among the  $\tilde q_i$)  and 
 let $(v_1, v_2)$ be a basis of $T_xX$ such that $v_1$ is tangent to $C_X$. Then $(\tilde \sigma_i)_\varstar(v_1) = v_1$  so the matrix of $(\tilde \sigma_i)_\varstar$ in the basis $(v_1, v_2)$ is upper triangular, and the other eigenvalue of this upper-triangular matrix is $-1$ (cf. the proof of 
 Formula~\eqref{eq:jacobian-is-1}). So any product of such matrices is of the form  $\lrpar{\begin{smallmatrix} 1 & * \\ 0&\pm1\end{smallmatrix}}$ and the result follows. 
\end{proof}

\subsection{Orbit closures}

\begin{thm}\label{thm:blanc_orbit_closures_real}
Let $C$, $q_1, \ldots , q_4$, $X$ and $\Gamma$ be as in Theorem~\ref{thm:blanc_stiffness_real}. Then every $x\in X(\R)\setminus C_X(\R)$ has a dense orbit in $X(\R)$. 
\end{thm}

\begin{proof} In~\cite{invariant}, we defined an invariant algebraic subset $\mathrm{STang}_\Gamma$, 
which is the union of the maximal invariant curve and a finite set, and we proved that for every $x\in X(\C)$, either $\Gamma\cdot x$ is dense in $X(\C)$, or 
$\mathrm{Acc}(\Gamma\cdot x)\setminus \mathrm{STang}_\Gamma$ 
is locally equal to some $\Gamma$-invariant real surface. 
In our situation, $\mathrm{STang}_\Gamma   = C_X$, because $C_X$ is the maximal invariant curve and every orbit outside $C_X$ is infinite. 
For $x\in X(\R)\setminus C_{X}(\R)$, we deduce that 
if $\mathrm{Acc}\lrpar{\Gamma\cdot x}\setminus C_{X}(\R)$ is non-empty, then it is open 
(and closed) in $X(\R)\setminus C_{X}(\R)$, so by connectedness $\Gamma\cdot x$ is dense in $X(\R)$. To sum up, all we need to show is that $\mathrm{Acc}\lrpar{\Gamma\cdot x}\setminus C_{X}(\R)\neq\emptyset$ for every 
$x\in X(\R)\setminus C_{X}(\R)$. For this, we use the structure 
of the invariant fibrations of parabolic elements in $\Gamma$. 

\noindent{\bf{Step 1.--}} Geometry of the invariant fibration of $g_{ii'}$.

Fix two distinct indices, say $i=1$ and $i'=2$. In this step 
 we study the geometry of the invariant fibration $\pi_{12}$ of 
$g_{12}  = \tilde \sigma_1\circ \tilde \sigma_2$ near $C_X$.

Let us first work over $\C$. Recall from \S~\ref{subs:parabolic_blanc} that the fibration comes from the pencil of quadrics going through $q_1$ and $q_2$ with multiplicity 2 and through the $p_{i,j}$ with multiplicity 1 for $i=1,2$. 
In the surface $X_{12}$ obtained by blowing up these $10$ points (as in Lemma~\ref{lem:parabolic_gij}),
it corresponds to the linear system $\abs{C_{X_{12}} +M_{12}}$, where $M_{12}$ 
is the strict transform of the line $(q_1q_2)$. We   denote by $\pi_{12}\colon X_{12}\to \P^1$ this fibration, and   fix an affine coordinate $z$ on $\P^1$ such that  $\pi_{12}(C_{X_{12}}\cup M_{12})=0$ and $\pi_{12}(X_{12}(\R))\subset \P^1(\R)$.

Let $r$ be the third intersection point of $C$ and $(q_1q_2)$; we also denote  by $r$ its incarnation in $X_{12}$ or $X$ (because $r$ is not blown up).  It is the only intersection point of $C_{X_{12}}$ and $M_{12}$. 

Now, $M_{12}$ is a smooth rational curve with self-intersection $-1$, so it can be blown down to get a new, smooth projective surface $Y$.
The fibration $\pi_{12}$ gives a genus $1$ fibration $\pi_Y\colon Y\to \P^1$ and $C_{X_{12}}$ is a smooth fiber $C_Y$ of $\pi_Y$. In a small tubular neighborhood of $C_Y$, the fibration $\pi_Y$ is a submersion (otherwise, $C_Y$ would be a multiple fiber, and $C_X$ would have multiplicity $>1$ in $C_{X_{12}} +M_{12}$). The curve $M_{12}$ is contracted to a point $r_Y\in Y$, and $M_{12}$ is the exceptional divisor of the blow-up of $Y$ at $r_Y$.
Thus, the geometry of $\pi_{12}\colon X_{12}\to \P^1$ near $M_{12}$ is the geometry of a smooth foliation after a blow-up, and:
\begin{enumerate}
\item $\pi_{12}$ has a Morse singularity at the point $r\in X_{12}$; there are local coordinates such that $\pi_{12}(x,y)=xy$, with the two coordinate axes corresponding to $C_X$ and $M_{12}$, respectively;
\item if $b$ is close to $0$, the fiber $\pi_{12}^{-1}(b)$ is  close to $C_{X_{12}}\cup M_{12}$; it is the pull-back of a smooth fiber of $\pi_Y$ close to $C_Y$; as $b$ approaches $0$,   $\pi_{12}^{-1}(b)$ converges towards $C_{X_{12}}\cup M_{12}$ in the Hausdorff topology, and in the $C^1$ topology in the complement of $r$;
\item in  the real surface, $C_{X_{12}}(\R)$ is a topological circle with a Möbius band as tubular neighborhood, and so is $M_{12}(\R)$;   the smooth fibers of $\pi_{12}$ in $X(\R)$ near $C_{X_{12}}(\R)\cup M_{12}(\R)$ are topological circles,   turning once around  $C_{X_{12}}(\R)$  and  around $M_{12}(\R)$. 
\end{enumerate}

\begin{figure}
\begin{minipage}{17cm}
\includegraphics[width=4.5cm]{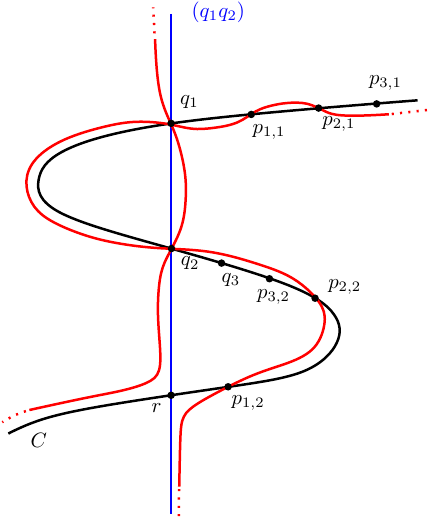}
\includegraphics[width=5.5cm]{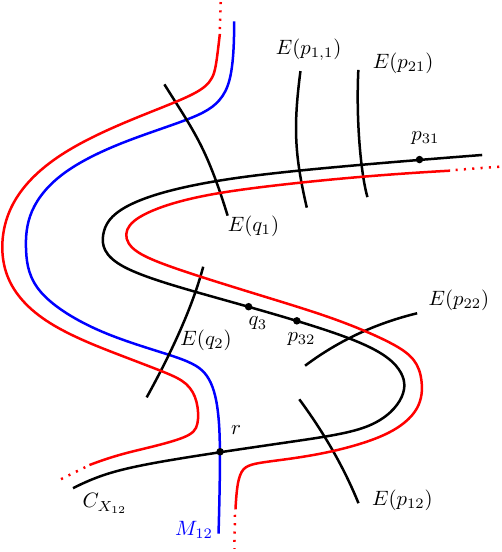}
\includegraphics[width=5.5cm]{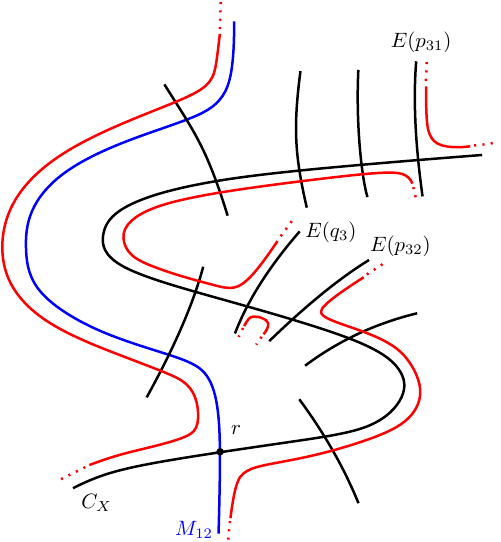}
\end{minipage}
\caption{\begin{small}Schematic view of a real fiber of $\pi_{12}$ (in red) 
close to the singular fiber on $\P^2$ (left),  on $X_{12}$ (middle)  and on $X$ (right). All components 
(including the red fiber, except on the left)
 are topological circles.   \end{small}}
 \label{fig:fibration}
\end{figure}
 
If we blow-up the  remaining points
to construct $X(\R)$, $C_{X_{12}}(\R)\cup M_{12}(\R)$ is replaced by its total transform: we add $6$ topological circles corresponding to the exceptional 
divisors obtained by blowin up  $q_3$, $q_4$, $p_{3, 1},\ldots p_{4,2}$, and 
the picture near each of these circles is similar to the one near $M_{12}$ (see Figure~\ref{fig:fibration} for a visual illustration).

\noindent{\bf{Step 2.--}} Dynamics of $g_{12}$ on smooth fibers. 

If we fix an affine coordinate $z$ on $(q_1q_2)$ such that $r=\infty$, $q_2=0$ and $q_1 =1$, then $\sigma_2(z)  = -z$ and $\sigma_1(z)  = 2-z$. So $\sigma_1\circ \sigma_2 (z) = z+2$, that is, 
on $M_{12}$
 $g_{12}$ acts as a parabolic transformation with fixed point at $r$.

For $b\in \P^1(\R)$ near $0$, denote by $Y_b(\R)$ the real part of $\pi_{Y}^{-1}(b)$; thus, $Y_0(\R) $ is the smooth fiber $C_Y(\R)$.
The complex curve $Y_b(\C)$ is  smooth of genus $1$; it is a quotient $\C/L_b$, where $L_b=\Z+\Z\tau(b)$ is a lattice in $\C$ 
and $Y_b(\R)$ corresponds to  $\R/\Z$ (because $C_Y(\R)$ and the fibers near it are connected). 
Thus, in a neighborhood $N$ of $0$, we get   a  real analytic map $\tau$
  with values in the upper half plane, and,  
 a map $\varphi_b\colon t\in \R/\Z\to \varphi_b(t)\in Y(\R)$ parametrizing $Y_b(\R)$,   depending analytically on $(b,t)\in N\times \R/\Z$. Pulling back  $\varphi_b$ by the natural birational map $\varpi_{12}\colon X_{12}\to Y$ (resp.\ $X\to Y$), we get a family of parametrizations $\tilde{\varphi_b}$ of the fibers $X_{12,b}$ (resp. $X_{b}$) of $\pi_{12}$, for $b\in N\setminus\set{0}$. Since $g_{12}$ acts by translation
along the curves $Y_b$ (resp. $X_b$, $b\neq 0$), and $g_{12}$ is the identity on $C_X$, we see that $\varphi_b$ conjugates $g_{12\vert Y_b}$ to 
a rotation of the circle with angle $\alpha(b)$ converging towards $0$ as $b$ goes to $0$.

\noindent{\bf{Step 3.--}} Conclusion. 

Pick $x\in X(\R)\setminus C_X(\R)$ such that $\Gamma\cdot x$ accumulates $C_X(\R)$; 
fix a sequence $(x_j)$ of distinct points of  $\Gamma \cdot x$ converging  to $C_X(\R)$, and set $b_j=\pi_{12}(x_j)$. 
Step~1 shows that $X_{b_j}(\R)$ converges in the Hausdorff topology to $C_X\cup M_{12}\cup E_{q_3}\cup  \cdots \cup E_{p_{4,2}}$; this convergence holds 
 in the $C^1$ topology away from the 
 singular points. Take an arbitrary point $y_0\in M_{12}(\R)\setminus \set{r}$ and 
consider  the segment $[y_0, g_{12}(y_0)]$ in 
$M_{12}(\R)$. By the description of $g_{12}$ in the first paragraph of
 Step~2, $[y_0, g_{12}(y_0)]$ is disjoint from $r$. If $I$ is a small interval containing $y_0$ and transverse to the fibration $\pi_{12}$, then the intervals $I$ and $J:=g_{12}(I)$ intersect each fiber $X_b$ near $M_{12}$ transversely into two points  $y_b$ and 
$g_{12}(y_b)$; as $b$ goes to $0$, the segment $[y_b, g_{12}(y_b)]\subset X_b(\R)$ converges towards the segment $[y_0, g_{12}(y_0)]\subset M_{12}$.

The following lemma is elementary and left to the reader. 

\begin{lem}
Let $R: \R/\Z \to \R/\Z$ be a rotation of angle  $\alpha\neq 0$.  Fix $t\in \R/\Z$ and let $K$ be the shortest closed segment joining $t$ to $R(t)$. Then for any $s\in \R/\Z$, the $R$-orbit of $s$ intersects~$K$. 
\end{lem}

This  lemma shows that the orbit of $x_j$ under $g_{12}$  must
 intersect $[y_{b_j}, g_{12}(y_{b_j})]$, hence so does $\Gamma\cdot x$. 
 Taking $j\to \infty$,  this  implies that $\Gamma\cdot x$ accumulates $[y_0, g_{12}(y_0)]$, which 
 is contained in $M_{12}(\R)\setminus C_X(\R)$. As explained before Step~1, this completes the  proof of the theorem. 
\end{proof}

 \subsection{Conclusion of the proof of Theorem~\ref{thm:blanc_examples}.}  We pick $C$ and the $q_i$ as in Theorem~\ref{thm:blanc_stiffness_real}. 
 Assertion~(1) follows from the smoothness of $C$ and the genus formula, and Assertion~(2) is a theorem of Blanc.
 Assertions~(3) and~(4) on $\Omega_X$ are described in Section~\ref{part:ergodic_blanc}.
 Property~(5) follows from Proposition~\ref{pro:no_finite_orbit_strong_version}.
 Assertion~(6) on the generic density of orbits is an elementary 
  consequence of Proposition~\ref{prop:ergodicity_volume}: indeed for any non-empty open set $U$, the set of $x\in X$ such that $(\Gamma\cdot x)\cap U = \emptyset$ has zero Lebesgue measure.  
The statement~(7) on orbit closures is Theorem~\ref{thm:blanc_orbit_closures_real}. The stiffness property~(8) is established in Theorem~\ref{thm:blanc_stiffness_real}, and finally,  
 Assertion~(9) is Theorem~\ref{thm:lyapunov_real2}.  
\qed
 
We conclude the  paper with an open question.  
By Breiman's ergodic theorem,
 for every $x\in X(\R)\setminus C(\R)$ and $\nu^N$-almost every $\omega$, 
 any cluster value of the sequence of 
 empirical measures $\unsur{n}\sum_{k=1}^n \delta_{f_\omega^k(x)}$ 
 is a probability measure on $C$, a priori depending on $x$, $\omega$ 
 and a choice of subsequence.  The question is about the complexity of the set of limiting measures:
 
\begin{que}
Which probability measures do arise in this way? 
Do the sequences of empirical measures typically converge or, on the contrary,  does  ``historic behavior'' occur? 
\end{que}

\newpage

 \bibliographystyle{plain}
\bibliography{biblio-serge}
 \end{document}